\documentclass[11pt,english]{article}
\usepackage{amsmath}
\usepackage{amsthm}
\usepackage{mathptmx}
\usepackage{helvet}
\usepackage{courier}
\usepackage{newtxmath}
\usepackage[T1]{fontenc}
\usepackage[utf8]{inputenc}
\usepackage{geometry}
\usepackage{color}
\usepackage{babel}
\usepackage{verbatim}
\usepackage{refstyle}
\usepackage{mathrsfs}
\usepackage{mathtools}
\usepackage{graphicx}
\usepackage[unicode=true,pdfusetitle,
 bookmarks=true,bookmarksnumbered=false,bookmarksopen=false,
 breaklinks=false,pdfborder={0 0 0},pdfborderstyle={},backref=false,colorlinks=true]
 {hyperref}

\makeatletter

\AtBeginDocument{\providecommand\secref[1]{\ref{sec:#1}}}
\AtBeginDocument{\providecommand\thmref[1]{\ref{thm:#1}}}
\AtBeginDocument{\providecommand\propref[1]{\ref{prop:#1}}}
\AtBeginDocument{\providecommand\figref[1]{\ref{fig:#1}}}
\AtBeginDocument{\providecommand\lemref[1]{\ref{lem:#1}}}
\AtBeginDocument{\providecommand\subsecref[1]{\ref{subsec:#1}}}
\AtBeginDocument{\providecommand\eqref[1]{\ref{eq:#1}}}
\AtBeginDocument{\providecommand\corref[1]{\ref{cor:#1}}}
\AtBeginDocument{\providecommand\defref[1]{\ref{def:#1}}}
\AtBeginDocument{\providecommand\remref[1]{\ref{rem:#1}}}
\AtBeginDocument{\providecommand\enuref[1]{\ref{enu:#1}}}
\RS@ifundefined{subsecref}
  {\newref{subsec}{name = \RSsectxt}}
  {}
\RS@ifundefined{thmref}
  {\def\RSthmtxt{theorem~}\newref{thm}{name = \RSthmtxt}}
  {}
\RS@ifundefined{lemref}
  {\def\RSlemtxt{lemma~}\newref{lem}{name = \RSlemtxt}}
  {}

\numberwithin{equation}{section}
\numberwithin{figure}{section}
\numberwithin{table}{section}
\theoremstyle{plain}
\newtheorem{thm}{\protect\theoremname}[section]
\theoremstyle{plain}
\newtheorem{prop}[thm]{\protect\propositionname}
\theoremstyle{plain}
\newtheorem{lem}[thm]{\protect\lemmaname}
\theoremstyle{plain}
\newtheorem{cor}[thm]{\protect\corollaryname}
\theoremstyle{definition}
\newtheorem{defn}[thm]{\protect\definitionname}
\theoremstyle{remark}
\newtheorem{rem}[thm]{\protect\remarkname}

\usepackage{refstyle}
\newref{rem}{name=Remark~,names=Remarks~}
\newref{lem}{name=Lemma~,names=Lemmas~}
\newref{def}{name=Definition~,names=Definitions~}
\newref{thm}{name=Theorem~,names=Theorems~}
\newref{prop}{name=Proposition~,names=Propositions~}
\newref{cond}{name=Condition~,names=Conditions~}
\newref{cor}{name=Corollary~,names=Corollaries~}
\newref{sec}{name=Section~,names=Sections~}
\newref{fig}{name=Figure~,names=Figures~}
\newref{subsec}{name=Subsection~,names=Subsections~}
\newref{eq}{name=,refcmd=(\ref{#1})}
\usepackage{tikz}
\usetikzlibrary{calc}
\usetikzlibrary{decorations.pathmorphing}
\usetikzlibrary{intersections}
\usetikzlibrary{through}
\hypersetup{final}

\@ifundefined{showcaptionsetup}{}{%
 \PassOptionsToPackage{caption=false}{subfig}}
\usepackage{subfig}
\makeatother

\providecommand{\corollaryname}{Corollary}
\providecommand{\definitionname}{Definition}
\providecommand{\lemmaname}{Lemma}
\providecommand{\propositionname}{Proposition}
\providecommand{\remarkname}{Remark}
\providecommand{\theoremname}{Theorem}

\begin{document}
\global\long\def\dist{\operatorname{dist}}%

\global\long\def\diam{\operatorname{diam}}%

\global\long\def\Hess{\operatorname{Hess}}%

\global\long\def\Law{\operatorname{Law}}%

\global\long\def\supp{\operatorname{supp}}%

\global\long\def\spn{\operatorname{span}}%

\global\long\def\tr{\operatorname{tr}}%

\global\long\def\Re{\operatorname{Re}}%

\global\long\def\Leb{\operatorname{Leb}}%

\global\long\def\interior{\operatorname{interior}}%

\global\long\def\Im{\operatorname{Im}}%

\global\long\def\dif{\mathrm{d}}%

\global\long\def\e{\mathrm{e}}%

\global\long\def\i{\mathrm{i}}%

\global\long\def\Dif{\mathrm{D}}%

\global\long\def\eps{\varepsilon}%

\global\long\def\Cov{\operatorname{Cov}}%

\global\long\def\Var{\operatorname{Var}}%

\global\long\def\sgn{\operatorname{sgn}}%

\global\long\def\width{\operatorname{width}}%

\global\long\def\height{\operatorname{height}}%

\global\long\def\AR{\operatorname{AR}}%

\global\long\def\Euc{\mathrm{E}}%

\global\long\def\Bernoulli{\operatorname{Bernoulli}}%

\title{Subsequential scaling limits for Liouville graph distance}
\author{Jian Ding\thanks{Partially supported by National Science Foundation     (NSF) grant DMS-1757479 and an Alfred Sloan fellowship.}\\
University of Pennsylvania\and Alexander Dunlap\thanks{Partially supported by a National Science Foundation Graduate Research Fellowship under grant DGE-1147470.}\\
Stanford University}
\maketitle
\begin{abstract}
For $0<\gamma<2$ and $\delta>0$, we consider the Liouville graph
distance, which is the minimal number of Euclidean balls of $\gamma$-Liouville
quantum gravity measure at most $\delta$ whose union contains a continuous
path between two endpoints. In this paper, we show that the renormalized
distance is tight and thus has subsequential scaling limits at $\delta\to0$.
In particular, we show that for all $\delta>0$ the diameter with
respect to the Liouville graph distance has the same order as the
typical distance between two endpoints.
\end{abstract}

\section{Introduction}

Let $\mathbb{R}$ be a rectangular subset of $\mathbf{R}^{2}$ (the
Euclidean plane), and let $\mathbb{R}^{*}$ be a box centered around
$\mathbb{R}$ whose sides are separated from those of $\mathbb{R}$
by twice the Euclidean diameter of $\mathbb{R}$. Let $h_{\mathbb{R}^{*}}$
be a Gaussian free field with Dirichlet boundary conditions on $\mathbb{R}^{*}$
and let $\mu_{\mathbb{R}}$ be the Liouville quantum gravity (LQG)
measure on $\mathbb{R}$, at inverse temperature $\gamma\in(0,2)$,
induced by $h_{\mathbb{R}^{*}}$. (We precisely define these objects
in \secref{Preliminaries}; see also the surveys \cite{She07,B15,RV14}.)
In this paper, we consider the \emph{Liouville graph distance} with
parameter $\delta$ on $\mathbb{R}$, which for any two points $x,y\in\mathbb{R}$
is given by $\underline{d}_{\mathbb{R},\delta}(x,y)$, the minimal
number of Euclidean balls of LQG measure at most $\delta$ that it
takes to cover a path between $x$ and $y$. (Again, see \secref{Preliminaries}
for a precise definition.) Let $Q_{\delta}$ be the median Liouville
graph distance from the left to the right side of $\mathbb{R}$. Our
goal in this paper is to prove the following:
\begin{thm}
\label{thm:maintheorem}For any sequence $\delta_{n}\downarrow0$,
there is a subsequence $(\delta_{n_{k}})$ and a limiting metric $\underline{d}_{\mathbb{R}}$
so that
\[
Q_{\delta_{n_{k}}}^{-1}\underline{d}_{\mathbb{R},\delta_{n_{k}}}\to\underline{d}_{\mathbb{R}}
\]
in distribution with respect to the uniform topology on functions
$\mathbb{R}\times\mathbb{R}\to\mathbf{R}$. Moreover, $\underline{d}_{\mathbb{R}}$
is Hölder-continuous with respect to the Euclidean metric on $\mathbb{R}$,
and the Euclidean metric on $\mathbb{R}$ is Hölder-continuous with
respect to $\underline{d}_{\mathbb{R}}$. 
\end{thm}

A consequence of \thmref{maintheorem}, but also a key step in its
proof (see \propref{chaining} below), is that for any $\delta$,
the LQG diameter of a box is comparable to its left--right crossing
distance:
\begin{thm}
\label{thm:diam}For any $\eps>0$, we have a $C=C(\eps)$ so that,
for all $\delta>0$,
\[
\mathbf{P}\left(\max_{x,y\in\mathbb{R}}\underline{d}_{\mathbb{R},\delta}(x,y)\ge C(\eps)Q_{\delta}\right)\le\eps.
\]
\end{thm}

The normalizing constant $Q_{\delta}$ is poorly-understood, and the
arguments in this paper do not rely on precise estimates of its value.
However, we will show the following in the course of the proof of
\thmref{maintheorem}:
\begin{thm}
\label{thm:plbound}We have a constant $0<C<\infty$ so that for all
$\delta$ we have
\[
C^{-1}\delta^{-C^{-1}}\le Q_{\delta}\le C\delta^{-C}.
\]
\end{thm}

\subsection{Background and motivation}

In the seminal paper \cite{Kahane85}, the Gaussian multiplicative
chaos was defined and constructed as a random measure obtained by
exponentiating log-correlated Gaussian fields. The last decade has
seen extensive study of Gaussian multiplicative chaos as well as Liouville
quantum gravity,\footnote{Our terminology for Liouville quantum gravity follows that in \cite{DS11}.
This is somewhat different from that adopted in Liouville field theory,
and one should be cautious about the underlying mathematical meaning
of LQG in Liouville field theory.} which is an important special case of Gaussian multiplicative chaos
where the underlying log-correlated field is a two-dimensional Gaussian
free field. See for example \cite{Kahane85,DS11,RV11,RV14,Berestycki17,RV17,APS18}.

Recently, a huge amount of effort has been devoted to understanding
the random metric associated with LQG. Building on \cite{MS16,DMS14},
in \cite{MS15b,MS16a,MS16b} the authors constructed in the continuum
a random metric which is \emph{presumably} the scaling limit of the
LQG distance for the specific choice $\gamma=\sqrt{8/3}$, and proved
deep connections with the Brownian map\cite{LeGall10,LeGall13,Miermont13}.
Note that in these works there is no mention of a discrete approximation
of the metric. In \cite{DZ16,DG16,GHS16,GHS17,DG18}, various bounds
on the distance and properties of the geodesics were obtained. In
\cite{DZ15,DZZ17}, some non-universality aspects (when considering
underlying log-correlated fields other than GFF) for LQG distances
were demonstrated. In \cite{MRVZ14,DZZ18}, a type of equivalence
between the Liouville graph distance and the heat kernel for Liouville
Brownian motion was proved. In \cite{DG18}, it was shown that there
is a single number which determines the distance exponents for a few
reasonable choices of distances associated with LQG, as well as the
distance exponents for random planar maps.

However, despite much dedicated effort, the two most outstanding problems
related to LQG distances remain open: (1) to compute the exact distance
exponents for LQG distances (although this is now known for $\gamma=\sqrt{8/3}$
by \cite{Angel03,DG18}), and (2) to derive a scaling limit of natural
discrete approximations of LQG distances. The present paper makes
some progress towards understanding the scaling limits of LQG distances.

\subsection{Two closely-related works}

The present article is closely related to \cite{DD16,DF18}. In \cite{DD16}
it was shown that discrete Liouville first-passage percolation (shortest-path
metric where the vertices are weighted by the exponential of the discrete
GFF) has subsequential scaling limits for sufficiently small $\gamma>0$.
On the other hand, in \cite{DF18}, the authors considered the case
when the underlying field is a type of log-correlated Gaussian field
(in the continuum) with short-range correlations (a so-called $\star$-scale
invariant field) and showed that there exists a parameter $\gamma^{*}>0$
such that the corresponding Liouville first-passage percolation has
a subsequential scaling limit for all $\gamma<\min(\gamma^{*},0.4)$.
The main contribution of the present article is that the result for
Liouville graph distance is valid throughout the subcritical regime;
i.e. $\gamma\in(0,2)$. A few further remarks are in order:
\begin{itemize}
\item In both \cite{DD16} and \cite{DF18}, the authors worked with Liouville
first-passage percolation, while in the present article we work with
Liouville graph distance. They are both natural approximations of
LQG distances and at the moment we are equally satisfied with proving
results for either of these choices (or any other reasonable choices).
As noted earlier, in \cite{DG18}, universality of the dimension exponents
was proved for all reasonable choices of metric we know when we stick
to the GFF as the underlying field (but we should also note that the
dimension exponent is not expected to be universal between GFF and
$\star$-scale invariant field; see \cite{DZ16}. Here we choose to
work with Liouville graph distance because it resembles random planar
maps in a more obvious fashion, and because it is (slightly) more
directly connected to the heat kernel for Liouville Brownian motion
\cite{DZZ18}.
\item In both \cite{DD16,DF18}, in the regime where subsequential scaling
limits were proved, a crucial step of the proofs was that the box
diameter has the same order as the typical distance between two points.
We also note that in \cite{DG18}, it was shown for all $\gamma\in(0,2)$
that the exponents for the diameter and the typical distance are the
same. The present article similarly shows that in our setting, the
diameter of a box has the same order as the typical distance between
two points.
\item In \cite{DF18}, results regarding the Weyl scaling were also obtained.
Weyl scaling is considered neither in \cite{DD16} nor in the present
article.
\item The $\star$-scale invariant fields form an important class of log-correlated
Gaussian fields, and they have been studied previously in \cite{ARV13,Madaule15}
for instance. We note that the Gaussian free field is not a $\star$-scale
invariant field.
\item At the moment, we are equally satisfied with proving results when
the underlying fields are discrete or continuous. But the continuous
case appears to be simpler thanks to conformal invariance-based arguments
introduced in \cite{DF18}.
\end{itemize}
We now discuss similarities and differences in proof methods. In \cite{DD16},
we presented a framework combining multiscale analysis and Efron--Stein-type
arguments with Russo--Seymour--Welsh (RSW)-type estimates (originally
introduced in \cite{Russo78,SW,Russo81}), which relate distances
between boundaries of rectangles in easy and hard directions. (Here
and throughout the paper, the ``easy'' direction across a rectangle
refers to crossings between the two longer sides, while the ``hard''
direction refers to crossings between the two shorter sides; see \figref{easyhard}.)
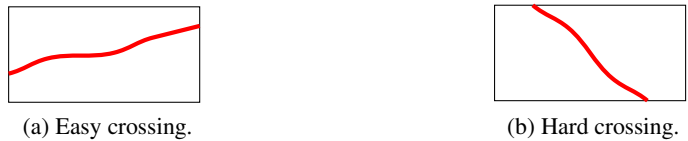
\begin{figure}
\begin{centering}
\hfill{}\subfloat[Easy crossing.]{\begin{centering}
\begin{tikzpicture}[x=0.5in,y=0.5in]
\draw[step=1,thin] (0,0) rectangle (2,1);
\draw[ultra thick, red,style={decorate,decoration={snake,amplitude=2,segment length=50}}] (0,0.3) -- (2,0.8);
\end{tikzpicture}
\par\end{centering}
}\hfill{}\subfloat[Hard crossing.]{\begin{centering}
\begin{tikzpicture}[x=0.5in,y=0.5in]
\draw[step=1,thin] (0,0) rectangle (2,1);
\draw[ultra thick, red,style={decorate,decoration={snake,amplitude=2,segment length=50}}] (1.6,0) -- (0.4,1);
\end{tikzpicture}
\par\end{centering}
}\hfill{}
\par\end{centering}
\caption{\label{fig:easyhard}Easy and hard crossings.}
\end{figure}
 This is more or less the framework used both in \cite{DF18} and
in the present article.

The key difference between \cite{DD16} and \cite{DF18} lies in the
implementation of the RSW estimate. In \cite{DD16} the RSW estimate
was inspired by \cite{tassion}, while in \cite{DF18} it was proved
using approximate conformal invariance. The proof in \cite{DD16}
draws a natural connection between random distances and percolation
theory. It also presents a more widely-applicable framework---based
on the method of \cite{tassion}---for proving RSW estimates, as
it does not rely on conformal invariance (and thus works for instance
in the lattice case, or potentially for more general fields). However,
while the proof in \cite{tassion} (for Voronoi percolation) is simply
beautiful, much of the beauty was obscured when it was implemented
in \cite{DD16}, due to the severe complications involved in considering
lengths of paths rather than simply connectivity. By contrast, the
proof method of \cite{DF18}, taking advantage of conformal invariance,
is insightful and beautiful. While the method of \cite{DF18} does
suffer the drawback of relying on conformal invariance and thus may
not be as widely applicable as the method in \cite{DD16}, this drawback
is irrelevant here since at the moment we are satisfied with the case
of a conformally-invariant field. For the present article, while it
seems likely that an RSW proof following \cite{DD16} is possible
(although it would require improving the analysis in \cite{DD16}
in a number of places), we switch to a conformal invariance-based
proof as in \cite{DF18} for the sake of simplicity.

In \cite{DD16}, since $\gamma$ is very small, in the multiscale
analysis the influence from the coarse field (which one can roughly
understand as circle averages with respect to macroscopic circles)
is negligible, which simplifies the analysis in a number of ways.
(The situation is in some sense similar in \cite{DF18} as there it
is only shown that $\gamma^{*}>0$.) Thus, the main new technical
challenge in the present article is controlling the influence from
coarse field for $0<\gamma<2$. (Naturally, it is most difficult when
$\gamma$ is close to $2$.) One manifestation of this challenge is
the fact that LQG measure only has a finite $p$th moment for $p$
approaching $1$ as $\gamma$ approaches $2$. While this creates
challenges in the multiscale analysis, from a conceptual point of
view our proof crucially relies on the fact that the LQG measures
for all Euclidean balls with radius at most $\eps$ uniformly converge
to $0$ for $\gamma\in(0,2)$---this fact is essential crucial for
the diameter to have the same order as the typical distance.

\subsection{Ingredients of the proof}

Here we describe the essential structure of our proof.

By a chaining argument similar to the ones used in \cite{DD16,DF18},
we can bound the diameter of a box by a sum of hard crossing distances
of dyadic subboxes at successively smaller scales. (See \propref{chaining}.)
Controlling the fluctuation of box diameters uniformly in the scale
is essentially enough to uniformly control the fluctuations of a Hölder
norm of the metric with respect to the Euclidean metric, which yields
tightness in the uniform norm. Thus, our essential goal is to bound
the fluctuations of hard crossing distances, again uniformly in the
scale. The form of our bound on these fluctuations will be a variance
bound on the logarithm of these fluctuations (\thmref{variancebound}),
which allows us to relate different quantiles of the scale as described
in \lemref{relatequantiles}.

Our bound on the variance of the logarithm of the hard crossing distance
takes place in the framework of the Efron--Stein inequality. We write
the underlying Gaussian free field as a function of the space-time
white noise, partition space-time into chunks, and then express the
variance of the logarithm of the Liouville graph distance as the sum
of the expected squared \emph{multiplicative} (since we are considering
the logarithm) changes in the distance when each box is resampled.
The white noise decomposition is by now a widely-used way to decompose
the GFF into a ``fine field,'' with no regularity but only short-range
correlations, and a ``coarse field,'' which is smooth but has long-range
correlations. See for example \cite{RV14,Shamov16} for general descriptions
of the white noise decomposition and \cite{DG16,DZZ18} for previous
uses of this decomposition in the setting of the LQG metric.

We control the fluctuations due to the coarse field using a standard
Gaussian concentration argument in \subsecref{coarsefieldeffect};
this is sufficient because we choose the decomposition so that the
coarse field is sufficiently smooth. The more serious task is to control
the fluctuations due to the fine field. Of course, the most pronounced
effect of the fine field in a certain box is on the part of the geodesic
close to the same box. Controlling the fluctuations on this part of
the geodesic (done in \subsecref{closefinefield}) essentially has
two key components:
\begin{itemize}
\item First, we observe that since the coarse field is smooth, and thus
can be treated as a constant on the small scale box, it suffices to
control the change in the distance for the LQG metric in boxes at
a smaller scale. This smaller scale, though, is \emph{not} the size
of the smaller boxes. From the perspective of the Liouville graph
distance, when the coarse field is positive, subboxes look larger
than they are, while when the coarse field is negative, subboxes look
smaller than they are. This is to say that the small boxes have an
``effective size'' depending on the value of the coarse field. (This
is quantified by the scaling covariance of the Liouville quantum gravity;
see \propref{conformalcovariance} and \propref{fieldonsmallerbox}.)
The requirement that $\gamma<2$ is precisely what is needed to ensure
that, even considering the maximum effect of the coarse field, the
effective size of the small boxes is strictly smaller than that of
the overall large box with high probability.
\item Second, the fact that the subboxes have a smaller effective size with
high probability means that we can apply an inductive procedure. We
assume that we have already proved our variance bound at smaller scales,
and thus can use this variance bound to bound quantiles of crossing
distances in our small boxes. This works  because the variance bound,
together with the RSW result (\thmref{RSW}) and percolation-type
arguments (given in \secref{Percolation}), allows us to show that
crossing distances are concentrated, and that ``easy'' and ``hard''
crossing distances are multiplicatively comparable. This means that,
using the inductive hypothesis, we can replace the part of the geodesic
close to the resampled box with a geodesic around a surrounding annulus
while only increasing the distance by, typically, a constant factor.
Moreover, since again the effective size of the smaller boxes is smaller
than the size of the large box, the increase in the distance incurred
in this way is also small compared to the total crossing distance
of the large box. (Here, we need to know that typical distances increase
at least polynomially in the scale, which is established in \propref{pllb-inductive}.)
This is what is needed to bound the variance of the logarithm.
\end{itemize}
As we have noted, the improvement of this paper compared to \cite{DD16,DF18}
is that we are able to obtain the tightness result for all $\gamma\in(0,2)$.
The main ingredient for this, not used in \cite{DD16,DF18}, is the
scaling covariance of the Liouville quantum gravity described above.
This allows us to perform a more precise multiscale analysis than
is done in either of these preceding works. In particular, this works
because we are able to show that precisely when $\gamma<2$, subboxes
of a larger box have a a strictly smaller ``effective size'' than
the larger box with high probability.

\subsection{Organization of the paper}

In \secref{Preliminaries} we introduce our cast of characters, collecting
a number of (semi-)standard facts about Gaussian free fields, Liouville
quantum gravity, and Liouville graph distance. In \secref{RSW} we
prove a conformal covariance-based RSW estimate for Liouville quantum
gravity, in the spirit of \cite{DF18}. In \secref{Percolation} we
prove concentration bounds on crossing distances using percolation
arguments in a multi-scale analysis framework, and set up some quantities
regarding the relationship between different quantiles of crossing
distances, which will form the objects of our inductive procedure.
(We defer the introduction of these quantities until that point because
they depend on certain constants which are introduced in the lemmas
of \secref{RSW} and \secref{Percolation}.) In \secref{inductiveargument},
we carry out the Efron--Stein argument, and in \secref{chaining}
we show that the diameter is within a constant of the typical distance
between two points, via a chaining argument, on our way to proving
the tightness of the metrics which is the main ingredient in \thmref[s]{maintheorem},~\ref{thm:diam},~and~\ref{thm:plbound}.
In \secref{invholder}, we prove that the the Euclidean metric is
Hölder-continuous with respect to any limiting metric, completing
the proof of \thmref{maintheorem}. Finally, we have relegated several
technical lemmas in analysis and probability, that do not relate in
particular to the subjects of this paper, to an appendix (beginning
on page~\pageref{sec:appendix-techlemmas}).

\subsection{Acknowledgments}

We thank Ewain Gwynne, Subhajit Goswami, Nina Holden, Jason Miller,
Dat Nguyen, Josh Pfeffer, and Xin Sun for interesting conversations.
We especially thank Hugo Falconet for pointing out an error in an
earlier version of the manuscript. We are also grateful to an anonymous
referee for an extremely careful reading and many insightful comments,
which greatly helped to improve the exposition.

\section{Preliminaries\label{sec:Preliminaries}}

\subsection{Notation}

We denote by $\mathbf{R}$ the set of real numbers and by $\mathbf{Q}$
the set of rational numbers. If $x,y\in\mathbf{R}$, we will use the
notation $x\wedge y=\min\{x,y\}$, $x\vee y=\max\{x,y\}$, $x^{+}=x\vee0$,
and $x^{-}=x\wedge0$. If $x\in\mathbf{R}^{2}$, we will denote by
$|x|$ the Euclidean norm of $x$. We denote by $B(x,r)$ the Euclidean
open ball with center $x$ and radius $r$. We will often use blackboard-bold
letters to refer to subsets of $\mathbf{R}^{2}$. If $\mathbb{X}\subset\mathbf{R}^{2}$,
then let $\diam_{\Euc}(\mathbb{X})$ denote the Euclidean diameter
of $\mathbb{X}$, and let $\overline{\mathbb{X}}$ denote the closure
of $\mathbb{X}$ in the Euclidean topology.

By a \emph{domain }we mean an open connected subset of $\mathbf{R}^{2}$.
By a \emph{box} we will mean a closed rectangular subset of $\mathbf{R}^{2}$
whose axes are aligned with the coordinate axes. If $\mathbb{B}$
is a box, let $\width(\mathbb{B})$ and $\height(\mathbb{B})$ denote
the width and height of $\mathbb{B}$, respectively. We define the
\emph{aspect ratio} of $\mathbb{B}$ as
\[
\AR(\mathbb{B})=\frac{\width(\mathbb{B})}{\height(\mathbb{B})}.
\]
For a box $\mathbb{B}$ and a scalar $\lambda>0$, define $\lambda\mathbb{B}$
to be the box with the same center as $\mathbb{B}$ but $\lambda$
times the side length.

We will frequently use the notation
\[
\mathbb{B}(S_{1},S_{2})=[-S_{1}/2,S_{1}/2]\times[-S_{2}/2,S_{2}/2]
\]
and
\[
\mathbb{B}(S)=\mathbb{B}(S,2S).
\]
Given a box $\mathbb{B}$, let $\mathrm{L}_{\mathbb{B}}$, $\mathrm{R}_{\mathbb{B}}$,
$\mathrm{T}_{\mathbb{B}}$, and $\mathrm{B}_{\mathbb{B}}$ denote
the left, right, top, and bottom sides of $\mathbb{B}$, respectively.
We will often omit the subscript $\mathbb{B}$ if it is clear from
context.

If $\mathbb{B}\subset\mathbf{R}^{2}$, we will use the notation
\[
\mathbb{B}^{(R)}=\{x\in\mathbf{R}^{2}\mid\dist_{\Euc}(x,\mathbb{B})<R\}=\bigcup_{x\in\mathbb{B}}B(x,R).
\]
If $\mathbb{B}\subset\mathbf{R}^{2}$ is a box, define $\mathbb{B}^{\circ}$
to be the smallest box containing $\mathbb{B}^{(\diam_{\Euc}\mathbb{B})}$,
and $\mathbb{B}^{*}$ to be the smallest box containing $\mathbb{B}^{(2\diam_{\Euc}\mathbb{B})}$.

Throughout this paper, we will often work with constants, usually
denoted by $C$ or similar notation if they are to be thought of as
large, or by $c$ or similar notation if they are to be thought of
as small. We will always allow constants to change from line to line
in a computation.

\subsection{The Gaussian free field\label{subsec:GFF}}

If $\mathbb{A}\subset\mathbf{R}^{2}$ is a domain, we will denote
by $h_{\mathbb{A}}$ a Gaussian free field with Dirichlet boundary
conditions on $\mathbb{A}$. (See \cite{She07,B15} for more systematic
introductions to the GFF.) We recall that the GFF is a distribution-valued
stochastic process with covariance function
\begin{equation}
G^{\mathbb{A}}(x,y)=\pi\int_{0}^{\infty}p_{t}^{\mathbb{A}}(x,y)\,\dif t,\label{eq:covariancekernel}
\end{equation}
where $G^{\mathbb{A}}$ is the Green's function for the Dirichlet
problem on $\mathbb{A}$, and $p_{t}^{\mathbb{A}}$ is the heat kernel
for a standard Brownian motion killed on $\partial\mathbb{A}$. If
$\mathbb{B}\supset\mathbb{A}$, then we define $h_{\mathbb{B}:\mathbb{A}}$
to be the harmonic interpolation of $h_{\mathbb{B}}$ on $\mathbb{A}$,
and simply $h_{\mathbb{B}}$ outside of $\mathbb{A}$. (It is not
quite obvious that this harmonic interpolation makes sense, because
$h_{\mathbb{B}}$ is a distribution and not a function. See \cite{She07}
or \cite{B15} for the precise construction.) We then have the following
standard property. (Again see \cite{She07,B15}.)
\begin{prop}[Gibbs--Markov property]
\label{prop:markovfieldproperty}The field
\begin{equation}
h_{\mathbb{B}}=h_{\mathbb{A}}-h_{\mathbb{A}:\mathbb{B}}\label{eq:markovfieldproperty}
\end{equation}
is a standard Gaussian free field on $\mathbb{B}$. Moreover, the
fields $h_{\mathbb{B}}$ and $h_{\mathbb{A}:\mathbb{B}}$ are independent.
\end{prop}

This paper will rely in crucial ways on multiscale analysis, which
means that we will want to consider Gaussian free fields defined on
all boxes simultaneously---in other words, we want to couple the
Gaussian free fields on all boxes. We can do this by considering a
whole-plane GFF $h_{\mathbf{R}^{2}}$ (see \cite[Section 4.1]{B15}),
pinned in some arbitrary way. Then we can define $h_{\mathbb{B}}=h_{\mathbf{R}^{2}}-h_{\mathbf{R}^{2}:\mathbb{B}}$
for each domain $\mathbb{B}$. Alternatively, we can use the construction
in \cite[Section 2.2.1]{DD16}: the coupling between GFFs on different
domains is easy to define if all domains $\mathbb{B}$ are subsets
of some universal box $\mathbb{A}$, because we can define $h_{\mathbb{B}}$
for all $\mathbb{B}\subset\mathbb{A}$ by \eqref{markovfieldproperty}.
But then the universal box $\mathbb{A}$ can be taken to be arbitrarily
large, and it is easy to check that the joint laws of the Gaussian
free field on finitely many domains remain the same no matter how
large we take $\mathbb{A}$. Thus, by Kolmogorov's extension theorem,
we can take $h_{\mathbb{B}}$ to be defined simultaneously on all
domains $\mathbb{B}\subset\mathbf{R}^{2}$ so that \eqref{markovfieldproperty}
holds for any nested pair of domains.

The following locality property of the Gaussian free field (which
can actually be seen as a consequence of the Gibbs--Markov property)
will also be essential.
\begin{prop}
\label{prop:independent-disjoint}If $\mathbb{A}$ and $\mathbb{A}'$
are disjoint domains, then $h_{\mathbb{A}}$ and $h_{\mathbb{A}'}$
are independent.
\end{prop}

A final crucial property of the Gaussian free field is its conformal
invariance. (See \cite[Theorem 1.19]{B15}.)
\begin{prop}
Suppose that $F:\mathbb{V}\to\mathbb{V}'$ is a conformal map between
two domains $\mathbb{V}$ and $\mathbb{V}'$. Then $h_{\mathbb{V}^{'}}$
has the same law as $h_{\mathbb{V}^{'}}\circ F^{-1}$.
\end{prop}

\subsection{Gaussian free field estimates}

In this section we introduce certain estimates on the Gaussian free
field that will be important throughout the paper.
\begin{lem}
\label{lem:greensfunctionbound}There is a constant $C<\infty$ so
that, if $\mathbb{B}$ is a box with $\width(\mathbb{B}),\height(\mathbb{B})\in[1/3,3]$,
and $x,y\in\mathbb{B}$, then
\begin{equation}
\left|G^{\mathbb{B}^{*}}(x,y)+\log|x-y|\right|\le C.\label{eq:greensfunctionest}
\end{equation}
\end{lem}

\begin{proof}
This can be computed from the explicit formula for the Green's function
of the Laplacian on the unit disk, using a uniform bound on the restriction
of the derivative to $\mathbb{B}$ of a Riemann map from $\mathbb{B}^{*}$
to the unit disk.
\end{proof}
\begin{lem}
\label{lem:varislog}There is a constant $C<\infty$ so that the following
holds. Let $\mathbb{B}\subset\mathbb{A}$ be two nested boxes such
that $\AR(\mathbb{A}),\AR(\mathbb{B})\in[1/3,3]$ and let $x$ be
the center point of $\mathbb{B}$. Then we have
\begin{equation}
\Var(h_{\mathbb{A}^{*}:\mathbb{B}^{*}}(x))\le-\log\frac{\diam_{\Euc}(\mathbb{B})}{\diam_{\Euc}(\mathbb{A})}+C.\label{eq:varislog}
\end{equation}
\end{lem}

\begin{proof}
This follows simply from writing down the integral expression for
the variance of $h_{\mathbb{A}^{*}:\mathbb{B}^{*}}(x)$ and applying
\eqref{greensfunctionest}.
\end{proof}
\begin{lem}
\label{lem:maxvsmin}There is a constant $C<\infty$ so that if $\mathbb{B}\subset\mathbb{B}^{*}\subset\mathbb{A}$
are boxes with aspect ratios between $1/3$ and $3$, then for all
$x,y\in\mathbb{B}^{\circ}$, we have
\[
\Var(h_{\mathbb{A}:\mathbb{B}^{*}}(x)-h_{\mathbb{A}:\mathbb{B}^{*}}(y))\le C\frac{|x-y|}{\diam_{\Euc}(\mathbb{B})}.
\]
\end{lem}

The proof of \lemref{maxvsmin} is given in the case when $\mathbb{B}$
is a ball in \cite[Lemma 6.1]{DG16}; the proof is essentially the
same in our setting.

\begin{comment}
\begin{cor}
\label{cor:maxfluct}There is a constant $C<\infty$ so that the following
holds. If $\mathbb{A},\mathbb{B}_{1},\ldots,\mathbb{B}_{J}$ is a
set of boxes with aspect ratios between $1/3$ and $3$, and $\mathbb{B}_{j}^{*}\subset\mathbb{A}$
for each $j$, then, for all $\theta>0$, we have
\[
\mathbf{P}\left(\max_{1\le j\le J}\max_{x,y\in\mathbb{B}_{j}^{\circ}}\left|h_{\mathbb{A}:\mathbb{B}_{j}^{*}}(x)-h_{\mathbb{A}:\mathbb{B}_{j}^{*}}(y)\right|\ge\theta\sqrt{\log(J+1)}\right)\le C(J+1)^{1-\theta^{2}/C}.
\]
\end{cor}

\begin{proof}
This follows from union-bounding the result of \lemref{fluctstailbound}
over $j\in\{1,\ldots,J\}$.
\end{proof}
\end{comment}

\begin{prop}
\label{prop:maxcoarse}For each $Q,\iota>0$, there is a constant
$C=C(Q,\iota)<\infty$ so that the following holds. Let $\mathbb{B}$
be a box so that $\AR(\mathbb{B})\in[1/3,3]$. Let $K\in(1,\infty)$
and let $\mathbb{C}_{1},\ldots,\mathbb{C}_{J}$, with $J\le QK^{2}$,
be a set of subboxes of $\mathbb{B}$ such that, for each $1\le j\le J$,
we have $\AR(\mathbb{C}_{j})\in[1/3,3]$ and
\[
\frac{\diam_{\Euc}(\mathbb{B})}{\diam_{\Euc}(\mathbb{C}_{j})}\in[K/3,3K].
\]
Then for all $\theta>0$, for each $j\in\{1,\ldots,J\}$ we have,
\begin{equation}
\mathbf{P}\left(\max_{x\in\mathbb{C}_{j}^{\circ}}h_{\mathbb{B}^{*}:\mathbb{C}_{j}^{*}}(x)\ge\theta\log K\right)\le CK^{-\theta^{2}/2}\label{eq:maxcoarse-single}
\end{equation}
and moreover%
\begin{comment}
\begin{equation}
\mathbf{P}\left(\max_{j=1}^{J}\max_{x\in\mathbb{C}_{j}^{\circ}}h_{\mathbb{B}^{*}:\mathbb{C}_{j}^{*}}(x)\ge\theta\log(K+1)+\tilde{\theta}\sqrt{\log(J+1)}\right)\le C(JK^{-\theta^{2}/2}+J^{1-\tilde{\theta}^{2}/C}).\label{eq:maxcoarse}
\end{equation}
In particular, if there is a constant $Q$ so that $J\le QK^{2}$,
and $\iota>0$, then we have a constant $C=C(Q,\iota)<\infty$, depending
only on $Q$ and $\iota$, so that
\end{comment}
\begin{equation}
\mathbf{P}\left(\max_{j=1}^{J}\max_{x\in\mathbb{C}_{j}^{\circ}}h_{\mathbb{B}^{*}:\mathbb{C}_{j}^{*}}(x)\ge\theta\log K\right)\le CK^{2-\theta^{2}/2}.\label{eq:maxcoarse-easiertouse}
\end{equation}
\end{prop}

\begin{proof}
First we note that by Fernique's inequality (\cite{Fer75}, see e.g.
\cite[Theorem 4.1]{A90} and \cite[equation (7.4), Theorem 7.1]{L01})
and \lemref{maxvsmin}, we have a constant $C$ independent of $j$
so that
\begin{equation}
\mathbf{E}\left(\max_{x\in\mathbb{C}_{j}^{\circ}}h_{\mathbb{B}^{*}:\mathbb{C}_{j}^{*}}(x)\right)\le C.\label{eq:expbd}
\end{equation}
Also, by the Borell--TIS inequality (see e.g. \cite[Theorem 7.1]{L01},
\cite[Theorem 6.1]{biskup}, or \cite[Theorem 2.1]{A90}) and \lemref{varislog},
we have constant $C$, again independent of $j$, so that
\begin{equation}
\mathbf{P}\left(\left|\max_{x\in\mathbb{C}_{j}^{\circ}}h_{\mathbb{B}^{*}:\mathbb{C}_{j}^{*}}(x)-\mathbf{E}\max_{x\in\mathbb{C}_{j}^{\circ}}h_{\mathbb{B}^{*}:\mathbb{C}_{j}^{*}}(x)\right|\ge\theta\log K\right)\le\exp\left\{ -\frac{1}{2}\frac{\theta^{2}(\log K)^{2}}{\log K+C}\right\} \le CK^{-\theta^{2}/2},\label{eq:gfffluctsbd}
\end{equation}
where the constant $C$ has changed from the second to the third expression.
\begin{comment}
Let $x_{j}$ be the center of the box $\mathbb{C}_{j}$. We have by
\eqref{varislog} a constant $C$ so that 
\[
\max_{j=1}^{J}\Var h_{\mathbb{B}^{*}:\mathbb{C}_{j}^{*}}(x)\le\log K+C.
\]
Then we see that
\begin{equation}
\mathbf{P}(h_{\mathbb{B}^{*}:\mathbb{C}_{j}^{*}}(x_{j})>\theta\log K)\le\exp\left\{ -\frac{1}{2}\frac{\theta^{2}(\log K)^{2}}{\log K+C}\right\} \le CK^{-\theta^{2}/2}\label{eq:xjbd}
\end{equation}
for some (different) constant $C$. By \lemref{fluctstailbound},
we have that
\begin{equation}
\mathbf{P}\left(\max_{x\in\mathbb{C}_{j}^{\circ}}\left|h_{\mathbb{B}^{*}:\mathbb{C}_{j}^{*}}(x_{j})-h_{\mathbb{B}^{*}:\mathbb{C}_{j}^{*}}(x)\right|>\iota\log K\right)\le C\e^{-\iota^{2}(\log K)^{2}/C}.\label{eq:maxflucts-single}
\end{equation}
\end{comment}
Combining \eqref{expbd} and \eqref{gfffluctsbd} yields \eqref{maxcoarse-single}.
Then taking a union bound over $j\in\{1,\ldots,J\}$ yields \eqref{maxcoarse-easiertouse}.%
\begin{comment}
On the other hand, union-bounding \eqref{xjbd} yields
\begin{equation}
\mathbf{P}\left(\max_{j=1}^{J}h_{\mathbb{B}^{*}:\mathbb{C}_{j}^{*}}(x_{j})>\theta\log K\right)\le CJK^{-\theta^{2}/2},\label{eq:maxcenters}
\end{equation}
and by \corref{maxfluct}, we also have that
\begin{equation}
\mathbf{P}\left(\max_{j=1}^{J}\max_{x\in\mathbb{C}_{j}^{\circ}}\left|h_{\mathbb{B}^{*}:\mathbb{C}_{j}^{*}}(x_{j})-h_{\mathbb{B}^{*}:\mathbb{C}_{j}^{*}}(x)\right|>\tilde{\theta}\sqrt{\log J}\right)\le CJ^{1-\tilde{\theta}^{2}/2}.\label{eq:maxflucts-1}
\end{equation}
Combining \eqref{maxcenters} and \eqref{maxflucts-1} yields \eqref{maxcoarse}.
\end{comment}
\end{proof}
The following corollary will be used over and over again throughout
the paper. It tells us that, for any $\theta_{0}>2$, if we divide
a box into order $K^{2}$-many smaller boxes of $1/K$ times the side
length, the maximum ``coarse field'' imposed on the small boxes
will be of size at most $\theta_{0}\log K$ with high probability.
\begin{cor}
\label{cor:maxcoarse}If $\theta_{0}>2$ and $Q,K>1$ are such that
$J\le QK^{2}$, then there is a constant $C$ depending on $Q$ and
$\theta_{0}$ so that the following holds. Suppose that $\mathbb{B}$
is a box so that $\AR(\mathbb{B})\in[1/3,3]$. Let $\mathbb{C}_{1},\ldots,\mathbb{C}_{J}$
be a set of subboxes of $\mathbb{B}$ so that $\AR(\mathbb{C}_{j})\in[1/3,3]$
for each $1\le j\le J$ and that
\[
\frac{\diam_{\Euc}(\mathbb{B})}{\diam_{\Euc}(\mathbb{C}_{j})}\in[K/3,3K].
\]
Let
\[
F=\max_{j=1}^{J}\max_{x\in\mathbb{C}_{j}}h_{\mathbb{B}^{*}:\mathbb{C}_{j}^{*}}(x).
\]
Then, for all $\nu>1$, we have
\begin{equation}
\mathbf{P}\left[\e^{\gamma F-\gamma\theta_{0}\log K}\ge\nu\right]\le C\nu^{-2/\gamma}\e^{-\frac{(\log\nu)^{2}}{2\gamma^{2}\log K}}.\label{eq:coarsebound-universal}
\end{equation}
\end{cor}

\begin{proof}
By \eqref{maxcoarse-easiertouse} of \propref{maxcoarse}, we have,
for $\theta>0$,
\[
\mathbf{P}[F\ge(\theta_{0}+\theta)\log K]\le CK^{2-(\theta_{0}+\theta)^{2}/2}\le CK^{-2\theta-\theta^{2}/2},
\]
with the last inequality because $\theta_{0}>2$. This is equivalent
to
\[
\mathbf{P}[\e^{\gamma F-\gamma\theta_{0}\log K}\ge K^{\gamma\theta}]\le CK^{-2\theta-\theta^{2}/2},
\]
so, putting $\theta=\frac{\log\nu}{\gamma\log K}$, we have
\begin{align*}
\mathbf{P}[\e^{\gamma F-\gamma\theta_{0}\log K}\ge\nu] & \le C\exp\left\{ -(2\theta+\theta^{2}/2)\log K\right\} =C\exp\left\{ -\frac{2\log\nu}{\gamma}-\frac{(\log\nu)^{2}}{2\gamma^{2}\log K}\right\} =C\nu^{-2/\gamma}\e^{-\frac{(\log\nu)^{2}}{2\gamma^{2}\log K}},
\end{align*}
which is \eqref{coarsebound-universal}.
\end{proof}
Finally, we will need a bound on the smoothness of the coarse field.
\begin{lem}
\label{lem:fluctstailbound}There is a constant $C<\infty$ so that
if $\mathbb{B}\subset\mathbb{B}^{*}\subset\mathbb{A}$ are boxes with
aspect ratios between $1/3$ and $3$, then for all $\theta>0$ we
have
\[
\mathbf{P}\left[\max_{x,y\in\mathbb{B}^{\circ}}|h_{\mathbb{A}:\mathbb{B}^{*}}(x)-h_{\mathbb{A}:\mathbb{B}^{*}}(y)|\ge\theta\right]\le C\e^{-\theta^{2}/C}.
\]
\end{lem}

\begin{proof}
By the Cauchy--Schwarz inequality and \lemref{maxvsmin}, we have
\[
\Var\left(\left(h_{\mathbb{A}:\mathbb{B}^{*}}(x)-h_{\mathbb{A}:\mathbb{B}^{*}}(y)\right)-\left(h_{\mathbb{A}:\mathbb{B}^{*}}(x')-h_{\mathbb{A}:\mathbb{B}^{*}}(y')\right)\right)\le C\frac{|x-y|+|x'-y'|}{\diam_{\Euc}\mathbb{B}}
\]
for all $x,x',y,y'\in\mathbb{B}^{\circ}$. The result follows from
this, \lemref{maxvsmin}, Fernique's inequality, and the Borell--TIS
inequality, similarly to the proof of \propref{maxcoarse}.%
\begin{comment}
This follows from \lemref{maxvsmin} and \cite[Lemma 3.4]{DG16}.
Note that the proof of \cite[Lemma 3.4]{DG16} is routine given Dudley's
entropy bound on the expected supremum of a Gaussian process. (See
\cite[Theorem 4.1]{A90} and \cite[equation (7.4), Theorem 7.1]{L01}.)
\end{comment}
\end{proof}

\subsection{Liouville quantum gravity}

In this section we briefly review the properties of the Liouville
quantum gravity measure. We first define the \emph{circle average
process} of a Gaussian free field. If $h_{\mathbb{A}}$ is a Gaussian
free field on a box $\mathbb{A}$, then, for $x\in\mathbb{A}$ so
that $\dist_{\Euc}(x,\partial\mathbb{A})>\eps$, we define $h_{\mathbb{A}}^{\eps}(x)$
to be the integral of $h_{\mathbb{A}}$ against the uniform measure
on $\partial B(x,\eps)$. The Liouville quantum gravity at inverse
temperature $\gamma$ is then supposed to be the limit as $\eps\downarrow0$
of the random measure
\begin{equation}
\mu_{h_{\mathbb{A}},\eps}(\dif x)=\eps^{\gamma^{2}/2}\e^{\gamma h_{\eps}(x)}\,\dif x.\label{eq:LQGapprox}
\end{equation}
Indeed, we have the following result of \cite{DS11}:
\begin{thm}
If $\gamma\in(0,2)$, then there is a random measure $\mu_{h_{\mathbb{A}}}$
such that, with probability $1$, we have
\begin{equation}
\lim_{k\to\infty}\mu_{h_{\mathbb{A}},2^{-k}}=\mu_{h_{\mathbb{A}}}\label{eq:LQGlimit}
\end{equation}
weakly.
\end{thm}

The fact that the convergence in \eqref{LQGlimit} is almost-sure
will be important for us, because in our multiscale analysis we will
often consider the Liouville quantum gravity on \emph{different} boxes,
with coupling induced by the coupling we have induced for the Gaussian
free field. The almost-sure convergence means that this coupling induces
a coupling on the corresponding LQG measures as well.

Throughout the paper, we will treat $\gamma$ as a \emph{fixed} constant
in $(0,2)$. To economize on indices, we will suppress $\gamma$ in
the notation for LQG and subsequently defined objects. All constants
throughout the paper will implicitly depend on $\gamma$. We will
also \emph{fix throughout the paper} a constant $\theta_{0}>2$ so
that
\begin{equation}
\eta\coloneqq\frac{2\gamma\theta_{0}}{4+\gamma^{2}}<1;\label{eq:etadef}
\end{equation}
$\eta$ will also be treated as a fixed constant throughout the paper.
The reason for insisting that $\theta_{0}>2$ is to match with the
condition in \corref{maxcoarse}: we will treat $\theta_{0}\log K$
as the cutoff below which the ``coarse field'' at scale $K$ must
be with high probability. The reason for insisting that $\eta<1$
is so that, according to \propref{conformalcovariance} below, a sub-box
of a box, once the maximum coarse field is considered, will ``look
like'' a strictly smaller box than the larger box from the perspective
of LQG measure.

We will also need the existence of positive and negative moments of
the Liouville quantum gravity measure, which was proved in \cite{Kahane85,Mol96};
see also \cite[Theorems 2.11 and 2.12]{RV14}.
\begin{prop}
\label{prop:momentsbounded}There is a $\nu_{0}>1$ so that if $0\le\nu<\nu_{0}$,
then for all domains $\mathbb{A}\subset\mathbf{R}^{2}$ we have
\begin{equation}
\mathbf{E}\left[\mu_{h_{\mathbb{A}}}(\mathbb{A})^{\nu}\right]<\infty.\label{eq:posmoments}
\end{equation}
Moreover, for any $0\le\nu<\infty$, we have for all domains $\mathbb{A}$
that
\begin{equation}
\mathbf{E}\left[\mu_{h_{\mathbb{A}}}(\mathbb{A})^{-\nu}\right]<\infty.\label{eq:negmoments}
\end{equation}
\end{prop}

\begin{prop}
\label{prop:ballsnonzero}Let $\mathbb{A}$ be a domain. It almost
surely holds that for every $x\in\mathbb{A}$ and $r>0$ so that $B(x,r)\subset\mathbb{A}$,
we have $\mu_{h_{\mathbb{A}}}(B(x,r))>0$.
\end{prop}

\begin{proof}
This follows from a simple union bound and \eqref{negmoments}.
\end{proof}
If $\mu$ and $\nu$ are Radon measures on the same set $\mathbb{X}$,
we say that $\mu\le\nu$ if $\nu-\mu$ is also a measure (not just
a signed measure). If $\mathbb{A}\subset\mathbb{B}$, then we have
by \eqref{LQGapprox} and the smoothness of $h_{\mathbb{B}:\mathbb{A}}$
that
\begin{equation}
\mu_{h_{\mathbb{B}}}|_{\mathbb{A}}=\exp\{\gamma h_{\mathbb{B}:\mathbb{A}}\}\mu_{h_{\mathbb{B}}}\label{eq:equatemeasures}
\end{equation}
almost surely. (Here, $\mu_{h_{\mathbb{B}}}|_{\mathbb{A}}$ denotes
the measure $\mu_{h_{\mathbb{B}}}$ restricted to $\mathbb{A}$.)
This implies that
\begin{equation}
\mu_{h_{\mathbb{B}}}|_{\mathbb{A}}\le\exp\left\{ \gamma\max_{x\in\mathbb{A}}h_{\mathbb{B}:\mathbb{A}}(x)\right\} \mu_{h_{\mathbb{B}}}.\label{eq:muineq}
\end{equation}

An important property of Liouville quantum gravity is the \emph{conformal
covariance} of the measure; see \cite[Theorem 2.8]{B15}:
\begin{prop}
\label{prop:conformalcovariance}Suppose that $\mathbb{V}$ and $\mathbb{V}'$
are domains and $F:\mathbb{V}\to\mathbb{V}'$ is a conformal homeomorphism.
Then we have that
\[
\mu_{h_{\mathbb{V}}}\circ F^{-1}=\e^{(2+\gamma^{2}/2)\log|(F^{-1})'|}\mu_{h_{\mathbb{V}}\circ F^{-1}}\overset{\mathrm{law}}{=}\e^{(2+\gamma^{2}/2)\log|(F^{-1})'|}\mu_{h_{\mathbb{V}'}}.
\]
\end{prop}

\subsection{Metrics defined in terms of measures\label{subsec:metrics-from-measures}}

In this section we describe the process we use to construct a metric
from a measure. The definitions and results in this section are purely
deterministic. Of course, we plan to apply this construction to the
case when the measure is a Liouville quantum gravity measure, which
we will do in the next section, yielding the Liouville graph distance.

If $\mathbb{B}$ is a box, then define the space of \emph{paths} $\mathcal{P}_{\mathbb{B}}$
to be the set of continuous images of $[0,1]$ in $\mathbb{B}$. If
$x,y\in\mathbb{B}$, then define $\mathcal{P}_{\mathbb{B}}(x,y)=\{\pi\in\mathcal{P}_{\mathbb{B}}\mid x,y\in\pi\}$.
Given a box $\mathbb{B}$ and a finite measure $\mu$ on $\mathbb{B}^{(R)}$,
define
\[
\underline{\mathscr{B}}(\mu,\mathbb{B},\delta)=\{\overline{B(x,r)}\mid r\in(0,\diam_{\Euc}(\mathbb{B})),x\in\mathbb{B},\mu(B(x,r))<\delta\}.
\]
Here and throughout the paper, $B(x,r)$ denotes the open Euclidean
ball with center $x$ and radius $r$.
\begin{defn}
Let $\mathbb{B}$ be a domain, $\mu$ a measure on $\mathbb{B}$,
$\delta>0$, and $R>0$. If $\mu$ is a measure on $\mathbb{B}^{(R)}$,
then define the \emph{$\mu$-graph length} of a path $\pi\in\mathcal{P}_{\mathbb{B}}$
as
\[
\underline{d}_{\mu,\mathbb{B},\delta}(\pi)=\min\{n\in\mathbf{N}\::\;\text{there are }B_{1},\ldots,B_{n}\in\underline{\mathscr{B}}(\mu,\mathbb{B},\delta)\text{ such that }\pi\subset B_{1}\cup\cdots\cup B_{n}\}.
\]
We further define
\[
\underline{d}_{\mu,\mathbb{B},\delta}(x,y)=\begin{cases}
0, & x=y;\\
\min\limits _{\pi\in\mathcal{P}_{\mathbb{B}}(x,y)}\underline{d}_{\mu,\mathbb{B},\delta}(\pi), & x\ne y.
\end{cases}
\]
For the purposes of establishing measurability of the distances when
the measure is random, we now show that the distance can be expressed
as the minimum of \emph{countably }many functions. This is a purely
technical point.
\end{defn}

\begin{lem}
\label{lem:GLrational}Let $\mathbb{B}$ be a domain, $\mu$ a measure
on $\mathbb{B}$, and $\delta>0$. Define
\[
\underline{\mathscr{B}}_{\mathbf{Q}}(\mu,\mathbb{B},\delta)=\{\overline{B(x,r)}\mid r\in(0,\diam_{\Euc}(\mathbb{B}))\cap\mathbf{Q},\mu(B(x,r))<\delta\}.
\]
Then we have that
\[
\underline{d}_{\mu,\mathbb{B},\delta}(\pi)=\min\left\{ n\in\mathbf{N}\;:\;\text{there are }B_{1},\ldots,B_{n}\in\underline{\mathscr{B}}_{\mathbf{Q}}(\mu,\mathbb{B},\delta)\text{ such that }\pi\subset B_{1}\cup\cdots\cup B_{n}\right\} .
\]
\end{lem}

\begin{proof}
This follows from the fact that if $B$ is a closed ball such that
$\mu(B)<\delta$, then there is a ball $B'$ containing $B$ so that
$\mu(B')<\delta$ and $B'$ has rational center and rational radius
arbitrarily close to that of $B$.
\end{proof}
It will be more convenient for us to work with a modified graph length,
which has a somewhat better ``continuity'' property with respect
to small perturbations of the field. (See equations \eqref{scaledists}
and \eqref{mostusefulscaling} below.) We will see shortly (\propref{ddunderline})
that the modified graph length differs from the original one only
by at most a factor of $2$. Define
\begin{equation}
\mathscr{B}(\mathbb{B},R)=\{\overline{B(x,r)}\mid r\in(0,R),x\in\mathbb{B}\}.\label{eq:Bsetdef-1}
\end{equation}

\begin{defn}
\label{def:modifiedgraphlength}For $\delta>0$, define
\begin{equation}
\kappa_{\delta}(t)=\max\{1,\delta^{-1}t\}.\label{eq:kappadelta}
\end{equation}
If $\mathbb{B}$ is a domain and $\mu$ is a measure on $\mathbb{B}^{(R)}$,
then define the \emph{$(\mu,\mathbb{B},\delta,R)$-graph length} of
a path $\pi\in\mathcal{P}_{\mathbb{B}}$ as
\[
d_{\mu,\mathbb{B},\delta,R}(\pi)=\inf\left\{ \sum_{k=1}^{n}\kappa_{\delta}(\mu(B_{k}))\;:\;\text{there are }B_{1},\ldots,B_{n}\in\mathscr{B}(\mathbb{B},R)\text{ such that }\pi\subset B_{1}\cup\cdots\cup B_{n}\right\} .
\]
We further define
\[
d_{\mu,\mathbb{B},\delta,R}(x,y)=\begin{cases}
0, & x=y;\\
\min\limits _{\pi\in\mathcal{P}_{\mathbb{B}}(x,y)}d_{\mu,\mathbb{B},\delta,R}(\pi), & x\ne y.
\end{cases}
\]
If $R>\diam_{\Euc}(\mathbb{B})$, then we define $d_{\mu,\mathbb{B},\delta,R}=d_{\mu,\mathbb{B},\delta,\diam_{\Euc}(\mathbb{B})}$,
since any ball of radius greater than $\diam_{\Euc}(\mathbb{B})$
can be replaced by one of radius less than $\diam_{\Euc}(\mathbb{B})$
without changing the minimum.
\end{defn}

\begin{rem}
\label{rem:exchangeratesucks}The difference between the definitions
of $d$ and $\underline{d}$ is that the definition of $d$ allows
``too large'' balls to be used, if one pays a ``price.'' The price,
however, is heuristically very steep, because one has to pay for the
measure of the entire large ball, rather than just the smaller region
around the path that would be required if one used smaller balls.
Thus, we do not expect $d$ and $\underline{d}$ to differ by very
much. The reason to use $d$ instead of $\underline{d}$ is that $d$
changes very little under \emph{small} multiplicative perturbations
of the measure (see \eqref{scaledists} and \eqref{mostusefulscaling}
below), while $\underline{d}$ may change by up to a factor of $2$
even if the measure is multiplied by $1+\eps$ for $\eps\ll1$.
\end{rem}

\begin{rem}
The parameter $R$ of $d$ restricts the maximum Euclidean size of
a ball that can be used to cover the path. If $\mathbb{B}$ and $R$
are of order $1$, and $\delta\ll1$, then we would not expect $R$
to play any significant role, because balls of Euclidean radius of
order $1$ are extremely unlikely to have LQG mass $\delta$. However,
it will often simplify our analysis to work with a metric in which
we know that no balls of radius greater than $R$ are used---more
importantly in order to use \propref{crossmanysmalls} below. We can
then deal with the error incurred by this modification separately.

We prove a measurability lemma for $d$, analogous to \lemref{GLrational}.
\end{rem}

\begin{lem}
\label{lem:MGLQ}Let $\mathbb{B}$ be a domain and $R>0$. Define
\[
\mathscr{B}_{\mathbf{Q}}(\mathbb{B},R)=\{\overline{B(x,r)}\mid r\in(0,R)\cap\mathbf{Q},x\in\mathbb{B}\cap\mathbf{Q}^{2}\}.
\]
Then we have that
\[
\underline{d}_{\mu,\mathbb{B},\delta,R}(\pi)=\inf\left\{ \sum_{k=1}^{n}\kappa_{\delta}(\mu(B_{k}))\;:\;\text{ there are }B_{1},\ldots,B_{n}\in\mathscr{B}(\mathbb{B},R)\text{ such that }\pi\subset B_{1}\cup\cdots\cup B_{n}\right\} .
\]
\end{lem}

\begin{proof}
This follows from the fact that if $B$ is a closed ball, then for
any $\eps>0$ there is a ball $B'$ containing $B$ so that $\mu(B')<\mu(B)+\eps$
and $B'$ has rational center and rational radius arbitrarily close
to that of $B$.
\end{proof}

\subsubsection{Notation for distances}

We will use a flexible notation for the geometrical notions of distance
that we will use. For a box $\mathbb{B}$, we recall that $\mathrm{L}_{\mathbb{B}}$,
$\mathrm{R}_{\mathbb{B}}$, $\mathrm{T}_{\mathbb{B}}$, and $\mathrm{B}_{\mathbb{B}}$
denote the left, right, top, and bottom edges of $\mathbb{B}$, respectively.
Given a box $\mathbb{B}$, we will call the ``easy direction'' across
$\mathbb{B}$ the direction between the longer sides of $\mathbb{B}$,
and the ``hard direction'' across $\mathbb{B}$ the direction between
the shorter sides of $\mathbb{B}$, as illustrated in \figref{easyhard}.
Formally, let
\[
(\mathrm{easy}_{\mathbb{B}},\mathrm{hard}_{\mathbb{B}})=\begin{cases}
((\mathrm{L}_{\mathbb{B}},\mathrm{R}_{\mathbb{B}}),(\mathrm{T}_{\mathbb{B}},\mathrm{B}_{\mathbb{B}})), & \width(\mathbb{B})<\height(\mathbb{B});\\
((\mathrm{T}_{\mathbb{B}},\mathrm{B}_{\mathbb{B}}),(\mathrm{L}_{\mathbb{B}},\mathrm{R}_{\mathbb{B}})), & \height(\mathbb{B})<\width(\mathbb{B}).
\end{cases}
\]
In all of these notations, we will drop the subscript $\mathbb{B}$
when it is clear from context. We will often use this in notation
like $d_{\mu,\mathbb{B},\delta,R}(\mathrm{L},\mathrm{R})$, denoting
the minimum $(\mu,\mathbb{B},\delta,R)$-graph length of paths crossing
from left to right in $\mathbb{B}$, or $d_{\mu,\mathbb{B},\delta,R}(\mathrm{easy})$,
denoting the minimum $(\mu,\mathbb{B},\delta,\mathbb{R})$-graph length
of an easy crossing of $\mathbb{B}$.

Also, define
\begin{equation}
d_{\mu,\mathbb{B},\delta,R}(\min;a)=\min_{\substack{x,y\in\mathbb{B}\\
|x-y|\ge a\diam_{\Euc}\mathbb{B}
}
}d_{\mu,\mathbb{B},\delta,R}(x,y).\label{eq:dminadef}
\end{equation}
Suppose that $\mathbb{A}$ is the intersection of a rectangular annulus
with a rectangle $\mathbb{R}$, both of whose side lengths are greater
than the diameter of the annulus, so that either $\mathbb{A}\subset\mathbb{R}$
or $\mathbb{A}\cap\partial R$ has exactly two connected components.
If $\mu$ is a measure on $\mathbb{A}^{(R)}$, define
\begin{equation}
d_{\mu,\mathbb{A},\delta,R}(\mathrm{around})=\min_{\pi}d_{\mu,\mathbb{A},\delta,R}(\pi),\label{eq:darounddef}
\end{equation}
where $\pi$ ranges over all circuits around $\mathbb{A}$ if $\mathbb{A}$
is an annulus, or simply as
\begin{equation}
d_{\mu,\mathbb{A},\delta,R}(\mathrm{around})=d_{\mu,\mathbb{A},\delta,R}(X,Y),\label{eq:darounddef2}
\end{equation}
where $X$ and $Y$ are the two connected components of $\mathbb{A}\cap\partial\mathbb{R}$,
if $\mathbb{A}$ is simply-connected.

\subsubsection{Comparison inequalities}

The following comparison inequalities are immediate consequences of
\defref{modifiedgraphlength}:
\begin{enumerate}
\item If $\delta\le\delta'$, then 
\begin{equation}
d_{\mu,\mathbb{B},\delta,R}\ge d_{\mu,\mathbb{B},\delta',R}.\label{eq:delta-comparison}
\end{equation}
\item If $R'\le R$, then 
\begin{equation}
d_{\mu,\mathbb{B},\delta,R}\le d_{\mu,\mathbb{B},\delta,R'}.\label{eq:R-comparison}
\end{equation}
\item For all $R>0$ we have
\begin{equation}
d_{\mu,\mathbb{B},\delta,R}=d_{\mu,\mathbb{B},\delta,R\wedge\diam_{\Euc}(\mathbb{B})}.\label{eq:nopointbeinghuge}
\end{equation}
\item If $\mu|_{\mathbb{B}^{(R)}}\le\nu|_{\mathbb{B}^{(R)}}$, then
\begin{equation}
d_{\mu,\mathbb{B},\delta,R}\le d_{\nu,\mathbb{B},\delta,R}.\label{eq:measure-comparison}
\end{equation}
\item If $\mathbb{B}\subset\mathbb{B}'$, then
\begin{equation}
d_{\mu,\mathbb{B}',\delta,R}(x,y)\le d_{\mu,\mathbb{B},\delta,R}(x,y)\label{eq:box-comparison}
\end{equation}
for all $x,y\in\mathbb{B}$.
\item For any $\alpha>0$, we have
\begin{equation}
d_{\alpha\mu,\mathbb{B},\delta,R}=d_{\mu,\mathbb{B},\alpha^{-1}\delta,R}.\label{eq:changemeasurechangedelta}
\end{equation}
If $\alpha\le1$, then
\begin{equation}
d_{\mu,\mathbb{B},\alpha\delta,R}(x,y)\le\alpha^{-1}d_{\mu,\mathbb{B},\delta,R}(x,y).\label{eq:scaledists}
\end{equation}
If $\alpha\ge1$, then
\begin{equation}
d_{\alpha\mu,\mathbb{B},\delta,R}\le\alpha d_{\mu,\mathbb{B},\delta,R}(x,y).\label{eq:mostusefulscaling}
\end{equation}
\item The triangle inequality holds: if $\pi=\pi_{1}\cup\pi_{2}$, then
\begin{equation}
d_{\mu,\mathbb{B},\delta,R}(\pi)\le d_{\mu,\mathbb{B},\delta,R}(\pi_{1})+d_{\mu,\mathbb{B},\delta,R}(\pi_{2}).\label{eq:triangleineq}
\end{equation}
\end{enumerate}
\begin{rem}
All of the above properties with the exception of \eqref{scaledists}
and \eqref{mostusefulscaling}, as well as those involving $R$, also
apply for $\underline{d}$. Note that $\underline{d}$ could have
been defined using the $R$ parameter in the same way as $d$, and
then \eqref{R-comparison} and \eqref{nopointbeinghuge} would hold
for $\underline{d}$ as well, but we will not need this in the paper.
\end{rem}

We will need a bound in the opposite direction for \eqref{R-comparison}:
a bound which controls how much the Liouville graph distance can increase
when we decrease $R$, requiring the use of smaller balls. We will
prove this as \lemref{RprimeR}. This first requires the following
lemma and definition.
\begin{lem}
\label{lem:kappascale}Suppose that $a_{1}+\cdots+a_{n}\le b$. Then
we have (recalling that $\kappa_{\delta}$ was defined in \eqref{kappadelta})
that
\[
\sum_{k=1}^{n}\kappa_{\delta}(a_{k})\le\kappa_{\delta}(b)+n.
\]
\end{lem}

\begin{proof}
We have that $\delta^{-1}t\le\kappa_{\delta}(t)\le\delta^{-1}t+1$
for all $t$, so
\[
\sum_{k=1}^{n}\kappa_{\delta}(a_{k})\le\sum_{k=1}^{n}(\delta^{-1}a_{k}+1)\le\delta^{-1}b+n\le\kappa_{\delta}(b)+n.\qedhere
\]
\end{proof}
\begin{defn}
\label{def:geodesics}Suppose that $\mathbb{B}$ is a domain, $\mu$
is a measure on $\mathbb{B}$, and $\delta,R>0$. Let $x,y\in\mathbb{B}$
and suppose that $\pi$ is a path in $\mathbb{B}$ connecting $x$
and $y$. We will say that $\pi$ is a \emph{$(\mu,\mathbb{B},\delta,R)$-geodesic
}if $\pi$ consists of a sequence of straight line segments between
successive points $x,z_{1},\ldots,z_{N},y\in\mathbb{B}$ and there
is a sequence of radii $r_{1},\ldots,r_{N}\in(0,R)$ so that $\sum_{j=1}^{N}\mu(\overline{B(z_{i},r_{i})})=d_{\mu,\mathbb{B},\delta,R}(x,y)$
and $\pi\subset\bigcup_{j=1}^{N}\overline{B(z_{j},r_{j})}$. We will
call $z_{1},\ldots,z_{N}$ the corresponding \emph{geodesic points,}
$r_{1},\ldots,r_{N}$ the corresponding \emph{geodesic radii}, and
$\overline{B(z_{1},r_{1})},\ldots,\overline{B(z_{N},r_{N})}$ the
corresponding \emph{geodesic balls}.
\end{defn}

\begin{lem}
\label{lem:RprimeR}There is a constant $C$ so that if $B$ is a
domain, $\mu$ is a measure on $\mathbb{B}$, $\delta>0$, and $0<R'\le R$,
then 
\begin{equation}
d_{\mu,\mathbb{B},\delta,R'}(x,y)\le d_{\mu,\mathbb{B},\delta,R}(x,y)+C\frac{R\Leb(\mathbb{B}^{(R)})}{(R')^{3}},\label{eq:RprimeR}
\end{equation}
where $\Leb$ denotes the Lebesgue measure, and also
\begin{equation}
d_{\mu,\mathbb{B},\delta,R'}(x,y)\le(1+CR/R')d_{\mu,\mathbb{B},\delta,R}(x,y).\label{eq:RprimeR-multiplicative}
\end{equation}
\end{lem}

\begin{proof}
Let $x,y\in\mathbb{B}$. Let $B_{1},\ldots,B_{N}$ be a set of geodesic
balls for a geodesic between $x$ and $y$. We note that for any $1\le i<j<k\le N$,
we have that $B_{i}\cap B_{j}\cap B_{k}=\emptyset$; if not, then
one of the balls could be eliminated to get a shorter distance. Therefore,
if $S=\{i\in\{1,\ldots,N\}\mid\diam_{\Euc}(B_{i})\ge R'\}$, then
we have that
\[
\frac{1}{2}|S|(R')^{2}\le\sum_{i\;:\;\diam_{\Euc}(B_{i})\ge R'}\Leb(B_{i})\le2\Leb\left(\bigcup_{i\;:\;\diam_{\Euc}(B_{i})\ge R'}B_{i}\right)\le2\Leb(\mathbb{B}^{(R)}),
\]
so
\[
|S|\le4\pi\frac{\Leb(\mathbb{B}^{(R)})}{(R')^{2}}.
\]
Now for each $i\in S$, $B_{i}$ can be replaced by a set of at most
$2\pi\lceil R/R'\rceil$ subset balls of radius $R'$, with centers
in $\mathbb{B}$, so that a path from $x$ to $y$ in $\mathbb{B}$
is still covered, and the new balls in each older ball have total
measure no larger than the measure of the original ball. This implies,
by \lemref{kappascale}, that
\begin{equation}
d_{\mu,\mathbb{B},\delta,R'}(x,y)\le d_{\mu,\mathbb{B},\delta,R}(x,y)+2\pi\lceil R/R'\rceil|S|\le d_{\mu,\mathbb{B},\delta,R}(x,y)+C\frac{R\Leb(\mathbb{B}^{(R)})}{(R')^{3}}\label{eq:RrprimeintermsofS}
\end{equation}
for some absolute constant $C$, which completes the proof of \eqref{RprimeR}.
On the other hand, \eqref{RprimeR-multiplicative} follows from \eqref{RrprimeintermsofS}
by the trivial bound $|S|\le N$.
\end{proof}
We will also want to relate our modified distances $d$ back to our
original distances of interest $\underline{d}$. This turns out to
be very simple, but we first require a definition.
\begin{defn}
\label{def:ballsnonzero}By an \emph{admissible measure} on a set
$\mathbb{A}$ we will mean a nonatomic Radon measure $\mu$ such that
for every $x\in\mathbb{A}$ and $r>0$ such that $B(x,r)\subset\mathbb{A}$,
we have
\begin{equation}
\mu(B(x,r))>0.\label{eq:ballsnonzero}
\end{equation}

By \propref{ballsnonzero}, the Liouville quantum gravity measures
are admissible almost surely. Note also that if $\mu$ is an admissible
measure then $\alpha\mu$ is also an admissible measure for any $\alpha\in(0,\infty)$.
Now we can state the relationship between $d$ and $\underline{d}$.
\end{defn}

\begin{prop}
\label{prop:ddunderline}If $\mu$ is an admissible measure on $\mathbb{B}$
and $\delta,R>0$, then
\[
d_{\mu,\mathbb{B},\delta,R}(x,y)\le\underline{d}_{\mu,\mathbb{B},\delta,R}(x,y)\le2d_{\mu,\mathbb{B},\delta,R}(x,y).
\]
\end{prop}

The first inequality is trivial by \eqref{kappadelta}. The second
follows from the following lemma.
\begin{lem}
\label{lem:internaltangent}Suppose that $\mu$ is an admissible measure
on $\mathbb{B}$, $\alpha\in(0,1)$, $B$ is a closed ball so that
$\mu(B)=\delta$, and $x,y\in\partial B$. Then there are closed balls
$B',B''\subset B$ so that $\mu(B_{1})<\alpha\delta$, $\mu(B_{2})<(1-\alpha)\delta$,
$B'\cup B''$ is path-connected, and $x,y\in B_{1}\cup B_{2}$.
\end{lem}

\begin{proof}
For $r\in[0,1]$, let $B_{r}$ be the closed ball of radius $r$ which
is internally tangent to $B$ at $x$, so $B_{0}=\{x\}$ and $B_{1}=B$.
Let $\tilde{B}_{r}$ be the closed ball which is internally tangent
to $B$ at $y$ and externally tangent to $B_{r}$. (See \figref{internaltangent}.)
\begin{figure}
\centering
% circles generated with GeoGebra
\begin{tikzpicture}[x=0.1in,y = 0.1in,label distance=-0.07in]
\filldraw [fill opacity=0.1] (2.9932734401629206,1.4491422063763153) circle (3.325612578441148);
\filldraw [fill opacity=0.1] (-1.627700077915068,-0.7880231891211238) circle (1.8084214360146655);
\filldraw [fill opacity=0.05] (0.7950078747294753,2.18543439716675) circle (5.643908868044657);
\draw (2.9932734401629206,1.4491422063763153) node[label=center:{$\tilde{B}_r$}] {};
\draw (-1.627700077915068,-0.7880231891211238) node[label=center:{$B_r$}] {};
\draw (-1.5,4) node[label=center:{$B$}] {};
\fill (-2.77,-2.19) circle (2pt) node [label=below left:{$x$}] {}; 
\fill (6.146701087847164,0.39292580207907596) circle (2pt) node [label=below right:{$y$}] {}; 
\end{tikzpicture}

\caption{Construction in the proof of \lemref{internaltangent}.\label{fig:internaltangent}}
\end{figure}
 Let $f(r)=\mu(B_{r})$ and let $\tilde{f}(r)=\mu(\tilde{B}_{r})$.
We note that $f$ is increasing and $\tilde{f}$ is decreasing. Since
$\mu$ is nonatomic, we have that 
\[
f(r)+\tilde{f}(r)\le\mu(B)=\delta.
\]
We claim that $f$ is right-continuous. Indeed, if $r_{n}\downarrow r$,
then we have
\[
\lim_{n\to\infty}f(r_{n})=\mu\left(\bigcap_{n=1}^{\infty}B_{r_{n}}\right)=\mu(B_{r}).
\]
Similarly, $\tilde{f}$ is left-continuous, because if $r_{n}\uparrow r$,
then we have
\[
\lim_{n\to\infty}\tilde{f}(r_{n})=\mu\left(\bigcap_{n=1}^{\infty}\tilde{B}_{r_{n}}\right)=\mu(\tilde{B}_{r}).
\]
In particular, this implies that $f$ and $\tilde{f}$ are both upper-semicontinuous.
Let
\begin{equation}
r_{*}=\sup\{r\in[0,1]\mid f(r)<\alpha\delta\}.\label{eq:rstardef}
\end{equation}
Because $f$ is upper-semicontinuous, we have that
\begin{equation}
f(r_{*})\ge\alpha\delta,\label{eq:frstar}
\end{equation}
so\,
\begin{equation}
\tilde{f}(r_{*})\le(1-\alpha)\delta.\label{eq:ftilderstar}
\end{equation}
Now if $\tilde{f}(r_{*})=(1-\alpha)\delta$, then \eqref{frstar}
implies that $f(r_{*})=\alpha\delta$, so $\mu(B\setminus(B'\cup B''))=0$,
contradicting the assumption that $\mu$ is admissible. Therefore,
\eqref{ftilderstar} implies that we must in fact have $\tilde{f}(r_{*})<(1-\alpha)\delta$.
By the left-continuity of $\tilde{f}$, there is an $r'<r_{*}$ so
that $\tilde{f}(r')<(1-\alpha)\delta$ as well, so since $r'<r_{*}$,
we must have that $f(r')<\alpha\delta$ by \eqref{rstardef} and the
fact that $f$ is increasing. Thus we can take $B'=B_{r'}$ and $B''=\tilde{B}_{r'}$
to complete the proof.
\end{proof}
Admissibility of a measure also implies that the crossing distance
of a box can be made arbitrarily large by requiring the geodesic balls
to have sufficiently small LQG measure.
\begin{lem}
\label{lem:limtoinfty-general}If $\mu$ is an admissible measure
on $\mathbb{B}$ and $R>0$, then
\begin{equation}
\lim_{\delta\downarrow0}d_{\mu,\mathbb{B},\delta,R}(\mathrm{L},\mathrm{R})=\infty.\label{eq:limtoinfty-general}
\end{equation}
\end{lem}

\begin{proof}
For any $r>0$, there is a finite collection of balls $B_{1},\ldots,B_{N}$
so that if $x\in\mathbb{B}$, then there is an $i\in\{1,\ldots,N\}$
so that $B_{i}\subset B(x,r)$. Let $\delta_{0}=\min\limits _{1\le i\le N}\mu(B_{i})$,
which is positive by \eqref{ballsnonzero}. Then if $\delta<\delta_{0}$,
any ball of $\mu$-measure at most $\delta$ must have radius at most
$r$, so
\[
d_{\mu,\mathbb{B},\delta,R}(\mathrm{L},\mathrm{R})\ge\lfloor r^{-1}\dist_{\Euc}(\mathrm{L}_{\mathbb{B}},\mathrm{R}_{\mathbb{B}})\rfloor,
\]
which implies the desired limit.
\end{proof}
The last two propositions of this section will be used in lower bounds
on the crossing distances. The first (\propref{crossingbigimpliescrossingsmall})
implies, informally, that if the geodesic between two points crosses
a smaller box, then the two points must be at least as far apart as
the crossing distance of the smaller box. The second (\propref{crossmanysmalls})
extends this to the situation where a geodesic between two points
crosses many smaller boxes, and says that the distance between the
two points must be at least the \emph{sum} of the crossing distances
of the boxes. The caveat, however, is that the smaller boxes must
be sufficiently separated so that a single geodesic ball cannot cover
several of the smaller boxes. This is the primary importance of the
parameter $R$---it is used to prevent the use of geodesic balls
with large Euclidean diameter that would cover many smaller boxes.
\begin{prop}
\label{prop:crossingbigimpliescrossingsmall}Suppose that $\mathbb{C}\subset\mathbb{R}$
are boxes and $\mu$ is a measure on $\mathbb{R}^{(R)}$. Let $\pi$
be a $(\mu,\mathbb{R},\delta,R)$-geodesic between two points in $\mathbb{R}$.
Suppose that $w,z\in\mathbb{C}\cap\pi$, that $w$ appears before
$z$ on the path $\pi$, and that the part of $\pi$ between $w$
and $z$ is completely contained in $\mathbb{C}$. Then
\[
d_{\mu,\mathbb{C},\delta,R}(w,z)=d_{\mu,\mathbb{R},\delta,R}(w,z).
\]
\end{prop}

\begin{proof}
By \eqref{box-comparison}, it is sufficient to prove that
\[
d_{\mu,\mathbb{C},\delta,R}(w,z)\le d_{\mu,\mathbb{R},\delta,R}(w,z).
\]
Let $y_{1},\ldots,y_{d}$ and $r_{1},\ldots,r_{d}$ be the geodesic
centers and radii for $\pi$, respectively. Let $y_{j},\ldots,y_{k}$
be the minimal substring of the $y_{i}$s so that the segment of $\pi$
between $w$ and $z$ is contained in $\overline{B(y_{j},r_{j})}\cup\cdots\cup\overline{B(y_{k},r_{k})}$.
Then $d_{\mu,\mathbb{R},\delta,R}(w,z)=\sum_{i=j}^{k}\kappa_{\delta}(\mu(\overline{B(y_{i},r_{i})}))$.
For each $j+1\le i\le k-1$, we have $y_{i}\in\mathbb{C}$, so $\overline{B(y_{i},r_{i})}\in\mathscr{B}(\mathbb{C},R)$.
Let $u$ be the intersection of the segment between $y_{j}$ and $y_{j+1}$
with $\partial\mathbb{C}$. We can create a closed ball $B$ around
$u$ that is contained in $\overline{B(y_{j},r_{j})}$ and internally
tangent to $\overline{B(y_{j},r_{j})}$ at every element of $\pi\cap\mathbb{C}\cap\partial B(y_{j},r_{j})$
(which is a singleton unless $y_{j}\in\partial\mathbb{C}$). Then,
we have $B\in\mathscr{B}(\mathbb{C},R)$, $\mu(B)\le\mu(\overline{B(y_{j},r_{j})})$,
and $\pi\cap\overline{B(y_{j},r_{j})}\cap\mathbb{C}\subset B$. We
can perform a similar construction around $u'$, the intersection
of the segment between $y_{k-1}$ and $y_{k}$ with $\partial\mathbb{C}$,
to get a closed ball $B'\in\mathscr{B}(\mathbb{C},R)$ so that $\pi\cap\overline{B(y_{k},r_{K})}\cap\mathbb{C}\subset B'$.
(See \figref{passcrossplus1-1-1}.) 
\begin{figure}
\centering{}\begin{tikzpicture}[x=2.7in,y=2.7in,label distance=-0.07in,remember picture]

\begin{pgfinterruptboundingbox}
\clip (-0.2,0.1) rectangle (1.2,1.2);
\end{pgfinterruptboundingbox}
\draw [black, name path=R1] (0,0.3) rectangle (1,1.2);
\coordinate (zi) at (-0.1,0.7);
\coordinate (zi1) at (0.2,0.6);
\fill (zi) circle (2pt) node [label=above left:{$y_j$}] {};
\fill (zi1) circle (2pt)  node [label=below right:{$y_{j+1}$}] {};
\filldraw [fill opacity=0.1,name path=ziball] (zi) circle (0.19);
\draw [name path=l1] (zi) -- (zi1);
\filldraw [fill opacity=0.1] (zi1) circle (0.15);
\node[coordinate, name intersections = {of = l1 and R1}] (u) at  (intersection-1) {};
\node[coordinate, name intersections = {of = l1 and ziball}] (zizi1edge) at (intersection-1) {};
\fill[red] (u) circle (2pt) node [label = above right:{$u$}] {};
%\fill[green] (zizi1edge) circle (2pt) node [label = below right:{$p$}] {};
\node [draw,fill,fill opacity=0.1,red] at (u) [circle through={(zizi1edge)}] {};
%\filldraw [fill opacity=0.1,red] (u) circle (0.2);
\coordinate (zi2) at (0.4,0.9);
\fill (zi2) circle (2pt) node {};
\filldraw [fill opacity = 0.1] (zi2) circle (0.25);
\coordinate (zi3) at (0.7,0.8);
\draw (zi1)--(zi2);
\draw (zi2)--(zi3);
\fill (zi3) circle (2pt) node [label = above right:{$y_{k-1}$}] {};
\filldraw [fill opacity = 0.1] (zi3) circle (0.22);
\coordinate (zi4) at (1.05,0.6);
\fill (zi4) circle (2pt) node [label=above right:{$y_{k}$}] {};
\filldraw [fill opacity = 0.1,name path=zjball] (zi4) circle (0.24);
\draw [name path=l2] (zi3)--(zi4);
\node[coordinate, name intersections = {of = l2 and R1}] (up) at  (intersection-1) {};
\node[coordinate, name intersections = {of = l2 and zjball}] (zjzj1edge) at (intersection-1) {};
\fill[red] (up) circle (2pt) node [label = below left:{$u'$}] {};
\node [draw,fill,fill opacity=0.1,red] at (up) [circle through={(zjzj1edge)}] {};
\draw (zi4) -- (1.3,0.55);
\draw (-0.2,0.65) -- (zi);
\end{tikzpicture}\caption{\label{fig:passcrossplus1-1-1}Construction in the proof of \propref{crossingbigimpliescrossingsmall}.}
\end{figure}
 This completes the proof, since a path from $u$ to $u'$ is contained
in $B\cup\overline{B(y_{j+1},r_{j+1})}\cup\cdots\cup\overline{B(y_{k-1},r_{k-1})}\cup B'$.
(If $j=k$, then $B'$ is not necessary and the construction still
works.)
\end{proof}
\begin{prop}
\label{prop:crossmanysmalls}Suppose that $\mathbb{C}_{1},\ldots,\mathbb{C}_{J}\subset\mathbb{R}$
are boxes,
\begin{equation}
R<\min_{1\le i,j\le J}\dist_{\Euc}(\mathbb{C}_{i},\mathbb{C}_{j}),\label{eq:Rnottoobig}
\end{equation}
and $\mu$ is a measure on $\mathbb{R}^{(R)}$. Let $x,y\in\mathbb{R}$
and let $\pi$ be a $(\mu,\mathbb{R},\delta,R)$-geodesic between
$x$ and $y$. Suppose that, for each $1\le i\le j$, we have $w_{i},z_{i}\in\mathbb{C}_{i}\cap\pi$,
with $w_{i}$ appearing before $z_{i}$ on the path $\pi$, and the
part of $\pi$ between $w_{i}$ and $z_{j}$ being completely contained
in $\mathbb{C}_{i}$. Then
\[
d_{\mu,\mathbb{R},\delta,R}(x,y)\ge\sum_{i=1}^{J}d_{\mu,\mathbb{R},\delta,R}(w_{i},z_{i}).
\]
\end{prop}

\begin{proof}
This follows from the proof of \propref{crossingbigimpliescrossingsmall}
when we note that the condition \eqref{Rnottoobig} implies that a
single ball cannot intersect two distinct $\mathbb{C}_{i}$s.
\end{proof}

\subsection{Liouville graph distance}

We now discuss the graph metric with respect to the Liouville quantum
gravity measure: the Liouville graph distance.
\begin{defn}
If $\mathbb{A}\subset\mathbb{B}$ are boxes and $\delta>0$, $\gamma\in(0,2)$,
and $R\le\dist_{\Euc}(\mathbb{A},\partial\mathbb{B})$, then we use
the notation
\[
d_{\mathbb{A},\mathbb{B},\delta,R}=d_{\mu_{h,\gamma},\mathbb{B},\delta,R}.
\]
\end{defn}

To avoid a proliferation of subscripts, we will often abbreviate the
notation for distances. If $\mathbb{A}$ is unspecified, we will mean
that $\mathbb{A}=\mathbb{B}^{*}$. If $R$ is unspecified then we
will mean that $R=\diam_{\Euc}(\mathbb{B})$ (which is a natural choice
in light of \eqref{nopointbeinghuge}), and if in addition to $R$,
the parameter $\delta$ is also unspecified, then we will mean that
$\delta=1$. Thus we have
\begin{align*}
d_{\mathbb{A},\mathbb{B},\delta} & =d_{\mathbb{A},\mathbb{B},\delta,\diam_{\Euc}(\mathbb{B})},\\
d_{\mathbb{B},\delta} & =d_{\mathbb{B}^{*},\mathbb{B},\delta,\diam_{\Euc}(\mathbb{B}),}\\
d_{\mathbb{A},\mathbb{B}} & =d_{\mathbb{A},\mathbb{B},1,\diam_{\Euc}(\mathbb{B})},\\
d_{\mathbb{B}} & =d_{\mathbb{B}^{*},\mathbb{B},1,\diam_{\Euc}(\mathbb{B})}.
\end{align*}
(Of course syntactically the second and third lines are ambiguous,
but it will always be clear, from the font if nothing else, which
variables represent boxes and which represent numbers, so there is
no risk of confusion.)
\begin{defn}
We define $\Theta_{\mathbb{A},\mathbb{B},\delta,R}^{\mathrm{LR}}(p)$,
$\Theta_{\mathbb{A},\mathbb{B},\delta,R}^{\mathrm{easy}}(p)$, and
$\Theta_{\mathbb{A},\mathbb{B},\delta,R}^{\mathrm{hard}}(p)$ to be
the $p$th quantiles of the random variables $d_{\mathbb{A},\mathbb{B},\delta,R}(\mathrm{L},\mathrm{R})$,
$d_{\mathbb{A},\mathbb{B},\delta,R}(\mathrm{easy})$, and $d_{\mathbb{A},\mathbb{B},\delta,R}(\mathrm{hard})$,
respectively. We will abbreviate these notations in the same way as
the $d$ notations.
\end{defn}

We now establish several properties of the Liouville graph distance.
\begin{prop}
If $\mathbb{B}^{(R)}\subset\mathbb{A}$ and $\delta,R>0$, then the
processes
\[
\{\underline{d}_{\mathbb{A},\mathbb{B},\delta,R}(x,y)\}_{x,y\in\mathbb{B}}
\]
and
\[
\{d_{\mathbb{A},\mathbb{B},\delta,R}(x,y)\}_{x,y\in\mathbb{B}}
\]
are measurable with respect to the field $h_{\mathbb{A}}$.
\end{prop}

\begin{proof}
This follows from \lemref{GLrational} and \lemref{MGLQ}, since both
distances can be written as a minimum of countably many functions.
\end{proof}
\begin{prop}
If $\mathbb{A}_{1}\cap\mathbb{A}_{2}=\emptyset$, $R_{1},R_{2}>0$,
$\mathbb{B}_{i}^{(R_{i})}\subset\mathbb{A}_{i}$, and $\delta_{i}>0$
for $i=1,2$, then $d_{\mathbb{A}_{1},\mathbb{B}_{1},\delta_{1},R_{1}}$
and $d_{\mathbb{A}_{2},\mathbb{B}_{2},\delta_{2},R_{2}}$ are independent.
\end{prop}

\begin{proof}
This follows from the independence of the Gaussian free field on disjoint
boxes (\propref{independent-disjoint}).
\end{proof}

\subsubsection{Covariance properties of the Liouville graph distance}

The conformal covariance of the LQG measure heuristically implies
conformal covariance of the Liouville graph distance. However, because
the definition of Liouville graph distance specifies \emph{Euclidean}
balls, which are not in general taken by conformal maps to Euclidean
balls, this is not exactly true. The next two propositions observe
that the Liouville graph distance is covariant under Euclidean similarity
transformations, which do of course takes Euclidean balls to Euclidean
balls. The RSW result presented in \secref{RSW} will rely on an approximate
more general conformal covariance of the Liouville graph distance,
which will be proved in that section.

Our first result is that the Liouville graph distance is invariant
under Euclidean isometries.
\begin{prop}
If $\mathbb{B}\subset\mathbb{A}$, $\delta,R>0$, and $f$ is a translation
or rotation map, then 
\[
\left\{ d_{\mathbb{A},\mathbb{B},\delta,R}(x,y)\right\} _{x,y\in\mathbb{B}}\overset{\mathrm{law}}{\mathrm{=}}\left\{ d_{f(\mathbb{A}),f(\mathbb{B}),\delta,R}(f(x),f(y))\right\} _{x,y\in\mathbb{B}}.
\]
\end{prop}

\begin{proof}
This follows from the conformal covariance (\propref{conformalcovariance})
and the fact that translations and rotations take Euclidean balls
to Euclidean balls.
\end{proof}
Now we show how the Liouville graph distance transforms under scalings.
\begin{prop}
\label{prop:LQGscaling}If $\mathbb{B}\subset\mathbb{A}$ and $\delta,R>0$,
then for any $\alpha>0$, we have
\[
\left\{ d_{\alpha\mathbb{A},\alpha\mathbb{B},\alpha^{\gamma^{2}/2+2}\delta,\alpha R}(\alpha x,\alpha y)\right\} _{x,y\in\mathbb{B}}\overset{\mathrm{law}}{\mathrm{=}}\left\{ d_{\mathbb{A},\mathbb{B},\delta,R}(x,y)\right\} _{x,y\in\mathbb{B}}.
\]
\end{prop}

\begin{proof}
Let $F:\mathbb{A}\to\alpha\mathbb{A}$ be given by scaling by $\mathbb{\alpha}$.
Then $F$ is a conformal map, so we have that
\[
\mu_{\alpha\mathbb{A}}\overset{\mathrm{law}}{=}\alpha^{2+\gamma^{2}/2}(\mu_{\mathbb{A}}\circ F^{-1})
\]
by the conformal covariance (\propref{conformalcovariance}). Then
we have, recalling \eqref{Bsetdef-1}, that
\begin{align*}
\mathscr{B}(\alpha\mathbb{B},\alpha R) & =\{\overline{B(x,r)}\mid r\in(0,\alpha R),x\in\alpha\mathbb{B}\}=\{\alpha\overline{B(\alpha^{-1}x,\alpha^{-1}r)}\mid r\in(0,\alpha R),x\in\alpha\mathbb{B}\}\\
 & =\{\alpha\overline{B(x,r)}\mid r\in(0,R),x\in\mathbb{B}\}=\alpha\mathscr{B}(\mathbb{B},R),
\end{align*}
and that
\[
\{\kappa_{\alpha^{2+\gamma^{2}/2}\delta}(\mu(\alpha B))\}_{B\in\mathscr{B}(\mathbb{B},R)}=\{\kappa_{\delta}(\alpha^{-2-\gamma^{2}/2}\mu(\alpha B))\}_{B\in\mathscr{B}(\mathbb{B},R)}\overset{\mathrm{law}}{=}\{\kappa_{\delta}(\mu(B))\}_{B\in\mathscr{B}(\mathbb{B},R)},
\]
which means that
\[
\left\{ d_{\alpha\mathbb{A},\alpha\mathbb{B},\alpha^{\gamma^{2}/2+2}\delta,\alpha R}(\alpha x,\alpha y)\right\} _{x,y\in\mathbb{B}}\overset{\mathrm{law}}{\mathrm{=}}\left\{ d_{\mathbb{A},\mathbb{B},\delta,R}(x,y)\right\} _{x,y\in\mathbb{B}}.\qedhere
\]
\end{proof}
We note that a simple consequence of \propref{LQGscaling} and \eqref{delta-comparison}
is the following monotonicity of the quantiles:
\begin{prop}
\label{prop:thetamonotone}If $\alpha\ge1$, then for all boxes $\mathbb{B}\subset\mathbb{A}$,
all $p\in(0,1)$, and all $R>0$ we have
\begin{equation}
\Theta_{\alpha\mathbb{A},\alpha\mathbb{B},1,\alpha R}^{\mathrm{LR}}(p)=\Theta_{\mathbb{A},\mathbb{B},\alpha^{-\gamma^{2}/2-2},R}^{\mathrm{LR}}(p)\ge\Theta_{\mathbb{A},\mathbb{B},R}^{\mathrm{LR}}(p).\label{eq:thetamonotone}
\end{equation}
\end{prop}

\begin{lem}
\label{lem:limtoinfty}For any $\mathbb{B}\subset\mathbb{A}$, any
$p\in(0,1)$, and any $R>0$, we have
\[
\lim_{\delta\to0}\Theta_{\mathbb{A},\mathbb{B},\delta,R}^{\mathrm{LR}}(p)=\infty.
\]
\end{lem}

\begin{proof}
This follows from \lemref{limtoinfty-general} and the fact that the
LQG measure is admissible almost surely.
\end{proof}
We also record a simple but useful consequence of the general facts
established in \subsecref{metrics-from-measures}.
\begin{prop}
\label{prop:fieldonsmallerbox}Suppose that $\mathbb{B}\subset\mathbb{B}^{\circ}\subset\mathbb{A}\subset\mathbb{A}'$.
We have
\begin{equation}
d_{\mathbb{A},\mathbb{B},\delta',R}\le d_{\mathbb{A}',\mathbb{B},\delta,R}\le d_{\mathbb{A},\mathbb{B},\delta'',R},\label{eq:scaledeltabymax}
\end{equation}
where
\begin{align*}
\delta' & =\exp\{-\gamma\min_{z\in\mathbb{B}^{\circ}}h_{\mathbb{A}':\mathbb{A}}(z)\}\delta,\\
\delta'' & =\exp\{-\gamma\max_{z\in\mathbb{B}^{\circ}}h_{\mathbb{A}':\mathbb{A}}(z)\}\delta.
\end{align*}
Moreover, we have
\begin{equation}
d_{\mathbb{A}',\mathbb{B},\delta,R}\le\exp\left\{ \gamma\left(\max_{z\in\mathbb{B}^{\circ}}h_{\mathbb{A}':\mathbb{A}}(z)\right)^{+}\right\} d_{\mathbb{A},\mathbb{B},\delta,R}.\label{eq:ddi}
\end{equation}
\end{prop}

\begin{proof}
The inequalities in  \eqref{scaledeltabymax} follow immediately from
\eqref{muineq} (relating the measures by multiplicative factors),
\eqref{measure-comparison} (going from a relationship between the
measures to a relationship between the distances) and \eqref{changemeasurechangedelta}
(relating a constant factor in the measure to a constant factor in
$\delta$). Then \eqref{ddi} is implied by \eqref{scaledists}.
\end{proof}
\begin{rem}
We have two ways of understanding the effect on the distance when
the measure is multiplied by a constant (``coarse field''): equality
\eqref{changemeasurechangedelta} (generally applied through \eqref{scaledeltabymax})
and inequality \eqref{scaledists} (often applied through \eqref{ddi}).
As pointed out in \remref{exchangeratesucks}, the estimate \eqref{scaledists}
(and thus also the estimate \eqref{ddi}) is poor unless $\alpha$
is very close to $1$. When we perform estimates that are saturated
as $\gamma\uparrow2$, we thus always need to use \eqref{changemeasurechangedelta}
rather than \eqref{scaledists}. On the other hand, the change in
$\delta$ in \eqref{changemeasurechangedelta} is less suitable if
we want to compare distances for the same path with respect to a slightly
changed measure---in such cases, \eqref{scaledists} will play a
crucial role.
\end{rem}

\section{The RSW theory\label{sec:RSW}}

In this section we give the proof of an RSW result, relating easy
crossings to hard crossings in our setting. This was previously done
in the high-temperature regime by the present authors in \cite{DD16},
using a framework used to prove an RSW result for Voronoi percolation
in \cite{tassion}. In the present paper, we follow the much simpler
RSW proof method which was developed in \cite{DF18} for first-passage
percolation on fields with an approximate conformal invariance property.
In this section, we adapt that proof to our setting. 

The proof of the RSW result comes essentially in two steps. The first
shows that the Liouville graph distance has an approximate conformal
invariance property. This is the main difference between the proof
we give here and that of \cite{DF18}. In \cite{DF18}, the authors
showed that the field on which they consider first-passage percolation
changes only by a limited amount when it is compared to an appropriately
coupled field on another domain, pulled back by a conformal isomorphism.
In our setting, our underlying field (the GFF) is \emph{exactly }conformally
invariant. However, our definition of graph distance is based on Euclidean
balls, which are not preserved by conformal transformations. Therefore,
to show that the Liouville graph distance is approximately conformally
invariant, we need to show that the Euclidean balls, when deformed
by the conformal maps that we consider, can be replaced by only a
constant number of Euclidean balls which lie inside the deformed balls.

The second step of the argument is then very similar to that of \cite{DF18}.
We consider a tall-and-skinny and a short-and-fat rectangle, and find
a conformal isomorphism between them using the Riemann mapping theorem.
(Actually, we need to use ellipses so that conformal maps have better
smoothness properties, but these can be related back to rectangles
by a simple estimate.) If we could map the vertical sides of one rectangle
to the vertical sides of the other, then it would be clear how to
map a left-right crossing of one rectangle to a left-right crossing
of another. This is of course impossible, but we can solve the problem
by finding some small segment of the right-hand side of the long rectangle
which has a geodesic coming to it with positive probability, and can
be mapped inside the right-hand side of the short rectangle. We explain
the details in \subsecref{the-rsw-proof}.

\subsection{Approximate conformal invariance}

The key ingredient of approximate conformal invariance in our setting
is the following.
\begin{prop}
\label{prop:conformalballs}Suppose that $\mathbb{U}$ and $\mathbb{V}$
are bounded domains (open subsets of $\mathbf{R}^{2}$) so that $\overline{\mathbb{U}}\subset\mathbb{V}$.
Suppose further that $\mathbb{V}'$ is another domain and $F:\mathbb{U}\to\mathbb{V}'$
is a conformal homeomorphism. Finally, suppose that $R,R'>0$. Then
there is an $N=N(\mathbb{U},\mathbb{V},\mathbb{V}',F,R,R')\in\mathbf{N}$,
such that for any closed ball $B\subset\mathbb{U}$ of radius at most
$R$, any two points $y_{1},y_{2}\in\partial F(B)$, and either connected
component $X$ of $\partial F(B)\setminus\{y_{1},y_{2}\}$, we have
a sequence of closed balls $B_{1},\ldots,B_{N}$ so that $B_{1}\cup\cdots\cup B_{N}$
is connected, $y_{1},y_{2}\in B_{1}\cup\cdots\cup B_{N}$, $B_{1}\cup\cdots\cup B_{N}\subset F(B)$,
and
\begin{equation}
\sup_{x\in B_{1}\cup\cdots\cup B_{N}}\dist_{\Euc}(x,X)\le R'.\label{eq:ballsnottoofar}
\end{equation}
Moreover, we can take
\begin{equation}
N(\mathbb{U},\mathbb{V},\mathbb{V}',F,R,R')=N(\alpha\mathbb{U},\alpha\mathbb{V},\mathbb{V}',x\mapsto F(x/\alpha),\alpha R,R')=N(\mathbb{U},\mathbb{V},\alpha\mathbb{V}',x\mapsto\alpha F(x),R,\alpha R')\label{eq:Nscale}
\end{equation}
for any $\alpha>0$.
\end{prop}

The statement of \propref{conformalballs} has an intuitive interpretation.
It says that a given conformal homeomorphism $F$, on a compact set
inside of its domain, cannot disturb the geometry of a ball \textbf{$B$}
so violently that an unbounded number of balls inside of $F(B)$ are
required to connect two points on the boundary of $F(B)$. This is
essentially true because the derivatives of the conformal homeomorphism
and its inverse must be bounded on a compact set in the interior of
the domain. Quantifying this intuition, however, requires a bit of
work. We first prove the following lemma.
\begin{lem}
\label{lem:curvaturecondition}Suppose that $\xi:[-1,1]\to\mathbf{R}^{2}$
is a smooth curve and $x\in\mathbf{R}^{2}$ is such that $x-\xi(0)$
is perpendicular to $\xi'(0)$. If
\begin{equation}
|x-\xi(0)|<\frac{|\xi'(0)|^{2}}{2\|\xi''\|_{\infty}},\label{eq:xxi0}
\end{equation}
then
\[
|x-\xi(s)|\ge|x-\xi(0)|
\]
for all $s$ satisfying
\begin{equation}
|s|\le\frac{|\xi'(0)|}{2\|\xi''\|_{\infty}}.\label{eq:tcond}
\end{equation}
\end{lem}

\begin{proof}
Put $v=|\xi'(0)|$ and $d=|x-\xi(0)|$. We may change coordinates
to assume that $\xi(0)=0$, $\xi'(0)=(v,0)$, and $x=(0,d)$. Then
we have that
\begin{align*}
|x-\xi(s)|^{2} & =\xi_{1}(s)^{2}+(d-\xi_{2}(s))^{2}=\left(vs+\int_{0}^{s}(s-r)\xi_{1}''(r)\,\dif r\right)^{2}+\left(d-\int_{0}^{s}(s-r)\xi_{2}''(r)\,\dif r\right)^{2}\\
 & \ge d^{2}+v^{2}s^{2}+2\int_{0}^{s}(s-r)\left[vs\xi_{1}''(r)-d\xi_{2}''(r)\right]\,\dif r\ge d^{2}+s^{2}\left(v^{2}-(vs+d)\sup_{|r|\le s}|\xi''(r)|\right).
\end{align*}
We note that
\[
(vs+d)\sup_{|r|\le s}|\xi''(r)|\le(vs+d)\|\xi''\|_{\infty}<v^{2},
\]
where the last inequality is by \eqref{xxi0} and \eqref{tcond},
so in fact we have $|x-\xi(s)|^{2}>|x-\xi(0)|^{2}$, as claimed.
\end{proof}
Now we can prove \propref{conformalballs}.
\begin{proof}[Proof of \propref{conformalballs}.]
Let $x$ and $r$ be the center and radius of $B$, respectively,
so $B=\overline{B(x,r)}$. Parameterize $\partial B$ by
\[
\xi(t)=x+r(\cos t,\sin t).
\]
Then $F\circ\xi$ is a parameterization of $\partial F(B)$. We have,
using the fact that $F$ is conformal, that
\begin{equation}
|(F\circ\xi)'(t)|=|F'(\xi(t))\cdot\xi'(t)|\in[R\|(F')^{-1}\|_{L^{\infty}(\mathbb{U})}^{-1},r\|F'\|_{L^{\infty}(\mathbb{U})}]\label{eq:1derivbd}
\end{equation}
and
\begin{align}
|(F\circ\xi)''(t)|=|F''(\xi(t))\cdot|\xi'(t)|^{2}+F'(\xi(t))\cdot\xi''(t)| & \le r^{2}\|F''\|_{L^{\infty}(\mathbb{U})}+r\|F'\|_{L^{\infty}(\mathbb{U})}\nonumber \\
 & \le rR\|F''\|_{L^{\infty}(\mathbb{U})}+r\|F'\|_{L^{\infty}(\mathbb{U})}.\label{eq:2derivub}
\end{align}
Let
\begin{align*}
t_{0} & =\frac{\|(F')^{-1}\|_{L^{\infty}(\mathbb{U})}^{-1}}{R\|F''\|_{L^{\infty}(\mathbb{U})}+\|F'\|_{L^{\infty}(\mathbb{U})}}, & \ell_{0} & =\frac{r\|(F')^{-1}\|_{L^{\infty}(\mathbb{U})}^{-2}}{R\|F''\|_{L^{\infty}(\mathbb{U})}+\|F'\|_{L^{\infty}(\mathbb{U})}}.
\end{align*}
Define, for $a>0$ to be fixed later,
\begin{equation}
\omega_{a}(t)=F(\xi(t))+a\frac{\mathcal{R}(F\circ\xi)'(t)}{r\|F'\|_{L^{\infty}(\mathbb{U})}},\label{eq:omegadef}
\end{equation}
where $\mathcal{R}$ denotes counterclockwise rotation by $\pi/2$.
Let
\begin{equation}
b(t)=|\omega_{a}(t)-F(\xi(t))|=\left|a\frac{\mathcal{R}(F\circ\xi)'(t)}{r\|F'\|_{L^{\infty}(\mathbb{U})}}\right|=a\frac{|(F\circ\xi)'(t)|}{r\|F'\|_{L^{\infty}(\mathbb{U})}}.\label{eq:bdef}
\end{equation}
Fix $t\in[0,2\pi]$. By \eqref{1derivbd}, we have that 
\begin{equation}
b(t)\le a.\label{eq:bta}
\end{equation}
Thus, by \lemref{curvaturecondition} (applied with $x=\omega_{a}(t)$,
$\xi(s)=F(\xi(s-t))$), \eqref{1derivbd}, and \eqref{2derivub},
if $a\le\ell_{0}$ and $s\in[t-t_{0},t+t_{0}]$ then
\begin{equation}
|F(\xi(s))-\omega_{a}(t)|\ge b(t).\label{eq:localbd}
\end{equation}
Moreover, we know that if $x\in\partial B\setminus\xi([t-t_{0},t+t_{0}])$,
then
\begin{equation}
|F(x)-F(\xi(t))|\ge cr\|(F')^{-1}\|_{L^{\infty}(\mathbb{U})}^{-1}t_{0}\label{eq:FxFxi}
\end{equation}
for some absolute constant $c$. Let $a=ra_{0}$, where
\begin{equation}
a_{0}=\min\left\{ r^{-1}\ell_{0},\frac{1}{2}c\|(F')^{-1}\|_{L^{\infty}(\mathbb{U})}^{-1}t_{0},\frac{R'}{2r}\right\} .\label{eq:a0def}
\end{equation}
Then we have, for $s\not\in[t-t_{0},t+t_{0}]+2\pi\mathbf{Z}$,
\begin{equation}
|F(\xi(s))-\omega_{a}(t)|\ge|F(\xi(s))-F(\xi(t))|-|F(\xi(t))-\omega_{a}(t)|\ge cr\|(F')^{-1}\|_{L^{\infty}(\mathbb{U})}^{-1}t_{0}-a\ge a\ge b(t),\label{eq:largescalebound}
\end{equation}
where the first inequality is the triangle inequality, the second
is by \eqref{FxFxi} and \eqref{bta}, the third is by \eqref{a0def},
and the fourth is by \eqref{bta}. Together, \eqref{localbd} and
\eqref{largescalebound} mean that \eqref{localbd} hold for \emph{all}
$s,t$. This implies that
\begin{equation}
\overline{B(\omega_{a}(s),b(s))}\subseteq\overline{F(B)}\label{eq:Bsingleok}
\end{equation}
for all $s$. Note that $a_{0}$ only depends on $\mathbb{U},\mathbb{V},F,R,R'$
and not on $B$, and moreover is invariant under the scalings indicated
in \eqref{Nscale}.

We can also apply the triangle inequality to \eqref{omegadef} to
obtain, for all $s,s'$,
\begin{align}
|\omega_{a}(s)-\omega_{a}(s')| & \le|F(\xi(s))-F(\xi(s'))|+\frac{a}{r\|F'\|_{L^{\infty}(\mathbb{U})}}|(F\circ\xi)'(s)-(F\circ\xi)'(s')|\nonumber \\
 & \le r\|F'\|_{L^{\infty}(\mathbb{U})}|s-s'|+\frac{a_{0}}{\|F'\|_{L^{\infty}(\mathbb{U})}}\left(rR\|F''\|_{L^{\infty}(\mathbb{U})}+r\|F'\|_{L^{\infty}(\mathbb{U})}\right)|s-s'|\nonumber \\
 & =r\left(\|F'\|_{L^{\infty}(\mathbb{U})}+a_{0}\left(R\|F''\|_{L^{\infty}(\mathbb{U})}\|F'\|_{L^{\infty}(\mathbb{U})}^{-1}+1\right)\right)|s-s'|,\label{eq:diffub}
\end{align}
where in the second inequality we used \eqref{1derivbd} and \eqref{2derivub}
. On the other hand, we can compute from \eqref{bdef} and \eqref{1derivbd}
that
\begin{equation}
b(t)=a\frac{|(F\circ\xi)'(t)|}{r\|F'\|_{L^{\infty}(\mathbb{U})}}\ge\frac{a_{0}r}{\|(F')^{-1}\|_{L^{\infty}(\mathbb{U})}\|F'\|_{L^{\infty}(\mathbb{U})}}.\label{eq:blb}
\end{equation}
Therefore, if we define
\[
\iota=\frac{2a_{0}\left(\|(F')^{-1}\|_{L^{\infty}(\mathbb{U})}\|F'\|_{L^{\infty}(\mathbb{U})}\right)^{-1}}{\|F'\|_{L^{\infty}(\mathbb{U})}+a_{0}\left(R\|F''\|_{L^{\infty}(\mathbb{U})}\|F'\|_{L^{\infty}(\mathbb{U})}^{-1}+1\right)},
\]
then whenever $|s-s'|\le\iota$, we have
\[
|\omega_{a}(s)-\omega_{a}(s')|\le b(s)+b(s')
\]
by \eqref{diffub} and \eqref{blb}. In particular, this means that
\begin{equation}
\overline{B(\omega_{a}(s),b(s))}\cap\overline{B(\omega_{a}(s'),b(s'))}\ne\emptyset.\label{eq:Bsintersect}
\end{equation}
Note that, like $a_{0}$, the constant $\iota$ depends only on $\mathbb{U},\mathbb{V},F,R,R'$
and not on $B$, and is invariant under the scalings indicated in
\eqref{Nscale}. Now if $s_{0},s_{1}\in\mathbf{R}$ with $s_{1}\in(s_{0},s_{0}+2\pi)$,
then we can choose a sequence of points
\[
s_{0}=t_{0}<t_{1}<t_{2}<\cdots<t_{N}=s_{1}
\]
so that $t_{i}-t_{i-1}<\iota$ for each $i$, and $N\le2\pi/\iota$,
which is invariant under the scalings indicated in \eqref{Nscale}
because $\iota$ is. Then we claim that
\[
\{\overline{B(\omega_{a}(t_{i}),b(t_{i}))}\;:\;i\in\{0,\ldots,N\}\}
\]
is a set of balls as claimed in the lemma if $s_{1}$ and $s_{2}$
are chosen so that $\{y_{1},y_{2}\}=\{\xi(s_{1}),\xi(s_{2})\}$ and
$X=\xi((s_{1},s_{2}))$. (See \figref{putinballs}.)
\begin{figure}
\centering{}\includegraphics{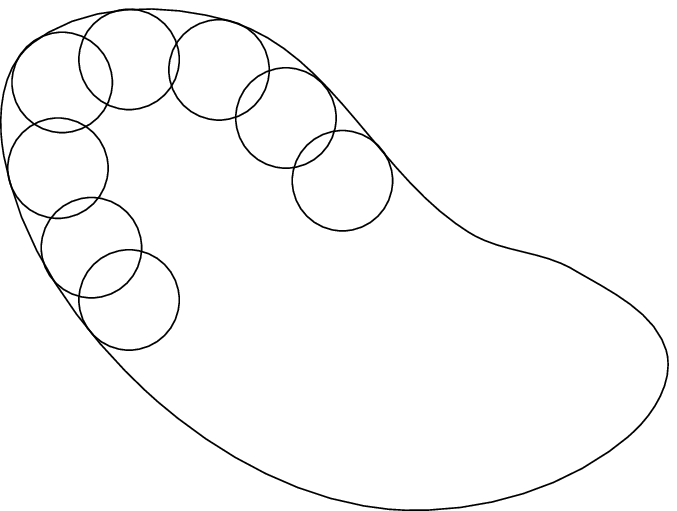}\caption{Connecting two points on the boundary of a deformed ball $F(B)$ with
smaller Euclidean balls along the boundary. The analysis in the proof
of \propref{conformalballs} essentially amounts to ensuring that
each smaller ball has sufficiently small radius of curvature to ``locally''
fit inside of $F(B)$, and sufficiently small diameter to ``globally''
fit inside of $F(B)$.\label{fig:putinballs}}
\end{figure}
 The fact that the balls all lie inside $\overline{F(B)}$ is \eqref{Bsingleok},
the fact that their union is connected is \eqref{Bsintersect}, and
the fact that they have radius at most $R'$, as well as \eqref{ballsnottoofar},
comes from \eqref{bdef} and \eqref{a0def}.
\end{proof}
\begin{cor}
\label{cor:conformalcomparison}Suppose that $\mathbb{U}$ and $\mathbb{V}$
are bounded open subsets of $\mathbf{R}^{2}$ so that $\overline{\mathbb{U}}\subset\mathbb{V}$.
Suppose further that $\mathbb{V}'$ is another domain and $F:\mathbb{V}\to\mathbb{V}'$
is a conformal homeomorphism. Define
\begin{align*}
R & \in(0,\tfrac{1}{4}\dist_{\Euc}(\mathbb{U},\partial\mathbb{V})),\\
R' & \in(0,\tfrac{1}{4}\dist_{\Euc}(F(\mathbb{U}),\partial\mathbb{V}')).
\end{align*}
Suppose further that $\partial B(x,r)\cap\mathbb{U}$ is connected
for every $x\in\mathbf{R}^{2}$, $r\in(0,R)$. Then, for all $x,y\in\mathbb{U}$,
we have
\[
d_{\mu_{h_{\mathbb{V}}}\circ F^{-1},F(\mathbb{U})^{(R')},\delta,R'}(F(x),F(y))\le Nd_{\mathbb{V},\mathbb{U},\delta,R}(x,y),
\]
where
\[
N=N(\mathbb{U},\mathbb{V},\mathbb{V}',F,R,R')
\]
is as in the statement of \propref{conformalballs}.
\end{cor}

\begin{proof}
Let $d=d_{\mathbb{V},\mathbb{U},\delta,R}(x,y)$. There is a path
$\pi\in\mathcal{P}_{\mathbb{U}}(x,y)$ and a sequence of closed balls
$B_{1},\ldots,B_{n}\subset\mathbb{U}^{(2R)}$ so that
\[
d=\sum_{j=1}^{n}\kappa_{\delta}(\mu_{h_{\mathbb{V}}}(B_{i}))
\]
and $\pi\subset B_{1}\cup\cdots\cup B_{d}$. Now by \propref{conformalballs},
for each $1\le j\le d$ there is a sequence of closed balls $B_{j;1},\ldots,B_{j;N}\subseteq F(B_{j})$,
where $N$ depends only on $\mathbb{U},\mathbb{V},\mathbb{V}',F,R,R'$,
so that if $P_{j}=B_{j;1}\cup\cdots\cup B_{j;N}$ and $P=P_{1}\cup\cdots\cup P_{d}$,
then $P$ is connected and $F(x),F(y)\in P$. Moreover, $B_{j;1},\ldots,B_{j;N}$
can be chosen to lie in $F(\mathbb{U})^{(R')}$ by \eqref{ballsnottoofar}
and the assumption that $\partial B_{j}\cap\mathbb{U}$ is connected.
In particular, we have that
\[
\mu_{h_{\mathbb{V}}}(F^{-1}(B_{j;i}))\le\mu_{h_{\mathbb{V}}}(B_{j})\le\delta.
\]
Therefore, by \lemref{kappascale}, we have that
\[
d_{\mu_{h_{\mathbb{V}}}\circ F^{-1},F(\mathbb{U})^{(R')},\delta,R}(F(x),F(y))\le Nd_{\mathbb{V},\mathbb{U},\delta,R}(x,y),
\]
as claimed.
\end{proof}

\subsection{The RSW result\label{subsec:the-rsw-proof}}

Now we can prove our RSW result. In this section, in order to access
the conformal mappings that we will need to make the argument work,
it will be more convenient to work with ellipses rather than rectangles.
Define $\mathbb{E}(S_{1},S_{2})$ to be the filled, closed ellipse
centered at $0$ with horizontal axis of length $S_{1}$ and vertical
axis of length $S_{2}$; that is,
\[
\mathbb{E}(S_{1},S_{2})=\{(x_{1},x_{2})\in\mathbf{R}^{2}\;:\;(2x_{1}/S_{1})^{2}+(2x_{2}/S_{2})^{2}\le1\}.
\]

\begin{thm}
\label{thm:RSW}Let $\mathbb{R}_{1}=\mathbb{B}(1,2)$ and $\mathbb{R}_{2}=\mathbb{B}(2,1)$.
There is a constant $c>0$ so that, for any $w>0$, we have
\begin{equation}
\mathbf{P}\left(d_{\mathbb{R}_{2},\delta}(\mathrm{L},\mathrm{R})\le w\right)\ge c\mathbf{P}\left(d_{\mathbb{R}_{1},\delta}(\mathrm{L},\mathrm{R})\le cw\right).\label{eq:RSW}
\end{equation}
Thus we have a constant $C_{\mathrm{RSW}}<\infty$ so that for any
$p\in(0,C_{\mathrm{RSW}})$,
\begin{equation}
\Theta_{\mathbb{R}_{1},\delta,R}^{\mathrm{hard}}(p)\le C_{\mathrm{RSW}}\Theta_{\mathbb{R}_{1},\delta,R}^{\mathrm{easy}}(C_{\mathrm{RSW}}p).\label{eq:RSW-quantile}
\end{equation}
\end{thm}

\begin{proof}
Define
\[
q_{0}=\mathbf{P}\left[d_{\mathbb{R}_{1},\delta}(\mathrm{L},\mathrm{R})\le w\right].
\]
Let $\mathbb{E}_{1}=\mathbb{E}(1,3)$ and $\mathbb{E}_{2}=\mathbb{E}(3,3/4)$.
Let $X$ and $Y$ be the left and right, respectively, connected components
of $\partial\mathbb{E}_{1}\cap\mathbb{R}_{1}$, and let $x_{0}$ and
$x_{1}$ be the bottom and top, respectively, endpoints of $X$. By
the Riemann mapping theorem and Schwarz reflection, for any $y\in Y$
there are open sets $\mathbb{F}_{1},\mathbb{F}_{2}(y)$ and a conformal
homeomorphism $F_{y}:\mathbb{F}_{1}\to\mathbb{F}_{2}(y)$ such that
the following properties hold:
\begin{enumerate}
\item We have $\mathbb{E}_{1}\subset\mathbb{F}_{1}\subset\mathbb{R}_{1}^{*}$,
$\mathbb{E}_{2}\subset\mathbb{F}_{2}(y)\subset\mathbb{R}_{2}^{*}$,
and $F_{y}(\mathbb{E}_{1})=\mathbb{E}_{2}$.
\item \label{enu:Fx0x1}$F_{y}(x_{0})$ and $F_{y}(x_{1})$ are the upper-
and lower-left, respectively, points of $\partial\mathbb{E}_{2}\cap\partial\mathbb{R}_{2}$.
\item $F_{y}(y)$ is the lower-right point of $\partial\mathbb{E}_{2}\cap\partial\mathbb{R}_{2}$.
\item We have that
\begin{equation}
Q\coloneqq\sup_{y\in Y}\max\left\{ \|F_{y}'\|_{L^{\infty}(\mathbb{E}_{1})},\|(F_{y}')^{-1}\|_{L^{\infty}(\mathbb{E}_{2})}\right\} <\infty.\label{eq:Qdef}
\end{equation}
(Here, $(F'_{y})^{-1}(x)=(F'_{y}(x))^{-1}$.)
\end{enumerate}
This construction was used in the proof of \cite[Theorem 3.1]{DF18}.
We note that condition \enuref{Fx0x1} implies that $F_{y}(X)$ is
the left connected component of $\overline{\mathbb{E}_{2}\setminus\mathbb{R}_{2}}$.
See \figref{RSWsetup} for a partial illustration of this setup.
\begin{figure}
\centering{}\begin{tikzpicture}[x=1in,y=1in,label distance=-0.07in,remember picture]

% A path that follows the edges of the current page
\tikzstyle{reverseclip}=[insert path={(current page.north east) --
  (current page.south east) --
  (current page.south west) --
  (current page.north west) --
  (current page.north east)}
]

\draw [black, name path=R1] (0,0) rectangle (1,2);
\filldraw [black, fill=black, fill opacity=0.05, name path=E1] (0.5,1) ellipse [x radius = 0.5, y radius = 1.5];
\filldraw [black, fill=black, fill opacity=0.05, name path=F1] (0.5,1) ellipse [x radius = 0.6, y radius = 1.6];
\begin{scope}
	\begin{pgfinterruptboundingbox}
	\clip (-1,0) rectangle (3,2);
	\end{pgfinterruptboundingbox}
	\draw [blue, ultra thick] (0.5,1) ellipse [x radius = 0.5, y radius = 1.5];
\end{scope}
\fill[red] ($(0.5,1)+(330:0.5 and 1.5)$) circle (2pt) node [label=left:{$y$}] {};
\fill[red,name intersections={of=R1 and E1}]
    (intersection-3) circle (2pt) node [label = below right:{$x_0$}] {}
    %(intersection-2) circle (2pt) node [label = below left:{$y_0$}] {}
    (intersection-5) circle (2pt) node [label = above right:{$x_1$}] {};
    %(intersection-6) circle (2pt) node [label = above left:{$y_1$}] {};
\draw[blue,name intersections={of=R1 and E1}]
    (intersection-1) node [label = right:{$X$}] {}
    (intersection-4) node [label = left:{$Y$}] {};

\draw ($(0.5,1)+(280:0.5 and 1.5)$) node [label = above left:{$\mathbb{E}_1$}] {};
\draw ($(0.5,1)+(310:0.6 and 1.6)$) node [label = below right:{$\mathbb{F}_1$}] {};
\draw[->,line width=2pt] (1.6,1) to (2.1,1);
\draw (1.85,0.75) node [label={$F_y$}] {};
\begin{scope}[shift = {(3,0.5)}]
\draw [black, name path=R2] (0,0) rectangle (2,1);
\filldraw [black,fill=black, fill opacity=0.05, name path=E2] (1,0.5) ellipse [x radius = 1.5, y radius = 0.375];
\begin{scope}
	\begin{pgfinterruptboundingbox}
		\path [clip] (0,-1) -- (4,-1) -- (4,2)-- (0,2) -- cycle [reverseclip];
	\end{pgfinterruptboundingbox}
	\draw [blue, ultra thick] (1,0.5) ellipse [x radius = 1.5, y radius = 0.375];
\end{scope}
\draw[blue] (-0.5,0.5) node [label = right:{$F_y(X)$}] {};
\fill[red,name intersections={of=R2 and E2}]
    (intersection-1) circle (2pt) node [label = below right:{$F_y(x_0)$}] {}
    (intersection-2) circle (2pt) node [label = above right:{$F_y(x_1)$}] {}
    (intersection-4) circle (2pt) node [label = above left:{$F_y(y)$}] {};

\end{scope}
\end{tikzpicture}\caption{Illustration of the geometrical setup for the proof of \thmref{RSW}.\label{fig:RSWsetup}}
\end{figure}
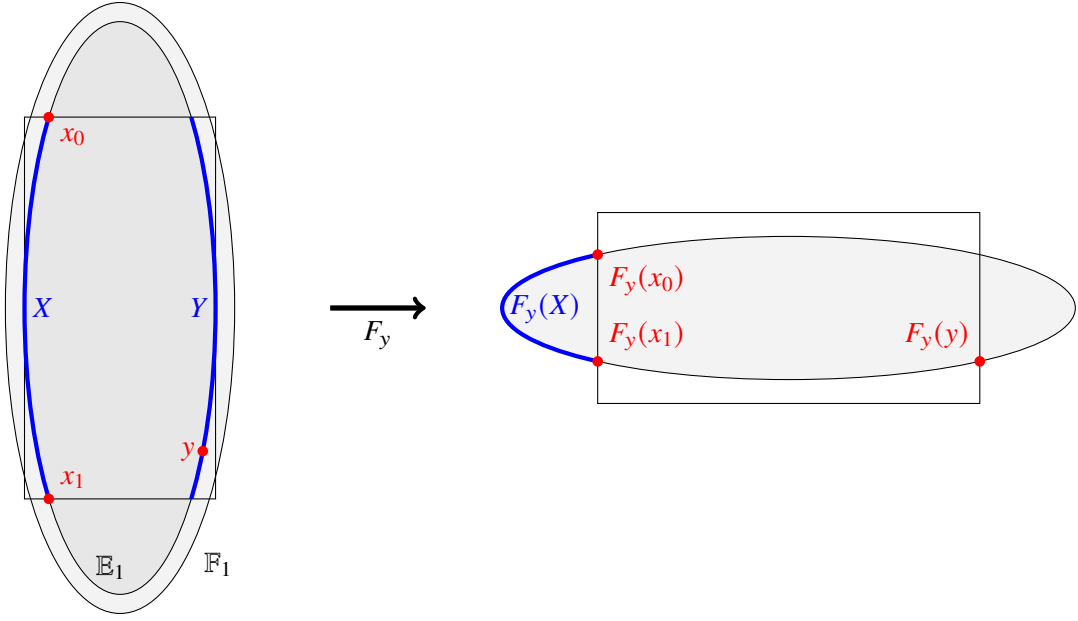

Let $R_{1}=\frac{1}{4}\dist(\mathbb{E}_{1},\partial\mathbb{F}_{1})$.
Now, on the event $E=\left\{ \max\limits _{z\in\mathbb{E}_{1}^{(R_{1})}}h_{\mathbb{R}_{1}^{*}:\mathbb{F}_{1}}(z)\le0\right\} $
we have that
\[
d_{\mathbb{F}_{1},\mathbb{E}_{1},\delta,R_{1}}(X,Y)\le d_{\mathbb{R}_{1},\delta,R_{1}}(X,Y)\le Cd_{\mathbb{R}_{1},\delta}(X,Y),
\]
where the second inequality is by \eqref{RprimeR-multiplicative}.
(Here we have folded the geometrical factor in \eqref{RprimeR-multiplicative}
into the constant $C$.) Now we note that $d_{\mathbb{F}_{1},\mathbb{E}_{1},\delta,R_{1}}(X,Y)$
and $E$ are independent. Let $q_{*}=\mathbf{P}[E]$. We note that
$q_{*}$ is strictly positive because $\mathbb{E}_{1}^{(R_{1})}$
is separated from $\partial\mathbb{F}_{1}$ by a positive Euclidean
distance and thus $h_{\mathbb{R}_{1}^{*}:\mathbb{F}_{1}}|_{\mathbb{E}_{1}^{(R_{1})}}$
is a uniformly smooth centered Gaussian process on a compact domain---thus
there is a positive probability that it is nonpositive on $\mathbb{E}_{1}^{(R_{1})}$.
Then we have that
\begin{equation}
\mathbf{P}\left(d_{\mathbb{F}_{1},\mathbb{E}_{1},\delta,R_{1}}(X,Y)\le Cw\right)\ge q_{*}q_{0}.\label{eq:reducetoellipses}
\end{equation}
Now write $Y$ as a union of disjoint curvilinear segments $Y_{1},\ldots,Y_{M}$
so that $\diam_{\Euc}(Y_{k})\le(2Q)^{-1}$, where $Q$ is as in \eqref{Qdef}
and $M$ is taken to satisfy $M\le3Q$. By \eqref{reducetoellipses},
we have that
\[
q_{*}q_{0}\le\sum_{k=1}^{M}\mathbf{P}(d_{\mathbb{F}_{1},\mathbb{E}_{1},\delta,R_{1}}(X,Y)\le Cw),
\]
so there is some $1\le k\le M$ so that
\begin{equation}
\mathbf{P}(d_{\mathbb{F}_{1},\mathbb{E}_{1},\delta,R_{1}}(X,Y_{k})\le Cw)\ge\frac{q_{*}q_{0}}{M}\ge\frac{q_{*}q_{0}}{3Q}.\label{eq:pickk}
\end{equation}
Fix $y$ to be the bottom endpoint of $Y_{k}$. By \eqref{Qdef},
we observe that $F_{y}(Y_{k})$ is contained in the right-hand connected
component of $\mathbb{E}_{2}\setminus\mathbb{R}_{2}$. Let
\[
R_{2}=\min\left\{ \frac{1}{4}\dist(\mathbb{E}_{2},\partial\mathbb{F}_{2}(y)),\frac{1}{32}\right\} 
\]
and define $\widetilde{\mathbb{E}}_{2}=\mathbb{E}_{2}^{(R_{2})}$.
By \corref{conformalcomparison}, we have an $N<\infty$ so that
\begin{equation}
d_{\mu_{h_{\mathbb{F}_{1}}}\circ F_{y}^{-1},\widetilde{\mathbb{E}}_{2},\delta,R_{2}}(F_{y}(X),F_{y}(Y_{k}))\le Nd_{\mathbb{F}_{1},\mathbb{E}_{1},\delta,R_{1}}(X,Y_{k}).\label{eq:apply-conformal-comparison}
\end{equation}
Note that, by \propref{conformalcovariance}, we have
\begin{equation}
d_{\mathbb{F}_{2}(y),\widetilde{\mathbb{E}}_{2},\delta,R_{2}}(F_{y}(X),F_{y}(Y_{k}))\overset{\mathrm{law}}{=}d_{\e^{-(2+\gamma^{2}/2)\log|(F_{y}^{-1})'|}\mu_{h_{\mathbb{F}_{1}}}\circ F_{y}^{-1},\widetilde{\mathbb{E}}_{2},\delta,R_{2}}(F_{y}(X),F_{y}(Y_{k})).\label{eq:law-the-same}
\end{equation}
On the other hand, we have
\[
d_{\mathbb{F}_{2}(y),\widetilde{\mathbb{E}}_{2},\delta,R_{2}}(F_{y}(X),F_{y}(Y_{k}))\ge d_{\mathbb{F}_{2}(y),\widetilde{\mathbb{E}}_{2},\delta,R_{2}}(F_{y}(X),F_{y}(Y_{k}))\ge d_{\mathbb{F}_{2}(y),\widetilde{\mathbb{E}}_{2},\delta,R_{2}}(\widetilde{\mathbb{E}}_{2}\cap\mathrm{L}_{\mathbb{R}_{2}},\widetilde{\mathbb{E}}_{2}\cap\mathrm{R}_{\mathbb{R}_{2}}),
\]
and on the event $E'=\left\{ \max\limits _{z\in\widetilde{\mathbb{E}}_{2}^{(R_{2})}}h_{\mathbb{R}_{2}^{*}:\mathbb{F}_{2}(y)}(z)\le0\right\} $
(which is independent of $d_{\mathbb{F}_{2}(y),\widetilde{\mathbb{E}}_{2},\delta,R_{2}}(F_{y}(X),F_{y}(Y_{k}))$)
we have that
\[
d_{\mathbb{F}_{2}(y),\widetilde{\mathbb{E}}_{2},\delta,R_{2}}(\widetilde{\mathbb{E}}_{2}\cap\mathrm{L}_{\mathbb{R}_{2}},\widetilde{\mathbb{E}}_{2}\cap\mathrm{R}_{\mathbb{R}_{2}})\ge d_{\mathbb{R}_{2}^{*},\widetilde{\mathbb{E}}_{2},\delta,R_{2}}(\widetilde{\mathbb{E}}_{2}\cap\mathrm{L}_{\mathbb{R}_{2}},\widetilde{\mathbb{E}}_{2}\cap\mathrm{R}_{\mathbb{R}_{2}}).
\]
This implies that, if we define $q_{**}=\mathbf{P}[E']$ (which again
is strictly positive since $\dist_{\Euc}(\widetilde{\mathbb{E}}_{2}^{(R_{2})},\partial\mathbb{F}_{2}(y))>0$),
then for any $v>0$,
\begin{equation}
\mathbf{P}\left(d_{\mathbb{R}_{2}^{*},\widetilde{\mathbb{E}}_{2},\delta,R_{2}}(\widetilde{\mathbb{E}}_{2}\cap\mathrm{L}_{\mathbb{R}_{2}},\widetilde{\mathbb{E}}_{2}\cap\mathrm{R}_{\mathbb{R}_{2}})\le v\right)\ge q_{**}\mathbf{P}\left(d_{\mathbb{F}_{2}(y),\widetilde{\mathbb{E}}_{2},\delta,R_{2}}(F_{y}(X),F_{y}(Y_{k}))\le v\right).\label{eq:F2to3R}
\end{equation}
We can also write, using \eqref{box-comparison} and then \eqref{R-comparison},
\begin{equation}
d_{\mathbb{R}_{2}^{*},\widetilde{\mathbb{E}}_{2},\delta,R_{2}}(\widetilde{\mathbb{E}}_{2}\cap\mathrm{L}_{\mathbb{R}_{2}},\widetilde{\mathbb{E}}_{2}\cap\mathrm{R}_{\mathbb{R}_{2}})\ge d_{\mathbb{R}_{2},\delta,R_{2}}(\widetilde{\mathbb{E}}_{2}\cap\mathrm{L}_{\mathbb{R}_{2}},\widetilde{\mathbb{E}}_{2}\cap\mathrm{R}_{\mathbb{R}_{2}})\ge d_{\mathbb{R}_{2},\delta,R_{2}}(\mathrm{L},\mathrm{R})\ge d_{\mathbb{R}_{2},\delta}(\mathrm{L},\mathrm{R}).\label{eq:final-chain}
\end{equation}
Therefore, we have
\begin{align*}
\mathbf{P}( & d_{\mathbb{R}_{2},\delta}(\mathrm{L},\mathrm{R})\le CNw)\\
 & \ge\mathbf{P}\left(d_{\mathbb{R}_{2}^{*},\widetilde{\mathbb{E}}_{2},\delta,R_{2}}(\widetilde{\mathbb{E}}_{2}\cap\mathrm{L}_{\mathbb{R}_{2}},\widetilde{\mathbb{E}}_{2}\cap\mathrm{R}_{\mathbb{R}_{2}})\le CNw\right)\ge q_{**}\mathbf{P}\left(d_{\mathbb{F}_{2}(y),\widetilde{\mathbb{E}}_{2},\delta,R_{2}}(F_{y}(X),F_{y}(Y_{k}))\le CNw\right)\\
 & \ge q_{**}\mathbf{P}\left(d_{\e^{-(2+\gamma^{2}/2)\log|(F_{y}^{-1})'|}\mu_{h_{\mathbb{F}_{1}}}\circ F_{y}^{-1},\widetilde{\mathbb{E}}_{2},\delta,R_{2}}(F_{y}(X),F_{y}(Y_{k}))\le CNw\right)\ge q_{**}\mathbf{P}\left(d_{\mathbb{F}_{1},\mathbb{E}_{1},\delta,R_{1}}(X,Y_{k})\le Cw\right)\\
 & \ge\frac{q_{**}q_{*}q_{0}}{3Q},
\end{align*}
where the first inequality is by \eqref{final-chain}, the second
is by \eqref{F2to3R}, the third is by \eqref{law-the-same}, the
fourth is by \eqref{apply-conformal-comparison} and \eqref{mostusefulscaling}
(applied with $\alpha=\e^{(2+\gamma^{2}/2)\log|(F_{y}^{-1})'|}\vee1$,
which we fold into the constant $C$ as it is bounded by \eqref{Qdef}),
and the last is by \eqref{pickk}. This implies \eqref{RSW} with
appropriately chosen constants.
\end{proof}

\section{Percolation arguments\label{sec:Percolation}}

In this section we prove concentration for crossing distances at a
given scale, as well as relationships between crossing distances at
different scales, using percolation-type arguments. Recall the definition
\eqref{dminadef} of $d_{\cdots}(\min;a).$
\begin{prop}
\label{prop:lowerbound-tail}There is a $p_{0}>0$ so that for every
$\theta>2$, there is a constant $C<\infty$ so that the following
holds. For any $a>0$, any box $\mathbb{R}$ with $\AR(\mathbb{R})\in[1/3,3]$,
any $R>0$, and any $K\in[1,\infty)$, if we define
\begin{align}
S & =\diam_{\Euc}(\mathbb{R}),\nonumber \\
\omega & =\frac{2\gamma\theta}{4+\gamma^{2}},\label{eq:ouromega}
\end{align}
then we have that
\begin{equation}
\mathbf{P}\left(d_{\mathbb{R},\delta,R}(\min;a)\le\Theta_{\mathbb{B}(K^{-1-\omega}S),\delta,K^{-\omega}R}^{\mathrm{easy}}(p_{0})\right)\le C\left(K^{2-\theta^{2}/2}+\e^{-aK/C}\right).\label{eq:lowerbound-tail-nogrowth}
\end{equation}
If we further assume that $R\le K^{-1}S$, then we have that
\begin{equation}
\mathbf{P}\left(d_{\mathbb{R},\delta,R}(\min;a)\le\frac{aK}{C}\Theta_{\mathbb{B}(K^{-1-\omega}),\delta,K^{-\omega}R}^{\mathrm{easy}}(p_{0})\right)\le C\left(K^{2-\theta^{2}/2}+\e^{-aK/C}\right).\label{eq:lowerbound-tail-growth}
\end{equation}
\end{prop}

We will prove \propref{lowerbound-tail} in \subsecref{lowerbounds}.
The constant $p_{0}$ will remain fixed throughout the remainder of
the paper. Note that if $a$ is treated as a fixed constant, as it
often will be in the sequel, then the second terms of the right-hand
sides of \eqref{lowerbound-tail-nogrowth} and \eqref{lowerbound-tail-growth}
can be ignored.
\begin{prop}
\label{prop:hardconcentration}We have constants $p_{1}<1$ and $C<\infty$
so that the following holds. Let $\mathbb{R}$ be a box with $\AR(\mathbb{R})\in[1/3,3]$,
let $S=\diam_{\Euc}(\mathbb{R})$, and let $K\in[1,\infty)$. Then
we have, with $\eta$ as in \eqref{etadef}, that
\begin{equation}
\mathbf{P}\left(d_{\mathbb{R},\delta,R}(\mathrm{L},\mathrm{R})\ge\e^{K^{7/8}}\Theta_{\mathbb{B}(K^{\eta-1}S),\delta,K^{\eta}R}^{\mathrm{hard}}(p_{1})\right)\le C\e^{-K/C}.\label{eq:hardconcentration-tail}
\end{equation}
Moreover, for every $p>0$ there is a $K_{0}<\infty$ so that if $K\ge K_{0}$,
then for every $\mathbb{R}$ with $\AR(\mathbb{R})\in[1/3,3]$, and
setting $S=\diam_{\Euc}(\mathbb{R})$, we have
\begin{equation}
\mathbf{P}\left(d_{\mathbb{R},\delta,R}(\mathrm{L},\mathrm{R})\ge K^{C}\Theta_{\mathbb{B}(K^{\eta-1}S),\delta,K^{\eta}R}^{\mathrm{hard}}(p_{1})\right)\le p.\label{eq:quantileub}
\end{equation}
 
\end{prop}

We will prove \propref{hardconcentration} in \subsecref{upperbounds}.
Like $p_{0}$, the constant $p_{1}$ will remain fixed throughout
the remainder of the paper.

Before establishing some important applications and consequences of
\propref[s]{lowerbound-tail}~and~\ref{prop:hardconcentration},
we will set up some notation that we use to express the ratios of
different quantiles to each other. Then, in \subsecref[s]{quantile-relationships}~and~\ref{subsec:tailsandmoments},
we will use the notation to express bounds on quantiles, moments,
and tails of certain crossing distances, using \thmref{RSW} and \propref[s]{lowerbound-tail}~and~\ref{prop:hardconcentration}.
Later, in \secref{inductiveargument}, we will show that these ratios
are not too large by induction on the scale.
\begin{defn}
\label{def:chidefs}For $U\in(0,\infty)$, define
\[
\chi_{U}=\frac{\Theta_{\mathbb{B}(U)}^{\mathrm{hard}}(p_{1})}{\Theta_{\mathbb{B}(U)}^{\mathrm{easy}}(p_{0})}
\]
and
\[
\overline{\chi}_{U}=\sup_{V\in(0,U]}\chi_{V}.
\]
Further define $\mathcal{S}_{\chi}=\{U>0\;:\;\chi_{U}\le\chi\}$.
\end{defn}

By \thmref{RSW} and \lemref{relatequantiles}, we have a constant
$C<\infty$ so that
\begin{equation}
\chi_{U}\le C\exp\left\{ C\sqrt{\Var(\log d_{\mathbb{B}(U)}(\mathrm{hard}))}\right\} .\label{eq:chiUvarbound}
\end{equation}
Thus, our strategy to bound $\chi_{U}$ will be to inductively bound
$\Var(\log d_{\mathbb{B}(U)}(\mathrm{hard}))$. Also, it is proved
in \cite[(74)]{DZZ18} that there is a constant $C$ so that
\begin{equation}
\overline{\chi}_{U}\le C\e^{(\log(U+1))^{0.95}}\label{eq:chiUaprioribound}
\end{equation}
for all $U$. This will be important for the base case of our induction.

\subsection{Quantile relationships\label{subsec:quantile-relationships}}

In this section, we show how \subsecref[s]{quantile-relationships}~and~\ref{subsec:tailsandmoments}
allow us to prove upper and lower bounds on ratios between quantiles
of crossing distances at different scales.

\subsubsection{Easy crossing quantile ratios}

Here we prove the following proposition establishing at least power-law
growth of easy crossing quantiles.
\begin{prop}
\label{prop:pllb-inductive}There are constants $T_{0}\ge0$ and $c>0$
so that for every $S\ge T_{0}$ and $K\ge1$, we have
\begin{equation}
\Theta_{\mathbb{B}(KS)}^{\mathrm{easy}}(p_{0})\ge cK^{c}\Theta_{\mathbb{B}(S)}^{\mathrm{easy}}(p_{0}).\label{eq:pllb-inductive}
\end{equation}
\end{prop}

To prove \propref{pllb-inductive}, we will require the following
lemma.
\begin{lem}
\label{lem:pllb}We have constants $c>0$ and $C<\infty$ so that
for any $M\ge1$ and any $S>0$, we have that
\begin{equation}
\Theta_{\mathbb{B}(MS)}^{\mathrm{easy}}(p_{0})\ge cM^{1/2}\Theta_{\mathbb{B}(S)}^{\mathrm{easy}}(p_{0})-CM^{3/2}.\label{eq:pllb}
\end{equation}
\end{lem}

\begin{proof}
Choose $\theta=\frac{4+\gamma^{2}}{2\gamma}>2$. Note that with this
choice of $\theta$ we have $\omega=1$ in the statement of \propref{lowerbound-tail}.
If $K$ is large enough so that $CK^{2-\theta^{2}/2}\le p_{0}$, then
we have that
\begin{equation}
\Theta_{\mathbb{B}(K^{2}S),1,\sqrt{5}KS}^{\mathrm{easy}}(p_{0})\ge\Theta_{\mathbb{B}(K^{2}S),1,\sqrt{5}KS}^{\mathrm{easy}}(CK^{2-\theta^{2}/2})\ge\frac{K}{C}\Theta_{\mathbb{B}(S),1,\sqrt{5}S}^{\mathrm{easy}}(p_{0})=\frac{K}{C}\Theta_{\mathbb{B}(S)}^{\mathrm{easy}}(p_{0})\label{eq:pllb-main}
\end{equation}
by \eqref{lowerbound-tail-growth} of \propref{lowerbound-tail}.
On the other hand, by \eqref{RprimeR} of \lemref{RprimeR}, we have
a constant $C$ so that
\[
\Theta_{\mathbb{B}(K^{2}S),1,\sqrt{5}KS}^{\mathrm{easy}}(p_{0})\le\Theta_{\mathbb{B}(K^{2}S)}^{\mathrm{easy}}(p_{0})+CK^{3}.
\]
Combining this with \eqref{pllb-main}, we have
\[
\Theta_{\mathbb{B}(K^{2}S)}^{\mathrm{easy}}(p_{0})+CK^{3}\ge\frac{K}{C}\Theta_{\mathbb{B}(S)}^{\mathrm{easy}}(p_{0}),
\]
which proves \eqref{pllb} when we take $M=K^{2}$.
\end{proof}
Now we can prove \propref{pllb-inductive} by induction.
\begin{proof}[Proof of \propref{pllb-inductive}.]
By \lemref{pllb}, we can choose $M>1$ so that $\Theta_{\mathbb{B}(MS)}^{\mathrm{easy}}(p_{0})\ge2\Theta_{\mathbb{B}(MS)}^{\mathrm{easy}}(p_{0})-C$
for some constant $C<\infty$ (distinct, of course, from those in
\eqref{pllb}). By induction, we have 
\begin{equation}
\Theta_{\mathbb{B}(M^{k}S)}^{\mathrm{easy}}(p_{0})\ge2^{k}\Theta_{\mathbb{B}(S)}^{\mathrm{easy}}(p_{0})-\sum_{j=0}^{k-1}2^{j}C\ge2^{k}\left(\Theta_{\mathbb{B}(S)}^{\mathrm{easy}}(p_{0})-C\right).\label{eq:powersofMinductiveTheta}
\end{equation}
Now let $k=\lfloor\log_{M}K\rfloor$. Then we have, by \eqref{thetamonotone}
and \eqref{powersofMinductiveTheta}, that
\[
\Theta_{\mathbb{B}(KS)}^{\mathrm{easy}}(p_{0})\ge\Theta_{\mathbb{B}(M^{k}S)}^{\mathrm{easy}}(p_{0})\ge2^{k}\left(\Theta_{\mathbb{B}(S)}^{\mathrm{easy}}(p_{0})-C\right)\ge\frac{1}{2}K^{\log_{M}2}\left(\Theta_{\mathbb{B}(S)}^{\mathrm{easy}}(p_{0})-C\right).
\]
Now by \lemref{limtoinfty}, there is a $T_{0}$ so that if $S\ge T_{0}$,
then $\Theta_{\mathbb{B}(S)}^{\mathrm{easy}}(p_{0})\ge2C$. Therefore,
\[
\Theta_{\mathbb{B}(KS)}^{\mathrm{easy}}(p_{0})\ge\frac{1}{2}K^{\log_{M}2}\Theta_{\mathbb{B}(S)}^{\mathrm{easy}}(p_{0}),
\]
which is \eqref{pllb-inductive} with the appropriate choice of $c$.
\end{proof}

\subsubsection{Hard crossing quantile ratios}

Here we show that the hard-crossing quantiles can grow at most like
a power of the scale. We will in fact use only \eqref{hardtoeasy}
of the following proposition in the sequel, but we include \eqref{hardtohard}
as well for completeness.
\begin{prop}
\label{prop:hardquantile}For every $p\in(0,1)$ there is a $K_{1}<\infty$
and a constant $C<\infty$ so that if $K\ge K_{1}$, then we have
\begin{equation}
\Theta_{\mathbb{B}(KS)}^{\mathrm{hard}}(p)\le K^{C}\Theta_{\mathbb{B}(S)}^{\mathrm{hard}}(p_{1})\label{eq:hardtohard}
\end{equation}
and
\begin{equation}
\Theta_{\mathbb{B}(KS)}^{\mathrm{easy}}(p)\le K^{C}\Theta_{\mathbb{B}(S)}^{\mathrm{hard}}(p_{1}).\label{eq:hardtoeasy}
\end{equation}
\end{prop}

\begin{proof}
This follows from \eqref{quantileub} of \propref{hardconcentration},
taking $K=K^{\eta-1}$.
\end{proof}

\subsection{Tail and moment estimates\label{subsec:tailsandmoments}}

In this section, we use the results of the percolation arguments (\propref[s]{lowerbound-tail}~and~\ref{prop:hardconcentration})
to prove moment and tail estimates on crossing probabilities. Recall
the definition of $\chi_{U}$ and $\overline{\chi}_{U}$ from \defref{chidefs}.

\subsubsection{Moment estimates for easy crossings}
\begin{prop}
\label{prop:minmomentbound}For any $A,Q>0$ there exists a constant
$C=C(A,Q)<\infty$ so that if $a\in(0,1)$, $\mathbb{R}$ is a rectangle
with $\AR(\mathbb{R})\in[1/3,3]$, and $S=\diam_{\Euc}(\mathbb{R})$,
then we have
\begin{equation}
\mathbf{E}[d_{\mathbb{R}}(\min;a)^{-A}]\le C(1+a^{-C})\overline{\chi}_{Q^{-1}S}^{A}\Theta_{\mathbb{B}(S)}^{\mathrm{easy}}(p_{0})^{-A}.\label{eq:minmomentbound}
\end{equation}
\end{prop}

\begin{proof}
We first recall the simple formula
\[
\mathbf{E}[d_{\mathbb{R}}(\min;a)^{-A}]=\int_{0}^{\infty}\mathbf{P}[d_{\mathbb{R}}(\min;a)\le u^{-1/A}]\,\dif u.
\]
We have by \eqref{lowerbound-tail-nogrowth} of \propref{lowerbound-tail}
that, if we define 
\begin{equation}
\omega(\theta)=\frac{2\gamma\theta}{4+\gamma^{2}}\label{eq:omegathetadef}
\end{equation}
as in \eqref{ouromega}, then
\begin{equation}
\mathbf{P}\left(d_{\mathbb{R}}(\min;a)\le\Theta_{\mathbb{B}(K^{-1-\omega(\theta)}S),1,K^{-\omega(\theta)}S}^{\mathrm{easy}}(p_{0})\right)\le C(K^{2-\theta^{2}/2}+\e^{-aK/C}).\label{eq:applylb}
\end{equation}
But on the other hand, we have, as long as $K^{1+\omega(\theta)}\ge K_{1}$
from \propref{hardquantile},
\begin{align*}
\Theta_{\mathbb{B}(K^{-1-\omega(\theta)}S),1,K^{-\omega(\theta)}\sqrt{5}S}^{\mathrm{easy}}(p_{0})=\Theta_{\mathbb{B}(K^{-1-\omega(\theta)}S)}^{\mathrm{easy}}(p_{0}) & \ge\chi_{K^{-1-\omega(\theta)}S}^{-1}\Theta_{\mathbb{B}(K^{-1-\omega(\theta)}S)}^{\mathrm{hard}}(p_{1})\\
 & \ge CK^{-C(1+\omega(\theta))}\chi_{K^{-1-\omega(\theta)}S}^{-1}\Theta_{\mathbb{B}(S)}^{\mathrm{easy}}(p_{1}),
\end{align*}
where the equality is by \eqref{nopointbeinghuge}, the first inequality
is by the definition of $\chi_{K^{-1-\omega(\theta)}S}$, and the
second inequality is by \eqref{hardtoeasy} of \propref{hardquantile}.
Applying this to \eqref{applylb} yields
\[
\mathbf{P}\left(d_{\mathbb{R}}(\min;a)\le K^{-C(1+\omega(\theta))}\chi_{K^{-1-\omega(\theta)}S}^{-1}\Theta_{\mathbb{B}(S)}^{\mathrm{hard}}(p_{1})\right)\le C(K^{2-\theta^{2}/2}+\e^{-aK/C}).
\]
This implies that
\[
\mathbf{P}\left(\left(\frac{d_{\mathbb{R}}(\min;a)\chi_{K^{-1-\omega(\theta)}S}}{\Theta_{\mathbb{B}(S)}^{\mathrm{easy}}(p_{1})}\right)^{-A}\ge K^{AC(1+\omega(\theta))}\right)\le C(K^{2-\theta^{2}/2}+\e^{-aK/C}).
\]
Putting $u=K^{AC(1+\omega(\theta))}$ yields, for each $Q<\infty$,
a $u_{0}<\infty$ so that if $u\ge u_{0}$ then
\[
\mathbf{P}\left(\left(\frac{d_{\mathbb{R}}(\min;a)\overline{\chi}_{Q^{-1}S}}{\Theta_{\mathbb{B}(S)}^{\mathrm{easy}}(p_{1})}\right)^{-A}\ge u\right)\le C\left(u^{\frac{2-\theta^{2}/2}{AC(1+\omega(\theta))}}+\exp\left\{ -\frac{a}{C}u^{\frac{1}{AC(1+\omega(\theta))}}\right\} \right),
\]
and thus
\[
\mathbf{E}\left[\left(\frac{d_{\mathbb{R}}(\min;a)}{\Theta_{\mathbb{B}(S)}^{\mathrm{easy}}(p_{1})}\right)^{-A}\right]\le C\overline{\chi}_{Q^{-1}S}^{A}\left(u_{0}+\int_{u_{0}}^{\infty}1\vee u^{\frac{2-\theta^{2}/2}{AC(1+\omega(\theta))}}\,\dif u+\int_{0}^{\infty}\exp\left\{ -\frac{a}{C}u^{\frac{1}{AC(1+\omega(\theta))}}\right\} \,\dif u\right).
\]
Now $\theta$ can be chosen large enough so that $\frac{2-\theta^{2}/2}{AC(1+\omega(\theta))}<1$
(recalling the definition \eqref{omegathetadef} of $\omega(\theta)$)
and hence the first integral of the last expression is finite. On
the other hand, by a change of variables, the second integral is a
constant times $a^{-AC(1+\omega(\theta))}$. This implies \eqref{minmomentbound}.
\end{proof}

\subsubsection{Tail and moment estimates for hard crossings}

\begin{figure}
\centering{}\hfill{}\subfloat[\label{fig:annularfromhard}Creating an annular crossing from hard
crossings of rectangles.]{\begin{tikzpicture}[x=0.5in,y=0.5in]
\draw[thin] (0,0) rectangle (2,1);
\draw[thick, red,style={decorate,decoration={snake,amplitude=2,segment length=50}}] (0,0.4) -- (2,0.6);
\draw[thin] (1,0) rectangle (3,1);
\draw[thick, red,style={decorate,decoration={snake,amplitude=2,segment length=50}}] (1,0.2) -- (3,0.2);
\draw[thin] (1,-1) rectangle (2,1);
\draw[thick, red,style={decorate,decoration={snake,amplitude=2,segment length=50}}] (1.3,-1) -- (1.2,1);
\draw[thin] (2,0) rectangle (3,2);
\draw[thick, red,style={decorate,decoration={snake,amplitude=2,segment length=50}}] (2.8,0) -- (2.7,2);
\draw[thin] (2,1) rectangle (4,2);
\draw[thick, red,style={decorate,decoration={snake,amplitude=2,segment length=50}}] (2.6,1) -- (2.2,3);
\draw[thin] (2,1) rectangle (3,3);
\draw[thick, red,style={decorate,decoration={snake,amplitude=2,segment length=50}}] (2,1.4) -- (4,1.9);
\draw[thin] (1,2) rectangle (3,3);
\draw[thick, red,style={decorate,decoration={snake,amplitude=2,segment length=50}}] (1,2.5) -- (3,2.8);
\draw[thin] (0,2) rectangle (2,3);
\draw[thick, red,style={decorate,decoration={snake,amplitude=2,segment length=50}}] (0,2.2) -- (2,2.3);
\draw[thin] (-1,1) rectangle (1,2);
\draw[thick, red,style={decorate,decoration={snake,amplitude=2,segment length=50}}] (1.8,2) -- (1.7,4);
\draw[thin] (1,2) rectangle (2,4);
\draw[thick, red,style={decorate,decoration={snake,amplitude=2,segment length=50}}] (0.4,1) -- (0.5,3);
\draw[thin] (0,0) rectangle (1,2);
\draw[thick, red,style={decorate,decoration={snake,amplitude=2,segment length=50}}] (-1,1.3) -- (1,1.4);
\draw[thick, red,style={decorate,decoration={snake,amplitude=2,segment length=50}}] (0.7,0) -- (0.8,2);
\draw[ultra thick] (1,1) rectangle (2,2);
\draw[ultra thick] (0,0) rectangle (3,3);
\end{tikzpicture}

}\hfill{}\subfloat[\label{fig:annulitocrossing}Assembling annuli into a crossing.]{\begin{centering}
\tiny
\begin{tikzpicture}[x=0.5in,y=0.5in]
\draw[step=1,gray,thin] (0,0) grid (5,4);
\draw[thick, red,style={decorate,decoration={snake,amplitude=1}}] (0.8,0.8) -- (2.2,0.8) -- (2.2,2.2) -- (0.8,2.2) -- cycle;
\draw[thick, red,style={decorate,decoration={snake,amplitude=1}}] (1.8,1.8) -- (3.2,1.8) -- (3.2,3.2) -- (1.8,3.2) -- cycle;
\draw[thick, red,style={decorate,decoration={snake,amplitude=1}}] (0,3.5) -- (1.2,3.5) -- (1.2,1.5) -- (0,1.5);
\draw[thick, red,style={decorate,decoration={snake,amplitude=1}}] (2.6,1.6) -- (4.2,1.6) -- (4.2,3.6) -- (2.6,3.6) -- (2.6,1.6);
\draw[thick, red,style={decorate,decoration={snake,amplitude=1}}] (5,0.6) -- (3.7,0.6) -- (3.7,2.2) -- (5,2.2);
\end{tikzpicture}
\par\end{centering}
}\hfill{}\caption{}
\end{figure}

In this section we establish estimates for the tails of hard-crossing
distances and for moments of crossings of annuli. We state our bound
for hard crossings in terms of a tail bound rather than in terms of
moments because in the proof of \propref{chaining} below, we will
need an explicit superpolynomial concentration bound for hard-crossing
distances.
\begin{prop}
\label{prop:concentration-boost-hard}There is a constant $C<\infty$
so that for any $S,\delta,R,\theta>0$ we have
\[
\mathbf{P}\left(d_{\mathbb{B}(S),\delta,R}(\mathrm{hard})\ge(1+\theta)\Theta_{\mathbb{B}(S),\delta,R}^{\mathrm{hard}}(p_{1})\right)\le C\e^{-(\log(1+\theta))^{8/7}/C}.
\]
\end{prop}

\begin{proof}
This follows from \eqref{hardconcentration-tail} of \propref{hardconcentration},
taking $K=(\log(1+\theta))^{8/7}$, and \propref{thetamonotone}.
\end{proof}
\begin{cor}
\label{cor:hardmomentbound}For any $B\ge0$ we have a constant $C<\infty$
so that the following holds. Let $\delta,R>0$. Suppose that $\mathbb{B}$
is an $S\times S$ square, $\mathbb{R}$ is a rectangle so that either
$3\mathbb{B}\setminus\mathbb{B}\subset\mathbb{R}$ or $(3\mathbb{B}\setminus\mathbb{B})\cap\partial\mathbb{R}$
has exactly two connected components, and $\mathbb{A}=(3\mathbb{B}\setminus\mathbb{B})\cap\mathbb{R}$.
Then we have, recalling the definition \eqref{darounddef} of $d_{\cdots}(\mathrm{around})$,
that
\begin{equation}
\mathbf{E}[d_{(3\mathbb{B})^{*},\mathbb{A},\delta,R}(\mathrm{around})^{B}]\le C\Theta_{\mathbb{B}(S),\delta,R}^{\mathrm{hard}}(p_{1})^{B}.\label{eq:hardmomentbound}
\end{equation}
\end{cor}

\begin{proof}
The construction illustrated in \figref{annularfromhard} shows how
to build an annular circuit out of hard crossings of twelve rectangles.
Thus, union-bounding \propref{concentration-boost-hard} over these
twelve rectangles, and using \eqref{ddi} (of \propref{fieldonsmallerbox})
and \corref{maxcoarse} to control the effect of the coarse field
on these rectangles, we have
\[
\mathbf{P}\left(d_{(3\mathbb{B})^{*},\mathbb{A},\delta,R}(\mathrm{around})\ge\theta\Theta_{\mathbb{B}(S),\delta,R}^{\mathrm{hard}}(p_{1})\right)\le C\e^{-(\log(\theta+1))^{8/7}/C}.
\]
Then we have
\[
\mathbf{E}\left[\left(\frac{d_{(3\mathbb{B})^{*},\mathbb{A},\delta,R}(\mathrm{around})}{\Theta_{\mathbb{B}(S),\delta,R}^{\mathrm{hard}}(p_{1})}\right)^{B}\right]=\int_{0}^{\infty}\mathbf{P}\left(\frac{d_{(3\mathbb{B})^{*},\mathbb{A},\delta,R}(\mathrm{around})}{\Theta_{\mathbb{B}(S),\delta,R}^{\mathrm{hard}}(p_{1})}\ge\theta^{1/B}\right)\,\dif\theta\le C\int_{0}^{\infty}\e^{-(\log(1+\theta))^{8/7}/(CB^{8/7})}\,\dif\theta<C,
\]
where as usual, the constant $C$ has been allowed to change from
step to step. This proves \eqref{hardmomentbound}.
\end{proof}

\subsection{Proof of Proposition~\ref{prop:lowerbound-tail}\label{subsec:lowerbounds}}

In this section, we prove \propref{lowerbound-tail}. We do this using
a percolation argument, relying on the fact that a path between two
macroscopically separated points must make many easy crossings of
smaller boxes. The geometry used in the arguments of this subsection
is essentially identical to that of \cite[Section 4]{DD16}.
\begin{proof}[Proof of \propref{lowerbound-tail}.]
We claim that there are constants $c>0$ and $C<\infty$, independent
of $a$, so that, for every $K\ge C$, there is a collection $\mathcal{C}_{K}$
of $K^{-1}S\times2K^{-1}S$ and $2K^{-1}S\times K^{-1}S$ subboxes
of $\mathbb{R}$ and a set $\mathscr{C}_{a,K}$ of subsets of $\mathcal{C}_{K}$
so that the following properties hold.
\begin{enumerate}
\item $|\mathcal{C}_{K}|\le CK^{2}$.
\item \label{enu:Ndef}All elements $\mathcal{D}\subset\mathscr{C}_{a,K}$
have the same cardinality $N\in[cKa,CKa]$.
\item \label{enu:numpathbound}$|\mathscr{C}_{a,K}|\le C^{N}$. (As we describe
below, this is a consequence of the fact that each element of $\mathscr{C}_{a,K}$
is a walk of length $N$ on a bounded-degree graph.)
\item \label{enu:boxesdisjoint}For each $\mathbb{C}_{1},\mathbb{C}_{2}\in\mathcal{D}\in\mathscr{C}_{a,K}$,
we have $\dist_{\Euc}(\mathbb{C}_{1}^{*},\mathbb{C}_{2}^{*})>0$.
\item \label{enu:Dpidef}Whenever $x,y\in\mathbb{R}$ and $|x-y|\ge a\diam_{\Euc}(\mathbb{R})$,
if $\pi$ is any path between $x$ and $y$ then we have a $\mathcal{D}(\pi)\in\mathscr{C}_{a,K}$
so that $\pi$ includes an easy crossing of $\mathbb{C}$ for each
$\mathbb{C}\in\mathcal{D}(\pi)$.
\end{enumerate}
These collections can be constructed as follows. First, construct
at most $CK^{2}$ annuli which are all congruent to $\mathbb{B}(3K^{-1}S,3K^{-1}S)\setminus\mathbb{B}(K^{-1}S,K^{-1}S)$,
such that the regions enclosed by the inner boundaries of the annuli
cover all of $\mathbb{R}$. Cover each annulus with eight boxes, each
having dimensions $K^{-1}S\times2K^{-1}S$ or $2K^{-1}S\times K^{-1}S$.
Let $\mathcal{C}_{K}$ be the set of all of these boxes, translated
by a minimal amount so that each lies completely inside $\mathbb{R}$.
Now any path that starts in the region enclosed by the inner boundary
of an annulus, and then travels outside of the annulus, must make
an easy crossing of one of the eight boxes associated to that annulus.
(\figref{annularmusteasy}.)
\begin{figure}
\centering
\begin{tikzpicture}[x=0.5in,y=0.5in]
\draw[thin] (0,0) rectangle (2,1);
\draw[thin] (1,0) rectangle (3,1);
\draw[thin] (2,0) rectangle (3,2);
%\draw[thin] (2,1) rectangle (4,2);
\draw[thin] (2,1) rectangle (3,3);
\draw[thin] (1,2) rectangle (3,3);
\fill[thin,fill opacity=0.3] (0,2) rectangle (2,3);
%\draw[thin] (1,2) rectangle (2,4);
\draw[thin] (0,0) rectangle (1,2);
\draw[ultra thick] (1,1) rectangle (2,2);
\draw[ultra thick] (0,0) rectangle (3,3);
\draw[thick, red,style={decorate,decoration={snake,amplitude=1}}] (1,1.4) -- (0.7,2.1) -- (1.4,3);
\end{tikzpicture}

\caption{\label{fig:annularmusteasy}A path from the inside boundary to the
outside boundary of an annulus must make an easy crossing of one of
the eight rectangles that cover the boundary of the annulus.}
\end{figure}
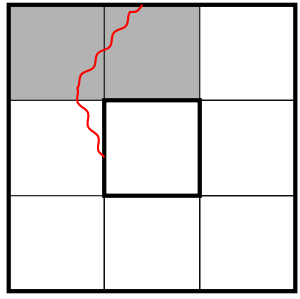
 Therefore, if a path travels a distance $a\diam_{\Euc}(\mathbb{R})$,
it must cross at least $caK$ of the boxes in $\mathcal{C}_{K}$.
From these boxes we can extract a subset of boxes $\{\mathbb{C}_{i}\}$
such that $\{\mathbb{C}_{i}^{*}\}$ (recall that $\mathbb{C}_{i}^{*}=\mathbb{C}_{i}^{(2\diam_{\Euc}\mathbb{C}_{i})}$)
is a pairwise-disjoint family by a greedy algorithm. There still must
be at least $caK$ elements in this subset, perhaps with a smaller
constant $c$, and thus we can truncate all of these subsets so they
all have the same size $N\ge cKa$. Call the subset associated to
the path $\mathcal{D}(\pi)$ and let $\mathscr{C}_{a,K}=\{\mathcal{D}(\pi)\mid x,y\in\mathbb{R}\text{, }|x-y|\ge a\diam_{\Euc}(\mathbb{R})\text{, and }\pi\text{ is a path between \ensuremath{x} and \ensuremath{y}}\}$.
Now the cardinality of $\mathscr{C}_{a,K}$ can be at most $C^{N}$,
since successive elements of $\mathcal{D}(\pi)$ can only be a bounded
distance from each other, and so $\mathcal{D}(\pi)$ is a walk on
a bounded-degree graph of size $N$.

Now let $x,y\in\mathbb{R}$ be such that $|x-y|\ge a\diam_{\Euc}(\mathbb{R})$
and let $\pi$ be a $(\mu_{\mathbb{R}^{*}},\mathbb{R},\delta,R)$-geodesic
(in the sense of \defref{geodesics}) between $x$ and $y$. Therefore,
there is a $\mathcal{D}\in\mathscr{C}_{a,K}$ so that $\pi$ includes
an easy-crossing of every $\mathbb{C}\in\mathcal{D}$. Let 
\[
F=\min\limits _{\mathbb{C}\in\mathcal{C}_{K}}\min\limits _{x\in\mathbb{C}^{\circ}}h_{\mathbb{R}^{*}:\mathbb{C}^{*}}(x).
\]
If $R\le K^{-1}S$, then we have
\begin{equation}
d_{\mathbb{R},\delta,R}(x,y)\ge\sum_{\mathbb{C}\in\mathcal{D}}d_{\mathbb{R}^{*},\mathbb{C},\delta,R}(\mathrm{easy})\ge\sum_{\mathbb{C}\in\mathcal{D}}d_{\mathbb{C},\e^{\gamma F},R\wedge(\sqrt{5}K^{-1}S)}(\mathrm{easy}),\label{eq:applycrossmanysmalls}
\end{equation}
where the first inequality is by \propref{crossmanysmalls} (in which
the hypothesis that the $\mathbb{C}_{i}$s are separated by at least
$R$ is satisfied by the condition on $R$ and property~\enuref{boxesdisjoint}
above), and the second inequality is by \propref{fieldonsmallerbox}.
On the other hand, if we do not assume that $R\le K^{-1}S$, then
we still have that
\begin{equation}
d_{\mathbb{R},\delta,R}(x,y)\ge\max_{\mathbb{C}\in\mathcal{D}}d_{\mathbb{R}^{*},\mathbb{C},\delta,R}(\mathrm{easy})\ge\max_{\mathbb{C}\in\mathcal{D}}d_{\mathbb{R}^{*},\mathbb{C},\delta,R\wedge(\sqrt{5}K^{-1}S)}(\mathrm{easy})\ge\max_{\mathbb{C}\in\mathcal{D}}d_{\mathbb{C},\e^{\gamma F}\delta,R\wedge(\sqrt{5}K^{-1}S)}(\mathrm{easy}),\label{eq:applycrossingbigcrossingsmall}
\end{equation}
where the first inequality is by \propref{crossingbigimpliescrossingsmall},
the second is by \eqref{nopointbeinghuge}, and the third is by \propref{fieldonsmallerbox}.

Recall $N$ from property~\ref{enu:Ndef} above and let $u>0$ (to
be fixed in \eqref{uspec} below). Under the assumption that $R\le K^{-1}S$,
\eqref{applycrossmanysmalls} implies that on the event that $d_{\mathbb{R},\delta,R}(\max;a)\le Nu/2$,
the set $\mathcal{D}(\pi)$ (defined in property~\ref{enu:Dpidef}
above) can contain no more than $N/2$ boxes $\mathbb{C}$ such that
$d_{\mathbb{C},\e^{\gamma F}\delta,R\wedge(\sqrt{5}K^{-1}S)}(\mathrm{easy})\le u$.
Therefore, we have
\begin{align}
\mathbf{P} & \left[d_{\mathbb{R},\delta,R}(\min;a)\le Nu/2\right]\nonumber \\
 & \le\mathbf{P}\left[\text{there exists a \ensuremath{\mathcal{D}\in\mathscr{C}_{a,K}} such that \ensuremath{\#\left\{ \mathbb{C}\in\mathcal{D}\;:\;d_{\mathbb{C},\e^{\gamma F}\delta,R\wedge(\sqrt{5}K^{-1}S)}(\mathrm{easy})\le u\right\} \ge N/2}}\right]\nonumber \\
 & \le\mathbf{P}\left[\text{there exists a \ensuremath{\mathcal{D}\in\mathscr{C}_{a,K}} such that \ensuremath{\#\left\{ \mathbb{C}\in\mathcal{D}\;:\;d_{\mathbb{C},K^{\gamma\theta}\delta,R\wedge(\sqrt{5}K^{-1}S)}(\mathrm{easy})\le u\right\} \ge N/2}}\right]\nonumber \\
 & \qquad\qquad+\mathbf{P}\left[F\le-\theta\log K\right]\nonumber \\
 & \eqqcolon\Xi_{1}+\Xi_{2},\label{eq:maybecoarsebig-RltS}
\end{align}
where the second inequality is by a union bound and $\Xi_{1}$ and
$\Xi_{2}$ are the first and second terms of the second-to-last expression,
respectively. On the other hand, if we do not assume that $R\le K^{-1}S$,
then we still have, by \eqref{applycrossingbigcrossingsmall}, that
\begin{align}
\mathbf{P} & \left[d_{\mathbb{R},\delta,R}(\min;a)\le u\right]\le\mathbf{P}\left[\text{there exists a \ensuremath{\mathcal{D}\in\mathscr{C}_{a,K}} such that \ensuremath{\max_{\mathbb{C}\in\mathcal{D}}d_{\mathbb{C},\e^{\gamma F}\delta,R\wedge(\sqrt{5}K^{-1}S)}(\mathrm{easy})\le u} }\right]\nonumber \\
 & \le\mathbf{P}\left[\text{there exists a \ensuremath{\mathcal{D}\in\mathscr{C}_{a,K}} such that \ensuremath{\max_{\mathbb{C}\in\mathcal{D}}d_{\mathbb{C},K^{\gamma\theta}\delta,R\wedge(\sqrt{5}K^{-1}S)}(\mathrm{easy})\le u}}\right]+\mathbf{P}\left[F\le-\theta\log K\right]\le\Xi_{1}+\Xi_{2},\label{eq:maybecoarsebig-noassumption}
\end{align}
where the last inequality is because, if $\max\limits _{\mathbb{C}\in\mathcal{D}}d_{\mathbb{C},K^{\gamma\theta}\delta,R\wedge(\sqrt{5}K^{-1}S)}(\mathrm{easy})\le u$,
then in particular
\[
\#\left\{ \mathbb{C}\in\mathcal{D}\;:\;d_{\mathbb{C},K^{\gamma\theta}\delta,R\wedge(\sqrt{5}K^{-1}S)}(\mathrm{easy})\le u\right\} =|\mathcal{D}|=N\ge N/2.
\]

Now define, for some $p_{0}\in(0,1)$ to be chosen later,
\begin{equation}
u=\Theta_{\mathbb{B}(K^{-1}S),K^{\gamma\theta}\delta,R\wedge(\sqrt{5}K^{-1}S)}^{\mathrm{easy}}(p_{0})=\Theta_{\mathbb{B}(K^{-\omega-1}S),\delta,K^{-\omega}(R\wedge(\sqrt{5}K^{-1}S))}^{\mathrm{easy}}(p_{0})=\Theta_{\mathbb{B}(K^{-\omega-1}S),\delta,K^{-\omega}R}^{\mathrm{easy}}(p_{0}),\label{eq:uspec}
\end{equation}
where the second equality is by \propref{LQGscaling} and the second
is by \eqref{nopointbeinghuge}. (Recall that $\omega$ was defined
in \eqref{ouromega}.) By \lemref{hugedeviations} and a union bound
along with property~\enuref{numpathbound} above, we have
\[
\Xi_{1}\le|\mathscr{C}_{a,K}|p_{0}^{N/2}\le C^{K}(Cp_{0})^{N/2}\le C^{N}p_{0}^{N/2}.
\]
On the other hand, we have
\[
\Xi_{2}\le CK^{2-\theta^{2}/2}
\]
by \eqref{maxcoarse-easiertouse} of \propref{maxcoarse}. Plugging
the last two bounds into \eqref{maybecoarsebig-RltS} and \eqref{maybecoarsebig-noassumption},
and noting \eqref{uspec}, we obtain
\begin{multline*}
\mathbf{P}\left[d_{\mathbb{R},\delta,R}(\min;a)\le\frac{N}{2}\Theta_{\mathbb{B}(K^{-\omega-1}S),\delta,K^{-\omega}R}^{\mathrm{easy}}(p_{0})\right],\mathbf{P}\left[d_{\mathbb{R},\delta,R}(\min;a)\le\Theta_{\mathbb{B}(K^{-\omega-1}S),\delta,K^{-\omega}R}^{\mathrm{easy}}(p_{0})\right]\\
\le C^{N}p_{0}^{N/2}+CK^{2-\theta^{2}/2}.
\end{multline*}
Then we obtain \eqref{lowerbound-tail-nogrowth} and \eqref{lowerbound-tail-growth}
by taking $p_{0}$ so small that $C\sqrt{p_{0}}<1$ and recalling
property~\enuref{Ndef} above.
\end{proof}

\subsection{Proof of Proposition~\ref{prop:hardconcentration}\label{subsec:upperbounds}}

Although the statement of \propref{hardconcentration} is similar
in spirit to \cite[Proposition 6.1]{DD16}, here we need a better
bound and thus use a more sophisticated percolation argument.
\begin{proof}[Proof of \propref{hardconcentration}.]
We assume that $\AR(\mathbb{R})\in\{1/3,3\}$ and that $K$ is an
integer, the general case follows from \propref{thetamonotone} and
a simple coarse-field bound. Choose an appropriate $L\in\{K/3,K\}$
and divide $\mathbb{R}$ into a $K\times L$ grid of $S\times S$
subboxes, indexing them according to their position in the grid $\mathcal{G}_{K,L}=\{1,\ldots,K\}\times\{1,\ldots,L\}$
as $(\mathbb{C}_{k,\ell})_{(k,\ell)\in\mathcal{G}}$, with the following
layout:
\[
\begin{array}{ccc}
\mathbb{\mathbb{C}}_{1,L} & \cdots & \mathbb{\mathbb{C}}_{K,L}\\
\vdots & \text{\reflectbox{\ensuremath{\ddots}}} & \vdots\\
\mathbb{\mathbb{C}}_{1,1} & \cdots & \mathbb{\mathbb{C}}_{K,1}
\end{array}.
\]
For each $(k,\ell)\in\mathcal{G}_{K,L}$, define $\mathbb{A}_{k,\ell}=(3\mathbb{C}_{k,\ell}\setminus\mathbb{C}_{k,\ell})\cap\mathbb{R}$,
so $\mathbb{A}_{k,\ell}$ is an intersection of a square annulus with
$\mathbb{R}$. Define $\mathcal{Q}$ to be the set of paths $\omega=(\omega_{1},\ldots,\omega_{J(\omega)})$
so that $\omega_{j}\in\mathcal{G}_{K,L}$ for each $j$ and $\omega_{1}^{(1)}=1$,
$\omega_{J(\omega)}^{(1)}=K$ (here the superscript $(1)$ means to
take the $x$-coordinate) and $|\omega_{j}-\omega_{j-1}|_{\infty}\le1$,
where $|\cdot|_{\infty}$ denotes the $\ell^{\infty}$ norm. This
is to say that the paths $\omega$ must cross from left to right in
the grid $\mathcal{G}_{K,L}$ and must move as a chess king: one square
at a time, either left, right, up, down, or diagonally. Then circuits
around $\mathbb{A}_{\omega_{j}}$, $j=1,\ldots,J(\omega)$ can be
joined together as illustrated in \figref{annulitocrossing} to form
a left--right crossing of $\mathbb{R}$, so we have
\[
d_{\mathbb{R},\delta,R}(\mathrm{L},\mathrm{R})\le\min_{\omega\in\mathcal{Q}}\sum_{j=1}^{J(\omega)}d_{\mathbb{R}^{*},\mathbb{A}_{\omega_{j}},\delta,R}(\mathrm{around}).
\]
(Recall the notation $d_{\cdots}(\mathrm{around})$ defined in \eqref{darounddef}--\eqref{darounddef2}.)
Define
\[
\mathcal{B}_{k,\ell}=\{\mathbb{B}\in\widetilde{\mathcal{B}}_{k,\ell}\::\;\mathbb{B}\subset\mathbb{R}\},
\]
where $\widetilde{\mathcal{B}}_{k,\ell}$ comprises the eight unions
of pairs of adjacent $\mathbb{C}_{k',\ell'}$s contained in $\mathbb{A}_{k,\ell}$,
as well as the four unions of pairs $(\mathbb{C}_{k',\ell'},\mathbb{C}_{k'',\ell''})$
such that $\mathbb{C}_{k',\ell'},\mathbb{C}_{k'',\ell''}$ are translates
of $\mathbb{C}_{k,\ell}$ by one and two squares, respectively, either
up, down, left, or right. That is, $\widetilde{\mathcal{B}}_{k,\ell}$
comprises twelve rectangles, and hard crossings of these rectangles
can be joined together to create a circuit around $\mathbb{A}_{k,\ell}$
as shown in \figref{annularfromhard}. This means that
\begin{equation}
d_{\mathbb{R}^{*},\mathbb{A}_{k,\ell},\delta,R}(\mathrm{around})\le\sum_{\mathbb{B}\in\mathcal{B}_{k,\ell}}d_{\mathbb{R}^{*},\mathbb{B},\delta,R}(\mathrm{hard}).\label{eq:breakintohards}
\end{equation}

Define
\[
F=\adjustlimits\max_{(k,\ell)\in\mathcal{G}_{K,L}}\max_{\mathbb{B}\in\mathcal{B}_{k,\ell}}\max_{x\in\mathbb{B}^{\circ}}h_{\mathbb{R}^{*}:\mathbb{B}^{*}}(x).
\]
By \corref{maxcoarse}, we have, for any $\nu>1$, that
\begin{equation}
\mathbf{P}[\e^{\gamma F-\gamma\theta_{0}\log K}\ge\nu]\le C\e^{-\frac{(\log\nu)^{2}}{2\gamma^{2}\log K}}.\label{eq:hbcoarse}
\end{equation}
(Recall that $\theta_{0}$ was fixed in \eqref{etadef}.) Combining
\eqref{breakintohards} and \eqref{hbcoarse}, we have
\begin{equation}
d_{\mathbb{R},\delta,R}(\mathrm{L},\mathrm{R})\le\min_{\omega\in\mathcal{Q}}\sum_{j=1}^{J(\omega)}\sum_{\mathbb{B}\in\mathcal{B}_{k,\ell}}d_{\mathbb{R}^{*},\mathbb{B},\delta,R}(\mathrm{hard})\le C\e^{(\gamma F-\gamma\theta_{0}\log K)^{+}}\min_{\omega\in\mathcal{Q}}\sum_{j=1}^{J(\omega)}\sum_{\mathbb{B}\in\mathcal{B}_{k,\ell}}d_{\mathbb{B},K^{-\gamma\theta_{0}}\delta,R}(\mathrm{hard}).\label{eq:dLRpercolation}
\end{equation}
Now we note that we can partition $\mathcal{G}$ into finitely many
subsets $\mathcal{G}_{1},\ldots,\mathcal{G}_{M}$, with $M$ an absolute
constant not depending on $K$, so that
\[
\left\{ \sum_{\mathbb{B}\in\mathcal{B}_{k,\ell}}d_{\mathbb{B},K^{-\gamma\theta_{0}}\delta,R}(\mathrm{hard})\;:\;(k,\ell)\in\mathcal{G}_{m}\right\} 
\]
is a collection of independent random variables for each $1\le m\le M$.
Noting that $|\mathcal{B}_{k,\ell}|\le12$ for all $k,\ell$, we have
by a union bound
\begin{align}
\mathbf{P}\left(\sum_{\mathbb{B}\in\mathcal{B}_{k,\ell}}d_{\mathbb{B},K^{-\gamma\theta_{0}}\delta,R}(\mathrm{hard})\ge w\right) & \le12\mathbf{P}\left(d_{\mathbb{B}(S),K^{-\gamma\theta_{0}}\delta,R}(\mathrm{hard})\ge w/12\right)\nonumber \\
 & =12\mathbf{P}\left(d_{\mathbb{B}(K^{\eta}S),\delta,K^{\eta}R}(\mathrm{hard})\ge w/12\right),\label{eq:unionboundannulus}
\end{align}
where the equality is by \eqref{etadef} and \propref{LQGscaling}.
Let $\mathcal{R}_{L}$ be the set of paths in $\mathcal{G}$ which
start at the top of $\mathcal{G}$ and have length $L$, where only
nearest-neighbor (horizontal and vertical, not diagonal) edges are
allowed. We note that
\begin{equation}
|\mathcal{R}_{L}|\le L4^{L}\le5^{L}.\label{eq:RLsize}
\end{equation}
If there is no $\omega\in\mathcal{Q}$ so that $\sum\limits _{\mathbb{B}\in\mathcal{B}_{k,\ell}}d_{\mathbb{B},K^{-\gamma\theta_{0}}\delta,R}(\mathrm{hard})\le w$
for all $(k,\ell)\in\omega$, then by planar duality there must be
a $\xi\in\mathcal{R}_{L}$ so that $\sum\limits _{\mathbb{B}\in\mathcal{B}_{k,\ell}}d_{\mathbb{B},K^{-\gamma\theta_{0}}\delta,R}(\mathrm{hard})\ge w$
for all $(k,\ell)\in\xi$. Therefore,
\begin{align}
\mathbf{P} & \left(\min_{\omega\in\mathcal{Q}}\sum_{j=1}^{J(\omega)}\sum_{\mathbb{B}\in\mathcal{B}_{k,\ell}}d_{\mathbb{B},K^{-\gamma\theta_{0}}\delta,R}(\mathrm{hard})\ge KLw\right)\le\mathbf{P}\left(\bigcup_{\xi\in\mathcal{R}_{L}}\bigcap_{(k,\ell)\in\xi}\left\{ \sum_{\mathbb{B}\in\mathcal{B}_{k,\ell}}d_{\mathbb{B},K^{-\gamma\theta_{0}}\delta,R}(\mathrm{hard})\ge w\right\} \right)\nonumber \\
 & \le|\mathcal{R}_{L}|\max_{\xi\in\mathcal{R}_{L}}\mathbf{P}\left(\bigcap_{(k,\ell)\in\xi}\left\{ \sum_{\mathbb{B}\in\mathcal{B}_{k,\ell}}d_{\mathbb{B},K^{-\gamma\theta_{0}}\delta,R}(\mathrm{hard})\ge w\right\} \right).\label{eq:minbound}
\end{align}
Now for any $\xi\in\mathcal{R}_{L}$, we have
\begin{align*}
\mathbf{P} & \left(\bigcap_{(k,\ell)\in\xi}\left\{ \sum_{\mathbb{B}\in\mathcal{B}_{k,\ell}}d_{\mathbb{B},K^{-\gamma\theta_{0}}\delta,R}(\mathrm{hard})\ge w\right\} \right)\le\min_{m=1}^{M}\mathbf{P}\left(\bigcap_{(k,\ell)\in\xi\cap\mathcal{G}_{m}}\left\{ \sum_{\mathbb{B}\in\mathcal{B}_{k,\ell}}d_{\mathbb{B},K^{-\gamma\theta_{0}}\delta,R}(\mathrm{hard})\ge w\right\} \right)\\
 & \le\min_{m=1}^{M}\prod_{(k,\ell)\in\mathcal{G}_{m}}\mathbf{P}\left(\sum_{\mathbb{B}\in\mathcal{B}_{k,\ell}}d_{\mathbb{B},K^{-\gamma\theta_{0}}\delta,R}(\mathrm{hard})\ge w\right)\le\min_{m=1}^{M}[12\mathbf{P}(d_{\mathbb{B}(K^{\eta}S),\delta,K^{\eta}R}(\mathrm{hard})\ge w/12)]^{|\xi\cap\mathcal{G}_{m}|}\\
 & \le[12\mathbf{P}(d_{\mathbb{B}(K^{\eta}S),\delta,K^{\eta}R}(\mathrm{hard})\ge w/12)]^{L/M},
\end{align*}
where in the third inequality we used \eqref{unionboundannulus}.
The right-hand side does not depend on $\xi$, so we have by \eqref{RLsize}
and \eqref{minbound} that
\[
\mathbf{P}\left(\min_{\omega\in\mathcal{Q}}\sum_{j=1}^{J(\omega)}\sum_{\mathbb{B}\in\mathcal{B}_{k,\ell}}d_{\mathbb{B},K^{-\gamma\theta_{0}}\delta,R}(\mathrm{hard})\ge12KLw\right)\le\left[5[12\mathbf{P}(d_{\mathbb{B}(K^{\eta}S),\delta,K^{\eta}R}(\mathrm{hard})\ge w/12]^{1/M}\right]^{L}.
\]
Now we take $w=12\Theta_{\mathbb{B}(KS),\delta,KR}^{\mathrm{hard}}(1-p)$,
with $p<1/(12\cdot10^{M})$, so
\[
\mathbf{P}\left(\min_{\omega\in\mathcal{Q}}\sum_{j=1}^{J(\omega)}\sum_{\mathbb{B}\in\mathcal{B}_{k,\ell}}d_{\mathbb{B},K^{-\gamma\theta_{0}}\delta,R}(\mathrm{hard})\ge12KL\Theta_{\mathbb{B}(K^{\eta}S),\delta,K^{\eta}R}^{\mathrm{hard}}(1-p)\right)\le\left[5\cdot(12p)^{1/M}\right]^{L}\le2^{-L}.
\]
Recalling \eqref{dLRpercolation}, \eqref{hbcoarse}, and the definition
of $L\in\{K/3,3K\}$, we obtain, for all $\nu>1$,
\[
\mathbf{P}\left(d_{\mathbb{R},\delta,R}(\mathrm{L},\mathrm{R})\ge36\nu K^{2}\Theta_{\mathbb{B}(K^{\eta}S),\delta,K^{\eta}R}^{\mathrm{hard}}(1-p)\right)\le C(2^{-K/3}+\e^{-\frac{(\log\nu)^{2}}{2\gamma^{2}\log K}}).
\]
Now taking $\nu=\e^{K^{3/4}}$ yields \eqref{hardconcentration-tail},
while taking $\nu=C\e^{(\log K)^{3/4}}$ for an appropriate constant
$C$ and taking $K$ sufficiently large (depending on the choice of
$p$) yields \eqref{quantileub}.
\end{proof}

\section{Fluctuations of the crossing distance: Efron--Stein inequality and
multiscale analysis\label{sec:inductiveargument}}

In this section, we prove the following theorem, bounding the variance
of the hard crossing distance at a given scale by an absolute constant.
Recall \defref{chidefs} of $\chi_{U}$.
\begin{thm}
\label{thm:variancebound}There is a constant $C<\infty$ so that
\begin{equation}
\sup_{U\in(0,\infty)}\chi_{U}\le C\label{eq:supchiU}
\end{equation}
and
\begin{equation}
\Var(\log d_{\mathbb{B}(U)}(\mathrm{hard}))\le C\label{eq:variancebound}
\end{equation}
for all $U\ge0$.
\end{thm}

Before we prove \thmref{variancebound}, we point out some corollaries.
\begin{cor}
\label{cor:Thetastardef}There is an increasing function $\Theta^{*}:\mathbf{R}_{>0}\to\mathbf{R}_{>0}$
so that for any $0<p<1$ there is a constant $C=C(p)$ so that, for
all $U\in(0,\infty)$,
\begin{equation}
\Theta_{\mathbb{B}(U)}^{\mathrm{easy}}(p),\Theta_{\mathbb{B}(U)}^{\mathrm{hard}}(p)\in[C^{-1}\Theta^{*}(U),C\Theta^{*}(U)].\label{eq:Thetastarprop}
\end{equation}
\end{cor}

\begin{proof}
Take $\Theta^{*}(U)=\Theta_{\mathbb{B}(U)}^{\mathrm{hard}}(1/2)$.
The fact that $\Theta^{*}$ is increasing then follows from \propref{thetamonotone}.
The fact that, for any $p\in(0,1)$, there is a constant $C=C(p)$
so that, for all $U>0$, we have $\Theta_{\mathbb{B}(U)}^{\mathrm{hard}}(p)\in[C^{-1}\Theta^{*}(U),C\Theta^{*}(U)]$
is a consequence of \lemref{relatequantiles} in light of \eqref{variancebound}.
Therefore, by \thmref{RSW} we have, for all $q\in(0,1)$ and $U\in(0,\infty)$,
that
\[
\Theta_{\mathbb{B}(U)}^{\mathrm{easy}}(q)\ge\frac{1}{C_{\mathrm{RSW}}}\Theta_{\mathbb{B}(U)}^{\mathrm{hard}}(q/C_{\mathrm{RSW}})\ge\frac{1}{C(q/C_{\mathrm{RSW}})\cdot C_{\mathrm{RSW}}}\Theta^{*}(U).
\]
On the other hand, we also have from \eqref{hardtoeasy} of \propref{hardquantile},
\propref{thetamonotone}, and \lemref{relatequantiles} in light of
\eqref{variancebound} that for every $p\in(0,1)$ there are constants
$c>0$, $C<\infty$ and $K\in[1,\infty)$ so that, for all $U\in(0,\infty)$,
\[
\Theta_{\mathbb{B}(U)}^{\mathrm{easy}}(p)\le K^{C}\Theta_{\mathbb{B}(K^{-1}U)}^{\mathrm{hard}}(p_{1})\le K^{C}\Theta_{\mathbb{B}(U)}^{\mathrm{hard}}(p_{1})\le CK^{C}\Theta^{*}(U).
\]
Thus we have \eqref{Thetastarprop}.
\end{proof}
As noted in the introduction, the proof of our main result works without
precise knowledge of the growth rate of the quantiles, hence without
precise knowledge of $\Theta^{*}(U)$. However, we do know that the
growth of $\Theta^{*}(U)$ is bounded above and below by power laws.
\begin{cor}
\label{cor:Thetapower}We have a constant $C<\infty$ so that, for
all $U\ge1$, we have
\[
\Theta^{*}(U)\in[C^{-1}U^{C^{-1}},CU^{C}].
\]
\end{cor}

\begin{proof}
This follows from \corref{Thetastardef}, relating $\Theta^{*}(U)$
to the easy and hard crossing quantiles, and \propref[s]{pllb-inductive}~and~\ref{prop:hardquantile},
which give power-law upper and lower bounds on those quantiles.
\end{proof}
We will see that \corref{Thetapower} implies \thmref{plbound}, as
$\Theta^{*}(U)$ will be the normalizing constant for the metric.
(See \corref{finalcor} below.)

Our proof of \thmref{variancebound} will be by induction on the scale
$S$. Throughout the argument in this section, we will fix a scale
$S$ and let $\mathbb{R}=\mathbb{B}(KS,LS)$, where $K/L\in[1/3,3]$.
The scale $S$ will be our inductive parameter and $K$ will represent
the increment of scales that we use in our induction. Ultimately,
$K$ will be chosen to be very large but fixed.

The inductive procedure relies on the Efron--Stein inequality. In
\subsecref{resampling}, we describe a procedure for breaking up an
underlying white noise for the Gaussian free field into independent
chunks, each of which can be resampled independently. One of the chunks
corresponds to a smooth ``coarse'' part of the GFF, while the rest
of the chunks correspond to the ``fine field'' on a small part of
the domain. We apply the Efron--Stein inequality and describe the
induction procedure in \subsecref{efronstein}. The induction procedure
works by assuming that \eqref{supchiU} holds at scale $S$, and then
proving that it still holds at scale $KS$, for the same constant
$C$. The analysis required to show that $C$ and $K$ can be chosen
so that the induction closes is rather delicate. 

To make the induction work, we need to bound the result of the Efron--Stein
inequality, which means we need to estimate how much the crossing
distance of the box changes when each piece of the white noise is
resampled. When we resample the coarse field, this comes from Gaussian
concentration as described in \subsecref{coarsefieldeffect}. When
we resample the fine field, this involves breaking up the geodesic
into a piece which is close to the resampled box and a part that is
far away from the resampled box; this construction is described in
\subsecref{splitting}. To estimate the increase in the number of
LQG balls needed to cover the ``far away'' part of the geodesic,
we don't need to obtain a new geodesic for the resampled field, since
the far away part of the GFF will change by a smooth function. (See
\subsecref{farfinefield}.) On the other hand, to estimate the increase
in the number of LQG balls needed to cover the ``close'' part of
the geodesic, we replace the geodesic by an encircling annular crossing
optimized for the resampled field. (See \subsecref{closefinefield}.)
Throughout the arguments, the RSW result \thmref{RSW} and the percolation-based
bounds of \secref{Percolation} will play essential roles in providing
the bounds that we need to close the induction.

\subsection{Resampling the field\label{subsec:resampling}}

Let $(\mathbb{C}_{k})_{k\in\mathcal{Q}_{*}}$ ($\mathcal{Q}_{*}$
being an index set) be a grid partition of $\mathbb{R}^{*}$ into
disjoint identical squares of size at most $S$ and at least $S/10$
so that $(\mathbb{C}_{i})_{i\in\mathcal{Q}_{\circ}}$ is a partition
of $\mathbb{R}^{\circ}$ for some subset $\mathcal{Q}_{\circ}\subset\mathcal{Q}_{*}$.
We note that there is a constant $C$ so that $|\mathcal{Q}_{*}|\le CK^{2}$.

We note that there is a space-time white noise $W$ on $\mathbb{R}^{*}\times(0,\infty)$
so that we can write
\begin{equation}
h_{\mathbb{R}^{*}}(x)=\sqrt{\pi}\int_{0}^{\infty}\int_{\mathbb{R}^{*}}p_{t/2}^{\mathbb{R}^{*}}(x,y)W(\dif y\,\dif t),\label{eq:WNexpr}
\end{equation}
where $p_{t}^{\mathbb{R}^{*}}$ is the transition kernel at time $t$
for a standard Brownian motion killed on the boundary of $\mathbb{R}^{*}$.
(This can be verified by checking that this yields the covariance
function \eqref{covariancekernel}.) Note that the symbol $h_{\mathbb{R}^{*}}(x)$
is an abuse of notation since the Gaussian free field does not take
pointwise values: \eqref{WNexpr} should be interpreted in the sense
of distributions. Now let $\widetilde{W}$ denote a space-time white
noise on $(0,\infty)\times\mathbb{R}^{*}$ which is independent of
$W$. Define (again in the sense of distributions)
\[
h_{\mathbb{R}^{*}}^{(i)}(x)=\sqrt{\pi}\int_{0}^{S^{2}}\int_{\mathbb{C}_{i}}p_{t/2}^{\mathbb{R}^{*}}(x,y)\widetilde{W}(\dif y\,\dif t)+\sqrt{\pi}\iint_{(\mathbf{R}_{>0}\times\mathbb{R}^{*})\setminus((0,S^{2})\times\mathbb{C}_{i})}p_{t/2}^{\mathbb{R}^{*}}(x,y)W(\dif y\,\dif t)
\]
for each $i\in\mathcal{Q}_{*}$, and
\[
\tilde{h}_{\mathbb{R}^{*}}(x)=\sqrt{\pi}\int_{S^{2}}^{\infty}\int_{\mathbb{R}^{*}}p_{t/2}^{\mathbb{R}^{*}}(x,y)\widetilde{W}(\dif y\,\dif t)+\sqrt{\pi}\int_{0}^{S^{2}}\int_{\mathbb{R}^{*}}p_{t/2}^{\mathbb{R}^{*}}(x,y)W(\dif y\,\dif t).
\]

Of course, $h_{\mathbb{R}^{*}}^{(i)}$, for each $i\in\mathcal{Q}_{*}$,
and $\tilde{h}_{\mathbb{R}^{*}}$ are all Gaussian free fields, identical
in law to $h_{\mathbb{R}^{*}}$. As in \subsecref{GFF}, for all $\mathbb{B}\subset\mathbb{R}^{*}$,
define $h_{\mathbb{R}^{*}:\mathbb{B}}^{(i)}$ (resp. $\tilde{h}_{\mathbb{R}^{*}:\mathbb{B}}$)
to be the harmonic interpolation of $h_{\mathbb{R}^{*}}^{(i)}$ (resp.
$\tilde{h}_{\mathbb{R}^{*}}$) onto $\mathbb{B}$, and simply $h_{\mathbb{R}^{*}}^{(i)}$
(resp. $\tilde{h}_{\mathbb{R}^{*}}$) outside of $\mathbb{B}$. Also
define $h_{\mathbb{B}}^{(i)}=h_{\mathbb{R}^{*}}^{(i)}-h_{\mathbb{R^{*}}:\mathbb{B}}^{(i)}$
and $\tilde{h}_{\mathbb{B}}=\tilde{h}_{\mathbb{R}^{*}}-\tilde{h}_{\mathbb{R}^{*}:\mathbb{B}}$
on $\mathbb{B}$.

For all $\mathbb{B}\subset\mathbb{R}^{*}$, let $\mu_{\mathbb{B}}^{(i)}$
(resp. $\tilde{\mu}_{\mathbb{B}}$) denote the Liouville quantum gravity
measure on $\mathbb{B}$ using the GFF $h_{\mathbb{B}}^{(i)}$ (resp.
$\tilde{h}_{\mathbb{B}}$) and let $d_{\mathbb{B},\ldots}^{(i)}$
(resp. $\tilde{d}_{\mathbb{B},\ldots}$) denote the corresponding
Liouville graph distance. Thus $d_{\mathbb{B},\ldots}^{(i)}$ and
$\tilde{d}_{\mathbb{B},\ldots}$ are identical in law to $d_{\mathbb{B},\ldots}$.
The only difference is that they are defined using different realizations
of the underlying white noise field.

\subsection{Splitting the geodesic\label{subsec:splitting}}

To control the variance of the LFPP crossing distance, we will use
the Efron--Stein inequality to control the effect of resampling each
piece of the white noise on the crossing distance. The idea is that,
resampling the white noise on a small box only has a significant effect
on the weight of the part of the geodesic close to that box. Put simply,
by breaking up the white noise and also breaking up the geodesic,
we can control the influence of each part of the white noise on each
part of the geodesic. In the last subsection, we described how we
break up the white noise. In this one, we describe how we break up
the geodesic.

Suppose that $\pi$ is an $(\mathbb{R}^{*},\mathbb{R},1,S)$-geodesic
of $\mathbb{R}$ and let $B_{1},\ldots,B_{N}$ be the corresponding
sequence of geodesic balls, so $N$ is the LQG distance between the
two endpoints of $\pi$. Fix a parameterization of $\pi$, which we
will also call $\pi:[0,1]\to\mathbb{R}$. Define
\[
\mathcal{T}=\pi^{-1}\left(\bigcup_{n=1}^{N-1}(B_{n}\cap B_{n+1})\right).
\]
Let $\tau_{0}=0$, and inductively define $\tau_{m}=1\wedge\min\{\tau\ge\tau_{m-1}\::\;\tau\in\mathcal{T}\text{ and }\diam_{\Euc}(\pi([\tau_{m-1},\tau]))\ge S\}$.
Let $M$ be the first $m$ so that $\tau_{m}=1$. For each $m=1,\ldots,M$,
let $x_{m}=\pi(\tau_{m})$ and let $\pi_{m}=\pi([\tau_{m-1},\tau_{m}])$.
Thus $\pi$ is the union of $\pi_{1},\ldots,\pi_{M}$. By the definition
of $\mathcal{T}$ and the fact that the $B_{i}$s have radius at most
$S$, we have
\begin{equation}
\diam_{\Euc}(\pi_{m})\le2S.\label{eq:diampibound}
\end{equation}
Also, by the definition of $\mathcal{T}$ and the triangle inequality,
we have
\begin{equation}
\sum_{m=1}^{M}d_{\mathbb{R},1,S}(x_{m-1},x_{m})=d_{\mathbb{R},1,S}(\mathrm{L},\mathrm{R}).\label{eq:splitisgood}
\end{equation}

We will need a way to identify which parts of our path are ``close''
to a given box $\mathbb{C}_{i}$. For $i\in\mathcal{Q}_{*}$, define
\begin{equation}
\mathcal{N}(i)=\{j\in\mathcal{Q}_{\circ}\;:\;\dist_{\Euc}(\mathbb{C}_{i},\mathbb{C}_{j})<4S\}\label{eq:Nidef}
\end{equation}
to be the set of indices of ``nearby'' boxes to $\mathbb{C}_{i}$,
and define
\[
\mathbb{D}_{i}=\bigcup_{j\in\mathcal{N}(i)}\mathbb{C}_{j}
\]
as the union of these boxes. (Thus, $\mathbb{D}_{i}$ is a larger
box surrounding $\mathbb{C}_{i}$.) Also define
\begin{equation}
\mathcal{D}(i)=\mathcal{Q}_{\circ}\setminus\mathcal{N}(i)=\{j\in\mathcal{Q}_{\circ}\;:\;\dist_{\Euc}(\mathbb{C}_{i},\mathbb{C}_{j})\ge4S\}.\label{eq:Didef}
\end{equation}

For $i\in\mathcal{Q}_{*}$, we then define
\[
\mathcal{M}(i)=\{1\le m\le M\;:\;\pi_{m}\cap\mathbb{D}_{i}\ne\emptyset\},
\]
the set of segments $\pi_{m}$ that ``get close to'' $\mathbb{C}_{i}$.
Define
\begin{equation}
\pi_{i}^{-}=\pi\setminus\bigcup_{m\in\mathcal{M}(i)}\pi_{m},\label{eq:piiminusdef}
\end{equation}
the part of $\pi$ comprised of segments which do not get too close
to $\mathbb{C}_{i}$. We note for later use that there is a constant
$C<\infty$ so that
\begin{equation}
\#\{j\in\mathcal{Q}_{\circ}\;:\;\pi_{m}\cap\mathbb{C}_{j}\ne\emptyset\}\le C\label{eq:jpiCj}
\end{equation}
and that
\begin{equation}
\#\{i\in\mathcal{Q}_{*}\;:\;m\in\mathcal{M}(i)\}\le C.\label{eq:iMsize}
\end{equation}

Now define 
\[
\mathbb{E}_{i}=\bigcup_{\substack{j\in\mathcal{Q}_{\circ}\\
\dist_{\Euc}(\mathbb{C}_{i},\mathbb{C}_{j})<8S
}
}\mathbb{C}_{j},
\]
which is an even larger box around $\mathbb{D}_{i}$; that is, we
have $\mathbb{C}_{i}\subset\mathbb{D}_{i}\subset\mathbb{E}_{i}$.
Then whenever $m\in\mathcal{M}(i)$, we have $\pi_{m}\subset\mathbb{E}_{i}$
by \eqref{diampibound}. That is, $\mathbb{E}_{i}$ completely contains
all segments of the geodesic which get close (in our sense) to $\mathbb{C}_{i}$.
By \propref{crossingbigimpliescrossingsmall}, we therefore have
\begin{equation}
d_{\mathbb{R},1,S}(x_{m-1},x_{m})=d_{\mathbb{R}^{*},\mathbb{E}_{i},1,S}(x_{m-1},x_{m}).\label{eq:pathstaysinEi}
\end{equation}

Another consequence of the fact $\pi_{m}\subset\mathbb{E}_{i}$ is
that $\pi_{m}$ can be ``replaced'' in the path $\pi$ by a loop
around an annulus surrounding $\mathbb{E}_{i}$. That is, if we define
\begin{equation}
\mathbb{A}_{i}=(3\mathbb{E}_{i}\setminus\mathbb{E}_{i})\cap\mathbb{R},\label{eq:Aidef}
\end{equation}
then we have
\begin{equation}
d_{\mathbb{R},1,S}^{(i)}(\mathrm{L},\mathrm{R})\le d_{\mathbb{R},1,S}^{(i)}(\pi_{i}^{-})+d_{\mathbb{R}^{*},\mathbb{A}_{i},1,S}^{(i)}(\mathrm{around}).\label{eq:replacepimbyanannulus}
\end{equation}
The motivation for using \eqref{replacepimbyanannulus} will be that,
when we resample the white noise in box $\mathbb{C}_{i}$, we will
not have pointwise control on the change in the Gaussian free field
close to $\mathbb{C}_{i}$. Therefore, there is no reason to believe
that the paths $\pi_{m}$, $m\in\mathcal{M}(i)$, will still be close
to geodesics after the resampling. %
\begin{comment}
There is not even a reason to believe that replacing $\pi_{m}$ by
a geodesic between the endpoints of $\pi_{m}$ will result in a new
path that is close to a geodesic. (That was the strategy used in \cite{DD16}.) 
\end{comment}
Instead, we will to replace the $\pi_{m}$s by a geodesic circuit
around an annulus, which will be optimized with respect to the resampled
field.

\subsection{The Efron--Stein argument and induction procedure\label{subsec:efronstein}}

We now describe how to apply the Efron--Stein inequality to obtain
a bound on the variance. Because our procedure is inductive, we will
need to consider separately an error term coming from the event that
the coarse field on a given box is larger than $\theta_{0}\log K$,
where $\theta_{0}$ is as in \eqref{etadef}: if this does not hold,
then the small boxes may not in fact look smaller than $\mathbb{R}$
when the coarse field is taken into account. We define the following
events, recalling the definition \eqref{Didef} of $\mathcal{D}(i)$
for each $i\in\mathcal{Q}_{*}$:
\begin{align}
\mathcal{A}_{1;i} & =\left\{ \max_{x\in(3\mathbb{E}_{i})^{\circ}}\max\{h_{\mathbb{R}^{*}:(3\mathbb{E}_{i})^{*}}^{(i)}(x),h_{\mathbb{R}^{*}:(3\mathbb{E}_{i})^{*}}(x)\}\le\theta_{0}\log K\right\} ;\label{eq:A1idef}\\
\mathcal{A}_{2;i} & =\left\{ \max_{j\in\mathcal{D}(i)}\max_{x\in(3\mathbb{E}_{j})^{\circ}}\max\{h_{\mathbb{R}^{*}:(3\mathbb{E}_{j})^{*}}^{(i)}(x),h_{\mathbb{R}^{*}:(3\mathbb{E}_{j})^{*}}(x)\}\le\theta_{0}\log K+S^{-1}\dist_{\Euc}(\mathbb{C}_{i},\mathbb{C}_{j}^{\circ})\right\} ;\label{eq:A2idef}\\
\mathcal{A}_{i} & =\mathcal{A}_{1;i}\cap\mathcal{A}_{2;i}.\label{eq:scrAidef}
\end{align}

\begin{lem}
\label{lem:Aicbound}We have a constant $C$ so that $\mathbf{P}(\mathcal{A}_{i}^{c})\le CK^{-\theta_{0}^{2}/2}$.
\end{lem}

\begin{proof}
We note that there is a constant $C$ so that
\begin{equation}
\mathbf{P}(\mathcal{A}_{1;i}^{c})\le CK^{-\theta_{0}^{2}/2}\label{eq:PAi1}
\end{equation}
by \eqref{maxcoarse-single}. Also, we have
\begin{align*}
\mathbf{P} & \left(\max_{x\in\mathbb{E}_{j}^{\circ}}\max\{h_{\mathbb{R}^{*}:\mathbb{E}_{i}^{*}}^{(i)}(x),h_{\mathbb{R}^{*}:\mathbb{E}_{i}^{*}}(x)\}\le\theta_{0}\log K+S^{-1}\dist_{\Euc}(\mathbb{C}_{i},\mathbb{C}_{j}^{\circ})\right)\\
 & \le2\mathbf{P}\left(\max_{x\in\mathbb{E}_{j}^{\circ}}h_{\mathbb{R}^{*}:\mathbb{E}_{i}^{*}}(x)\le\left(\theta_{0}+\frac{\dist_{\Euc}(\mathbb{C}_{i},\mathbb{C}_{j}^{\circ})}{S\log K}\right)\log K\right)\\
 & \le C\exp\left\{ -\frac{1}{2}\left(\theta_{0}+\frac{\dist(\mathbb{C}_{i},\mathbb{C}_{j}^{\circ})}{S\log K}\right)^{2}\log K\right\} \le CK^{-\theta_{0}^{2}/2}\exp\left\{ -\frac{(\theta_{0}+2)\dist(\mathbb{C}_{i},\mathbb{C}_{j}^{\circ})}{2S}\right\} ,
\end{align*}
where the second inequality is by \eqref{maxcoarse-single}. Therefore,
by a union bound, we obtain
\begin{equation}
\mathbf{P}(\mathcal{A}_{2;i})\le CK^{-\theta_{0}^{2}/2}\sum_{j\in\mathcal{D}(i)}\exp\left\{ -\frac{(\theta_{0}+2)\dist(\mathbb{C}_{i},\mathbb{C}_{j}^{\circ})}{2S}\right\} \le CK^{-\theta_{0}^{2}/2}.\label{eq:PAi2}
\end{equation}
Taking a union bound over \eqref{PAi1} and \eqref{PAi2}, we obtain
the conclusion of the lemma.
\end{proof}
Now we are ready to state the result of our application of the Efron--Stein
inequality.
\begin{lem}
\label{lem:efronstein-readyforlemmas}We have
\begin{align}
\Var & \left(\log d_{\mathbb{R}}(\mathrm{L},\mathrm{R})\right)\nonumber \\
 & \le\mathbf{E}\left[\left(\log\frac{\tilde{d}_{\mathbb{R},1,S}(\mathrm{L},\mathrm{R})}{d_{\mathbb{R},1,S}(\mathrm{L},\mathrm{R})}\right)^{2}\right]+2\sum_{i\in\mathcal{Q}_{*}}\mathbf{E}\left[\mathbf{1}_{\mathcal{A}_{i}}\left(\left(\frac{d_{\mathbb{R},1,S}^{(i)}(\pi_{i}^{-})}{d_{\mathbb{R},1,S}(\pi)}-1\right)^{+}\right)^{2}\right]+2\Var\left(\log\frac{d_{\mathbb{R},1,S}(\mathrm{L},\mathrm{R})}{d_{\mathbb{R}}(\mathrm{L},\mathrm{R})}\right)\nonumber \\
 & \qquad+2\left(\mathbf{E}\left[\max_{i\in\mathcal{Q}_{*}}d_{(3\mathbb{E}_{i})^{*},\mathbb{A}_{i},K^{-\gamma\theta_{0}},S}^{(i)}(\mathrm{around})^{3}\right]\mathbf{E}[d_{\mathbb{R},1,S}(\mathrm{L},\mathrm{R})^{-3}]\mathbf{E}\left[\left(\sum_{i\in\mathcal{Q}_{*}}\frac{\mathbf{1}_{\mathcal{A}_{i}}d_{\mathbb{R}^{*},\mathbb{A}_{i},1,S}^{(i)}(\mathrm{around})}{d_{\mathbb{R},1,S}(\mathrm{L},\mathrm{R})}\right)^{3}\right]\right)^{1/3}\nonumber \\
 & \qquad+\sum_{i\in\mathcal{Q}_{*}}\mathbf{E}\left[\left(\left(\log\frac{d_{\mathbb{R},1,S}^{(i)}(\mathrm{L},\mathrm{R})}{d_{\mathbb{R},1,S}(\mathrm{L},\mathrm{R})}\right)^{+}\right)^{2}\mathbf{1}_{\mathcal{A}_{i}^{c}}\right].\label{eq:efronstein-readyforlemmas}
\end{align}
\end{lem}

\begin{rem}
Before deriving \eqref{efronstein-readyforlemmas}, we give a brief
description of its terms. By symmetry, we only need to consider the
potential \emph{increase} in the distance after resampling the field.
To this end, we consider a geodesic path with respect to the original
field, and for each type of resampling use it to construct a slightly
perturbed path that has a not-too-much-larger LQG graph length with
respect to the resampled field. The first term in \eqref{efronstein-readyforlemmas}
represents the effect of resampling the white noise at times far in
the future, which corresponds to resampling a smooth ``coarse field.''
This has a ``Lipschitz'' effect on the weight of the path, so we
can bound it using Gaussian concentration inequalities. The second
term represents the effect of resampling the white noise at times
close to $0$ on the path far away from the resampled region. Here
again, the resampling should be smooth, but since we are only considering
the effect on part of the path, we need to use a more customized argument.
The fourth term represents the effect of resampling the white noise
on the path close to the resampled region. In this case, the LQG in
the relevant region should change substantially, so we replace the
path in this region by an optimal circuit of an annulus surrounding
the region. The third term accounts for the error incurred in passing
from $d_{\mathbb{R},1,S}(\mathrm{L},\mathrm{R})$ to $d_{\mathbb{R}}(\mathrm{L},\mathrm{R})$,
which is to say the error incurred by requiring the geodesic balls
to have radius at most $S$. This error should be negligible if the
scale is large enough; we deal with small scales by the \emph{a priori}
bound \eqref{chiUaprioribound}, which was proved in \cite{DZZ18}.
Finally, the last term is an error term incurred from the possibility
that the coarse field could be extraordinarily large, which happens
with low probability and thus is bounded in \subsecref{largecoarsefield}.
\end{rem}

\begin{proof}[Proof of \lemref{efronstein-readyforlemmas}.]
By the Cauchy--Schwarz inequality, we have
\begin{equation}
\Var\left(\log d_{\mathbb{R}}(\mathrm{L},\mathrm{R})\right)\le2\Var\left(\log d_{\mathbb{R},1,S}(\mathrm{L},\mathrm{R})\right)+2\Var\left(\log\frac{d_{\mathbb{R},1,S}(\mathrm{L},\mathrm{R})}{d_{\mathbb{R}}(\mathrm{L},\mathrm{R})}\right).\label{eq:splitoffSpiece}
\end{equation}
By the Efron--Stein inequality and the decomposition of the white
noise described in \subsecref{resampling}, we have
\begin{equation}
\begin{aligned}\Var\left(\log d_{\mathbb{R}}(\mathrm{L},\mathrm{R})\right) & \le\frac{1}{2}\mathbf{E}\left[\left(\log\frac{\tilde{d}_{\mathbb{R},1,S}(\mathrm{L},\mathrm{R})}{d_{\mathbb{R},1,S}(\mathrm{L},\mathrm{R})}\right)^{2}\right]+\frac{1}{2}\sum_{i\in\mathcal{Q}_{*}}\mathbf{E}\left[\left(\log\frac{d_{\mathbb{R},1,S}^{(i)}(\mathrm{L},\mathrm{R})}{d_{\mathbb{R},1,S}(\mathrm{L},\mathrm{R})}\right)^{2}\right].\end{aligned}
\label{eq:efronstein}
\end{equation}

We note that the first term of \eqref{efronstein} corresponds to
the resampling of the coarse field, while the second corresponds to
the resampling of the fine field. For the present, we satisfy ourselves
with further developing the fine field term. Using the exchangeability
of the resampled and unresampled white noises, we obtain
\begin{equation}
\mathbf{E}\left[\left(\log\frac{d_{\mathbb{R},1,S}^{(i)}(\mathrm{L},\mathrm{R})}{d_{\mathbb{R},1,S}(\mathrm{L},\mathrm{R})}\right)^{2}\right]=2\mathbf{E}\left[\left(\left(\log\frac{d_{\mathbb{R},1,S}^{(i)}(\mathrm{L},\mathrm{R})}{d_{\mathbb{R},1,S}(\mathrm{L},\mathrm{R})}\right)^{+}\right)^{2}\right].\label{eq:useplus}
\end{equation}
Now we split the right-hand side of \eqref{useplus} as
\begin{equation}
\mathbf{E}\left[\left(\left(\log\frac{d_{\mathbb{R},1,S}^{(i)}(\mathrm{L},\mathrm{R})}{d_{\mathbb{R},1,S}(\mathrm{L},\mathrm{R})}\right)^{+}\right)^{2}\right]=\mathbf{E}\left[\left(\left(\log\frac{d_{\mathbb{R},1,S}^{(i)}(\mathrm{L},\mathrm{R})}{d_{\mathbb{R},1,S}(\mathrm{L},\mathrm{R})}\right)^{+}\right)^{2}(\mathbf{1}_{\mathcal{A}_{i}}+\mathbf{1}_{\mathcal{A}_{i}^{c}})\right].\label{eq:splitbyAi}
\end{equation}

Recalling the elementary inequality $(\log\frac{X}{Y})^{+}\le(X-Y)^{+}/Y$
for $X,Y>0$, we have
\begin{equation}
\left(\left(\log\frac{d_{\mathbb{R},1,S}^{(i)}(\mathrm{L},\mathrm{R})}{d_{\mathbb{R},1,S}(\mathrm{L},\mathrm{R})}\right)^{+}\right)^{2}\le\left(\frac{\left(d_{\mathbb{R},1,S}^{(i)}(\mathrm{L},\mathrm{R})-d_{\mathbb{R},1,S}(\mathrm{L},\mathrm{R})\right)^{+}}{d_{\mathbb{R},1,S}(\mathrm{L},\mathrm{R})}\right)^{2}.\label{eq:boundlocalpart}
\end{equation}
Now we have by \eqref{replacepimbyanannulus} that
\begin{align}
\left(d_{\mathbb{R},1,S}^{(i)}(\mathrm{L},\mathrm{R})-d_{\mathbb{R},1,S}(\mathrm{L},\mathrm{R})\right)^{+} & \le\left(d_{\mathbb{R},1,S}^{(i)}(\pi_{i}^{-})+d_{\mathbb{R}^{*},\mathbb{A}_{i},1,S}^{(i)}(\mathrm{around})-d_{\mathbb{R},1,S}(\pi)\right)^{+}\nonumber \\
 & \le\left(d_{\mathbb{R},1,S}^{(i)}(\pi_{i}^{-})-d_{\mathbb{R},1,S}(\pi)\right)^{+}+d_{\mathbb{R}^{*},\mathbb{A}_{i},1,S}^{(i)}(\mathrm{around}).\label{eq:breakup-useannulus}
\end{align}
Substituting \eqref{breakup-useannulus} into \eqref{boundlocalpart},
and using the fact that $\pi$ is an $(\mathbb{R}^{*},\mathbb{R},1,S)$-geodesic,
gives us
\begin{align}
\left(\left(\log\frac{d_{\mathbb{R},1,S}^{(i)}(\mathrm{L},\mathrm{R})}{d_{\mathbb{R},1,S}(\mathrm{L},\mathrm{R})}\right)^{+}\right)^{2} & \le\left(\frac{\left(d_{\mathbb{R},1,S}^{(i)}(\pi_{i}^{-})-d_{\mathbb{R},1,S}(\pi)\right)^{+}+d_{\mathbb{R}^{*},\mathbb{A}_{i},1,S}^{(i)}(\mathrm{around})}{d_{\mathbb{R},1,S}(\mathrm{L},\mathrm{R})}\right)^{2}\nonumber \\
 & \le2\left(\left(\frac{d_{\mathbb{R},1,S}^{(i)}(\pi_{i}^{-})}{d_{\mathbb{R},1,S}(\pi)}-1\right)^{+}\right)^{2}+2\left(\frac{d_{\mathbb{R}^{*},\mathbb{A}_{i},1,S}^{(i)}(\mathrm{around})}{d_{\mathbb{R},1,S}(\mathrm{L},\mathrm{R})}\right)^{2}.\label{eq:developfinefield}
\end{align}
Summing the $\mathbf{1}_{\mathcal{A}_{i}}$ part of \eqref{splitbyAi},
and then applying \eqref{developfinefield}, we get
\begin{multline}
\sum_{i\in\mathcal{Q}_{*}}\mathbf{E}\left[\mathbf{1}_{\mathcal{A}_{i}}\left(\left(\log\frac{d_{\mathbb{R},1,S}^{(i)}(\mathrm{L},\mathrm{R})}{d_{\mathbb{R},1,S}(\mathrm{L},\mathrm{R})}\right)^{+}\right)^{2}\right]\le2\sum_{i\in\mathcal{Q}_{*}}\mathbf{E}\left[\mathbf{1}_{\mathcal{A}_{i}}\left(\left(\frac{d_{\mathbb{R},1,S}^{(i)}(\pi_{i}^{-})}{d_{\mathbb{R},1,S}(\pi)}-1\right)^{+}\right)^{2}\right]\\
+2\sum_{i\in\mathcal{Q}_{*}}\mathbf{E}\left[\mathbf{1}_{\mathcal{A}_{i}}\left(\frac{d_{\mathbb{R}^{*},\mathbb{A}_{i},1,S}^{(i)}(\mathrm{around})}{d_{\mathbb{R},1,S}(\mathrm{L},\mathrm{R})}\right)^{2}\right].\label{eq:sumitup}
\end{multline}
We bound the second term of \eqref{sumitup} by 
\begin{align}
\sum_{i\in\mathcal{Q}_{*}}\mathbf{E} & \left[\mathbf{1}_{\mathcal{A}_{i}}\left(\frac{d_{\mathbb{R}^{*},\mathbb{A}_{i},1,S}^{(i)}(\mathrm{around})}{d_{\mathbb{R},1,S}(\mathrm{L},\mathrm{R})}\right)^{2}\right]\le\mathbf{E}\left[\left(\frac{\max\limits _{i\in\mathcal{Q}_{*}}\mathbf{1}_{\mathcal{A}_{i}}d_{\mathbb{R}^{*},\mathbb{A}_{i},1,S}^{(i)}(\mathrm{around})}{d_{\mathbb{R},1,S}(\mathrm{L},\mathrm{R})}\right)\left(\frac{\sum\limits _{i\in\mathcal{Q}_{*}}\mathbf{1}_{\mathcal{A}_{i}}d_{\mathbb{R}^{*},\mathbb{A}_{i},1,S}^{(i)}(\mathrm{around})}{d_{\mathbb{R},1,S}(\mathrm{L},\mathrm{R})}\right)\right]\nonumber \\
 & \le\left(\mathbf{E}\left[\max_{i\in\mathcal{Q}_{*}}d_{\mathbb{R}^{*},\mathbb{A}_{i},1,S}^{(i)}(\mathrm{around})^{3}\mathbf{1}_{\mathcal{A}_{i}}\right]\mathbf{E}[d_{\mathbb{R},1,S}(\mathrm{L},\mathrm{R})^{-3}]\mathbf{E}\left[\left(\sum_{i\in\mathcal{Q}_{*}}\frac{\mathbf{1}_{\mathcal{A}_{i}}d_{\mathbb{R}^{*},\mathbb{A}_{i},1,S}^{(i)}(\mathrm{around})}{d_{\mathbb{R},1,S}(\mathrm{L},\mathrm{R})}\right)^{3}\right]\right)^{1/3}.\label{eq:firstHolder}
\end{align}
We note that by \eqref{scaledeltabymax} and \eqref{A1idef}, we have
\[
d_{\mathbb{R}^{*},\mathbb{A}_{i},1,S}^{(i)}(\mathrm{around})\mathbf{1}_{\mathcal{A}_{i}}\le d_{(3\mathbb{E}_{i})^{*},\mathbb{A}_{i},K^{-\gamma\theta_{0}},S}^{(i)}(\mathrm{around});
\]
applying this in \eqref{firstHolder} and substituting into \eqref{sumitup},
we bound the first term of the sum of \eqref{splitbyAi} by
\begin{multline}
\sum_{i\in\mathcal{Q}_{*}}\mathbf{E}\left[\mathbf{1}_{\mathcal{A}_{i}}\left(\left(\log\frac{d_{\mathbb{R},1,S}^{(i)}(\mathrm{L},\mathrm{R})}{d_{\mathbb{R},1,S}(\mathrm{L},\mathrm{R})}\right)^{+}\right)^{2}\right]\le2\sum_{i\in\mathcal{Q}_{*}}\mathbf{E}\left[\mathbf{1}_{\mathcal{A}_{i}}\left(\left(\frac{d_{\mathbb{R},1,S}^{(i)}(\pi_{i}^{-})}{d_{\mathbb{R},1,S}(\pi)}-1\right)^{+}\right)^{2}\right]+\\
+2\left(\mathbf{E}\left[\max_{i\in\mathcal{Q}_{*}}d_{(3\mathbb{E}_{i})^{*},\mathbb{A}_{i},K^{-\gamma\theta_{0}},S}^{(i)}(\mathrm{around})^{3}\right]\mathbf{E}[d_{\mathbb{R},1,S}(\mathrm{L},\mathrm{R})^{-3}]\mathbf{E}\left[\left(\sum_{i\in\mathcal{Q}_{*}}\frac{\mathbf{1}_{\mathcal{A}_{i}}d_{\mathbb{R}^{*},\mathbb{A}_{i},1,S}^{(i)}(\mathrm{around})}{d_{\mathbb{R},1,S}(\mathrm{L},\mathrm{R})}\right)^{3}\right]\right)^{1/3}.\label{eq:firstterm-final}
\end{multline}

Then \eqref{efronstein-readyforlemmas} follows from \eqref{splitoffSpiece},
\eqref{efronstein}, \eqref{useplus}, \eqref{splitbyAi} and \eqref{firstHolder}.
\end{proof}
Most of the remainder of this section will be devoted to bounding
the terms of \eqref{efronstein-readyforlemmas}. Before we get to
that, though, we state the overall bound we obtain, and how we use
it to prove \thmref{variancebound}.
\begin{lem}
\label{lem:finalvariancebound}We have constants $0<c,C<\infty$ and
$K_{0}\ge1$ so that, if $K\ge K_{0}$ and $S\ge K^{-\eta}T_{0}$
(where $\eta$ is defined in \eqref{etadef} and $T_{0}$ is defined
as in \propref{pllb-inductive}),
\begin{equation}
\Var(\log d_{\mathbb{R}}(\mathrm{L},\mathrm{R}))\le C\log K+CK^{-c/2}\overline{\chi}_{KS/100}^{3}+CK^{C}S^{-c}.\label{eq:finalvariancebound}
\end{equation}
\end{lem}

We will prove \lemref{finalvariancebound} in \subsecref{proofoffinalvarbound}.
First we show how to use \lemref{finalvariancebound} to carry out
the inductive argument to prove \thmref{variancebound}.
\begin{proof}[Proof of \thmref{variancebound}.]
\label{varianceboundproof}We have constants $0<z,Z<\infty$ such
that the following holds. First, by \lemref{finalvariancebound},
we have that whenever $K\ge Z$ and $S\ge K^{-\eta}T_{0}$, 
\begin{equation}
\Var(\log d_{\mathbb{R}}(\mathrm{L},\mathrm{R}))\le Z\left(\log K+K^{-z/2}\overline{\chi}_{KS/100}^{3}+K^{Z}S^{-z}\right).\label{eq:finalvariancebound-1}
\end{equation}
Moreover, by \eqref{chiUaprioribound}, for every $S_{0}>0$, we have
\begin{equation}
\overline{\chi}_{S_{0}}\le Z\e^{(\log S_{0})^{0.95}}.\label{eq:chiS0basecase}
\end{equation}
Finally, by \eqref{chiUvarbound}, for all $U>0$ we have
\begin{equation}
\chi_{U}\le Z\exp\left\{ Z\sqrt{\Var\left(\log d_{\mathbb{B}(U)}(\mathrm{hard})\right)}\right\} .\label{eq:chiUvarbound-application}
\end{equation}
At this point we treat the values of $z$ and $Z$ as fixed, and do
not let them change throughout the proof as we have been doing for
constants labeled $c$ and $C$. We want to choose $K$ and $S_{0}$
so that
\begin{align}
K & \ge Z,\label{eq:Kbigenough}\\
K^{\eta}S_{0} & \ge T_{0},\label{eq:S0satisfies}\\
K^{Z}S_{0}^{-z} & \le\log K,\label{eq:firsttermOK}\\
ZK^{-z}\e^{3(\log S_{0})^{0.95}\vee3Z\sqrt{3Z\log K}} & \le\log K.\label{eq:secondtermOK}
\end{align}
We can do this as follows. First, we fix
\begin{equation}
S_{0}=K^{Z/z+1},\label{eq:S0def}
\end{equation}
which guarantees that \eqref{firsttermOK} holds as long as $K$ is
sufficiently large. Then $K^{\eta-1}S_{0}=K^{Z/z+\eta}$, so as long
as $K$ is sufficiently large, \eqref{S0satisfies} will hold. Finally,
for \eqref{secondtermOK} to hold, we need
\[
\e^{(3(Z/z+z)^{0.95}(\log K)^{0.95})\vee3Z\sqrt{3Z\log K}-z\log K}\le\log K,
\]
which can be achieved by taking $K$ sufficiently large. Finally,
\eqref{Kbigenough} of course holds for sufficiently large $K$. Therefore,
we can achieve \eqref{Kbigenough}--\eqref{secondtermOK} by fixing
$K$ sufficiently large and then imposing \eqref{S0def}. Henceforth
we assume that $K$ and $S_{0}$ have been chosen in this way.

We note that \eqref{chiS0basecase} and \eqref{secondtermOK} imply
that if we put
\begin{equation}
\chi\coloneqq\overline{\chi}_{S_{0}}\vee Z\e^{Z\sqrt{3Z\log K}},\label{eq:chidef}
\end{equation}
then
\begin{equation}
K^{-z}\chi\le\log K.\label{eq:middletermOK}
\end{equation}

Now note that $[0,S_{0}]\subset\mathcal{S}_{\chi}$ by \eqref{chidef}
and the definition of $\mathcal{S}_{\chi}$ in \defref{chidefs}.
Suppose that $S$ is such that $KS\ge S_{0}$ and $[0,KS/100]\subset\mathcal{S}_{\chi}$.
By \eqref{S0satisfies}, we have that
\[
K^{\eta}S=K^{\eta-1}KS\ge K^{\eta-1}S_{0}\ge T_{0},
\]
so by plugging in \eqref{firsttermOK} and \eqref{middletermOK} into
\eqref{finalvariancebound-1}, we obtain
\[
\Var\left(\log d_{\mathbb{R}}(\mathrm{L},\mathrm{R})\right)\le Z\left(\log K+K^{-z}\chi^{3}+K^{Z}S^{-z}\right)\le Z\left(\log K+K^{-z}\chi^{3}+K^{Z}S^{-z}\right)\le3Z\log K.
\]
Therefore, by \eqref{chiUvarbound-application}, we have
\[
\chi_{KS}\le Z\e^{Z\sqrt{3Z\log K}}\le\chi,
\]
so $KS\in\mathcal{S}_{\chi}$ as well. By induction, this implies
that $\mathcal{S}_{\chi}=[0,\infty)$. From this, \eqref{finalvariancebound-1}
implies that the variance $\Var(\log d_{\mathbb{R}}(\mathrm{L},\mathrm{R}))$
is bounded uniformly in $S$, hence the result.
\end{proof}
Thus it remains to prove \lemref{finalvariancebound} by bounding
the terms of \eqref{efronstein-readyforlemmas}. That will be our
goal in the following subsections.

\subsubsection{The coarse field effect\label{subsec:coarsefieldeffect}}

First we address the effect of the coarse field (the first term of
\eqref{efronstein-readyforlemmas}) using Gaussian concentration.
\begin{lem}
\label{lem:coarsefieldeffect}We have a constant $C$ so that
\begin{equation}
\mathbf{E}\left[\left(\log\frac{\tilde{d}_{\mathbb{R},1,S}(\mathrm{L},\mathrm{R})}{d_{\mathbb{R},1,S}(\mathrm{L},\mathrm{R})}\right)^{2}\right]\le C\log K.\label{eq:coarseexpr}
\end{equation}
\end{lem}

\begin{proof}
Let $\mathcal{F}_{S^{2}}$ denote the $\sigma$-algebra generated
by the white noise on $\mathbb{R}^{*}\times[0,S^{2}]$. We note that
\begin{equation}
\frac{1}{2}\mathbf{E}\left[\left(\log\frac{\tilde{d}_{\mathbb{R},1,S}(\mathrm{L},\mathrm{R})}{d_{\mathbb{R},1,S}(\mathrm{L},\mathrm{R})}\right)^{2}\right]=\mathbf{E}\Var(\log d_{\mathbb{R},1,S}(\mathrm{L},\mathrm{R})\mid\mathcal{F}_{S^{2}}).\label{eq:condvar}
\end{equation}
Let
\[
h_{\mathrm{coarse}}(x)=\sqrt{\pi}\int_{S^{2}}^{\infty}\int_{\mathbb{R}^{*}}p_{t/2}^{\mathbb{R}^{*}}(x,y)W(\dif y\,\dif t).
\]
Now we claim that there is a constant $C<\infty$ so that, for all
$x\in\mathbb{R}^{\circ}$, we have
\begin{equation}
\sup_{x\in\mathbb{R}^{\circ}}\Var(h_{\mathrm{coarse}}(x))\le C\log K.\label{eq:coarsevarbound}
\end{equation}
This is because we can write 
\[
h_{\mathrm{coarse}}(x)=\int_{\mathbb{R}^{*}}p_{S^{2}/2}^{\mathbb{R}^{*}}(x,y)h_{\mathbb{R}^{*}}(y)\,\dif y;
\]
from this is it is clear that the pointwise variance of $h_{\mathrm{coarse}}$
is finite, and \eqref{coarsevarbound} comes from rescaling so that
$S=1$ and then using \eqref{greensfunctionest}.

Since $\log d_{\mathbb{R},1,S}$ is measurable with respect to the
$\sigma$-algebra generated by $h_{\mathrm{fine}}$ and $h_{\mathrm{coarse}}$,
there is a deterministic functional $\widetilde{\mathscr{D}}$ so
that 
\[
\log d_{\mathbb{R},1,S}(\mathrm{L},\mathrm{R})=\widetilde{\mathscr{D}}(W|_{\mathbb{R}^{*}\times[0,S^{2}]},h_{\mathrm{coarse}})
\]
with probability $1$. Let $\mathscr{D}$ be an $\mathcal{F}_{S^{2}}$-measurable
random functional on $L^{\infty}(\mathbb{R}^{\circ})$ given by
\[
\mathscr{D}(h)=\widetilde{\mathscr{D}}(W|_{\mathbb{R}^{*}\times[0,S^{2}]},h).
\]
We note that, with probability $1$, $\mathscr{D}$ is a $\gamma$-Lipschitz
functional on $L^{\infty}(\mathbb{R}^{\circ})$. This is to say that
we have, for all $\Delta\in L^{\infty}(\mathbb{R}^{\circ})$, that
\[
\left|\mathscr{D}(h_{\mathrm{coarse}}+\Delta)-\mathscr{D}(h_{\mathrm{coarse}})\right|\le\gamma\|\Delta\|_{L^{\infty}(\mathbb{R}^{\circ})}
\]
by \eqref{measure-comparison} and \eqref{mostusefulscaling}.

Let $Q$ be the median value of $\mathscr{D}(h_{\mathrm{coarse}})$,
conditional on $\mathcal{F}_{S^{2}}$, so $Q$ is an $\mathcal{F}_{S^{2}}$-measurable
random variable so that 
\[
\mathbf{P}\left(\mathscr{D}(h_{\mathrm{coarse}})\le Q\mid\mathcal{F}_{S^{2}}\right),\mathbf{P}\left(\mathscr{D}(h_{\mathrm{coarse}})\ge Q\mid\mathcal{F}_{S^{2}}\right)\ge\frac{1}{2}
\]
almost surely. Let $\mathcal{X}=\{h\in L^{\infty}(\mathbb{R}^{\circ})\::\;\mathscr{D}(h_{\mathrm{coarse}})\le Q\}$,
so $\mathbf{P}(\mathscr{D}(h_{\mathrm{coarse}})\in\mathcal{X}\mid\mathcal{F}_{S^{2}})\ge\frac{1}{2}$
almost surely. We can then apply (an infinite-dimensional version
of) the Gaussian concentration inequality given in \cite[Lemma 2.1]{DZZ18},
along with \eqref{coarsevarbound}, to observe that
\[
\mathbf{P}\left(\log d_{\mathbb{R},1,S}(\mathrm{L},\mathrm{R})\ge Q+\lambda\mid\mathcal{F}_{S^{2}}\right)\le\mathbf{P}\left(\min_{h\in\mathcal{X}}\|h_{\mathrm{coarse}}-h\|_{L^{\infty}(\mathbb{R}^{\circ})}\ge\gamma^{-1}\lambda\;\middle|\;\mathcal{F}_{S^{2}}\right)\le C\exp\left\{ -\frac{(\lambda-C\sqrt{\log K})^{2}}{C\log K}\right\} 
\]
almost surely. A similar argument implies that
\[
\mathbf{P}\left(\log d_{\mathbb{R},1,S}(\mathrm{L},\mathrm{R})\le X-\lambda\mid\mathcal{F}_{S^{2}}\right)\le C\exp\left\{ -\frac{(\lambda-C\sqrt{\log K})^{2}}{C\log K}\right\} 
\]
almost surely. This implies that $\Var\left(\log d_{\mathbb{R},1,S}(\mathrm{L},\mathrm{R})\mid\mathcal{F}_{S^{2}}\right)\le C\log K$
almost surely, which means that 
\[
\mathbf{E}\Var\left(\log d_{\mathbb{R},1,S}(\mathrm{L},\mathrm{R})\mid\mathcal{F}_{S^{2}}\right)\le C\log K,
\]
implying \eqref{coarseexpr} by \eqref{condvar}.
\end{proof}

\subsubsection{The far fine field effect\label{subsec:farfinefield}}

Now we turn to the second term of \eqref{efronstein-readyforlemmas}.
First we deal with the part of the path which is ``far away'' from
the white noise being resampled. In this case we don't need to change
the path when we resample to get our bound. Rather, we simply bound
the increase in the weight of the path by the maximum of the increase
in the LQG measure. The bound we obtain in \lemref{farawaystuff}
is in terms of the maximum annular circuit distance and the total
crossing distance. These terms will also appear in the close fine
field bound in the next section, so we wait to bound them together
in \subsecref{maxsmall}.
\begin{lem}
\label{lem:Gbounds}For every $\lambda>0$, there is a $C<\infty$
so that, for all $i\in\mathcal{Q}_{*}$ and $j\in\mathcal{D}(i)$,
we have
\begin{equation}
\mathbf{E}\exp\left\{ \lambda\|h_{\mathbb{R}^{*}}-h_{\mathbb{R}^{*}}^{(i)}\|_{L^{\infty}(\mathbb{C}_{j}^{\circ})}\right\} \le C.\label{eq:Gijmoment}
\end{equation}
Moreover, if we define
\begin{equation}
G^{*}=\max_{\substack{i\in\mathcal{Q}_{*}\\
j\in\mathcal{D}(i)
}
}\left(\|h_{\mathbb{R}^{*}}-h_{\mathbb{R}^{*}}^{(i)}\|_{L^{\infty}(\mathbb{C}_{j}^{\circ})}\exp\left\{ \frac{\dist_{\Euc}(\mathbb{C}_{i},\mathbb{C}_{j}^{\circ})^{2}}{CS^{2}}\right\} \right)\label{eq:Gstardef}
\end{equation}
(where $\mathcal{D}(i)$ is defined as in \eqref{Didef}) then we
have
\begin{equation}
\mathbf{E}\exp\{\lambda G^{*}\}\le C\e^{C\sqrt{\log K}}.\label{eq:Gstarmoment}
\end{equation}
\end{lem}

\begin{proof}
Let $j\in\mathcal{D}(i)$. Define $D=\dist_{\Euc}(\mathbb{C}_{i},\mathbb{C}_{j}^{\circ})$
and note that $D>0$ by \eqref{Didef}. We have, a constant $C$,
independent of $i,j,D$, so that for $x\in\mathbb{C}_{j}^{\circ}$,
\begin{align}
\Var & \left(h_{\mathbb{R}^{*}}^{(i)}(x)-h_{\mathbb{R}^{*}}(x)\right)=2\pi\int_{0}^{S^{2}}\int_{\mathbb{C}_{i}}|p_{t/2}^{\mathbb{R}^{*}}(x,y)|^{2}\,\dif y\,\dif t\le C\int_{0}^{S^{2}}\int_{\mathbb{C}^{2}}t^{-2}\exp\left\{ -\frac{|x-y|^{2}}{Ct}\right\} \,\dif y\,\dif t\nonumber \\
 & \le C\int_{\mathbb{C}_{i}}\int_{0}^{S^{2}|x-y|^{-2}}|x-y|^{-2}t^{-2}\e^{-1/(Ct)}\,\dif t\,\dif y\le CD^{-2}S^{2}\int_{0}^{S^{2}D^{-2}}t^{-2}\e^{-1/(Ct)}\,\dif t\le C\e^{-(D/S)^{2}/C}.\label{eq:varbound}
\end{align}
Also, for $x,z\in\mathbb{C}_{j}^{\circ}$, using \lemref{gradientbound},
we see that
\begin{align}
\Var & \left(\left(h_{\mathbb{R}^{*}}^{(i)}(x)-h_{\mathbb{R}^{*}}(x)\right)-\left(h_{\mathbb{R}^{*}}^{(i)}(z)-h_{\mathbb{R}^{*}}(z)\right)\right)=2\pi\int_{0}^{S^{2}}\int_{\mathbb{C}_{i}}|p_{t/2}^{\mathbb{R}^{*}}(x,y)-p_{t/2}^{\mathbb{R}^{*}}(z,y)|^{2}\,\dif y\,\dif t\nonumber \\
 & \le2\pi\int_{0}^{S^{2}}\int_{\mathbb{C}_{i}}|x-z|^{2}\|\nabla p_{t/2}^{\mathbb{R}^{*}}(\cdot,y)\|_{L^{\infty}(\mathbb{C}_{j}^{\circ})}^{2}\,\dif y\,\dif t\le C|x-z|^{2}S^{2}\int_{0}^{S^{2}}\frac{D^{2}}{t^{4}}\exp\left\{ -\frac{2D^{2}}{t}\right\} \,\dif t\nonumber \\
 & =C|x-z|^{2}S^{2}\int_{0}^{S^{2}}\frac{D^{2}}{t^{4}D^{8}}\exp\left\{ -\frac{2D^{2}}{tD^{2}}\right\} \,\dif(D^{2}t)=C|x-z|^{2}\frac{S^{2}}{D^{4}}\int_{0}^{S^{2}D^{-2}}\e^{-2/t}\,\dif t\nonumber \\
 & =C\left|\frac{x-z}{S}\right|^{2}\frac{S^{4}}{D^{4}}\e^{-(D/S)^{2}/C}.\label{eq:vardiffbd}
\end{align}
Fernique's inequality (\cite{Fer75}; see e.g. \cite[Theorem 4.1]{A90})
in light of \eqref{vardiffbd} then implies that 
\[
\mathbf{E}\sup_{x\in\mathbb{C}_{j}^{\circ}}\left(h_{\mathbb{R}^{*}}^{(i)}(x)-h_{R}(x)\right)\le C\e^{-(D/S)^{2}/C}=C\exp\left\{ -\frac{\dist_{\Euc}(\mathbb{C}_{i},\mathbb{C}_{j}^{\circ})^{2}}{CS^{2}}\right\} .
\]
Thus, by the Borell--TIS inequality (see e.g. \cite[Theorem 7.1]{L01},
\cite[Theorem 6.1]{biskup}, or \cite[Theorem 2.1]{A90}) and \eqref{varbound},
we have, for all $u\ge0$, that
\[
\mathbf{P}\left(\sup_{x\in\mathbb{C}_{j}^{\circ}}\left(h_{\mathbb{R}^{*}}^{(i)}(x)-h_{R}(x)\right)\exp\left\{ \frac{\dist_{\Euc}(\mathbb{C}_{i},\mathbb{C}_{j}^{\circ})^{2}}{CS^{2}}\right\} \ge C+u\right)\le\e^{-u^{2}}.
\]
Applying the same argument to $-\left(h_{\mathbb{R}^{*}}^{(i)}(x)-h_{R}(x)\right)$
and taking a union bound yields
\begin{equation}
\mathbf{P}\left(\|h_{\mathbb{R}^{*}}^{(i)}-h_{R}\|_{L^{\infty}(\mathbb{C}_{j}^{\circ})}\exp\left\{ \frac{\dist_{\Euc}(\mathbb{C}_{i},\mathbb{C}_{j}^{\circ})^{2}}{CS^{2}}\right\} \ge C+u\right)\le\e^{-u^{2}}.\label{eq:finallinftybd}
\end{equation}
Dropping the $\exp\left\{ \frac{\dist_{\Euc}(\mathbb{C}_{i},\mathbb{C}_{j}^{\circ})^{2}}{CS^{2}}\right\} $
in \eqref{finallinftybd} yields \eqref{Gijmoment}. On the other
hand, taking a union bound over all $i\in\mathcal{Q}_{*},j\in\mathcal{D}(i)$,
and noting that the such of all such $i,j$ has size of order $K^{4}$,
we obtain \eqref{Gstarmoment}.
\end{proof}
\begin{lem}
\label{lem:farawaystuff}There is a constant $C$ so that
\begin{equation}
\mathbf{E}\sum_{i\in\mathcal{Q}_{*}}\left(\left(\frac{d_{\mathbb{R},1,S}^{(i)}(\pi_{i}^{-})}{d_{\mathbb{R},1,S}(\pi)}-1\right)^{+}\right)^{2}\mathbf{1}_{\mathcal{A}_{i}}\le C\e^{C\sqrt{\log K}}\left(\mathbf{E}\max_{j\in\mathcal{Q}_{\circ}}d_{(3\mathbb{E}_{i})^{*},\mathbb{A}_{j},K^{-\gamma\theta_{0}},S}(\mathrm{around})^{3}\right)^{1/3}\left(\mathbf{E}d_{\mathbb{R},1,S}(\pi)^{-3}\right)^{1/3}.\label{eq:farawaystuff}
\end{equation}
\end{lem}

\begin{proof}
Note that by \eqref{piiminusdef}, \eqref{triangleineq}, and \eqref{ddi}
of \propref{fieldonsmallerbox}, we have
\begin{align*}
d_{\mathbb{R},1,S}^{(i)}(\pi_{i}^{-})\mathbf{1}_{\mathcal{A}_{i}} & \le\sum_{j\in\mathcal{D}(i)}d_{\mathbb{R},1,S}^{(i)}(\pi_{i}^{-}\cap\mathbb{C}_{j})\mathbf{1}_{\mathcal{A}_{2;i}}\le\mathbf{1}_{\mathcal{A}_{2;i}}\sum_{j\in\mathcal{D}(i)}\exp\left\{ \gamma\|h_{\mathbb{R}^{*}}^{(i)}-h_{\mathbb{R}^{*}}\|_{L^{\infty}(\mathbb{C}_{j}^{\circ})}\right\} d_{\mathbb{R},1,S}(\pi\cap\mathbb{C}_{j}).
\end{align*}
Thus we obtain
\begin{align}
 & \sum_{i\in\mathcal{Q}_{*}}\left(\left(\frac{d_{\mathbb{R},1,S}^{(i)}(\pi_{i}^{-})}{d_{\mathbb{R},1,S}(\pi)}-1\right)^{+}\right)^{2}\mathbf{1}_{\mathcal{A}_{i}}\nonumber \\
 & \le d_{\mathbb{R},1,S}(\pi)^{-2}\sum_{i\in\mathcal{Q}_{*}}\left(\sum_{j\in\mathcal{D}(i)}\left(\exp\left\{ \gamma\|h_{\mathbb{R}^{*}}-h_{\mathbb{R}^{*}}^{(i)}\|_{L^{\infty}(\mathbb{C}_{j}^{\circ})}\right\} -1\right)d_{\mathbb{R},1,S}(\pi\cap\mathbb{C}_{j})\right)^{2}\mathbf{1}_{\mathcal{A}_{i}}\nonumber \\
 & \ =d_{\mathbb{R},1,S}(\pi)^{-2}\sum_{j_{1},j_{2}\in\mathcal{Q}_{\circ}}\left[d_{\mathbb{R},1,S}(\pi\cap\mathbb{C}_{j_{1}})d_{\mathbb{R},1,S}(\pi\cap\mathbb{C}_{j_{2}})\sum_{\substack{i\in\mathcal{Q}_{*}\\
\mathcal{D}(i)\ni j_{1},j_{2}
}
}\mathbf{1}_{\mathcal{A}_{i}}\prod_{j\in\{j_{1},j_{2}\}}\left(\exp\left\{ \gamma\|h_{\mathbb{R}^{*}}-h_{\mathbb{R}^{*}}^{(i)}\|_{L^{\infty}(\mathbb{C}_{j}^{\circ})}\right\} -1\right)\right]\nonumber \\
 & \ \le d_{\mathbb{R},1,S}(\pi)^{-2}\left(\sum_{j\in\mathcal{Q}_{\circ}}d_{\mathbb{R},1,S}(\pi\cap\mathbb{C}_{j})\right)\cdot\nonumber \\
 & \quad\qquad\cdot\max_{j_{2}\in\mathcal{Q}_{\circ}}\left[\sum_{j_{1}\in\mathcal{Q}_{\circ}}\left(d_{\mathbb{R},1,S}(\pi\cap\mathbb{C}_{j_{1}})\sum_{\substack{i\in\mathcal{Q}_{*}\\
\mathcal{D}(i)\ni j_{1},j_{2}
}
}\mathbf{1}_{\mathcal{A}_{i}}\prod_{j\in\{j_{1},j_{2}\}}\left(\exp\left\{ \gamma\|h_{\mathbb{R}^{*}}-h_{\mathbb{R}^{*}}^{(i)}\|_{L^{\infty}(\mathbb{C}_{j}^{\circ})}\right\} -1\right)\right)\right],\label{eq:expandthesquare}
\end{align}
where the equality comes from expanding out the square and switching
the order of summation, and the inequality is the $\ell^{1}$-$\ell^{\infty}$
Hölder inequality. Now we have that
\begin{equation}
\sum_{j\in\mathcal{Q}_{\circ}}d_{\mathbb{R},1,S}(\pi\cap\mathbb{C}_{j})\le\sum_{j\in\mathcal{Q}_{\circ}}\sum_{\substack{1\le m\le M\\
\pi_{m}\cap\mathbb{C}_{j}\ne\emptyset
}
}d_{\mathbb{R},1,S}(\pi_{m})\le\sum_{1\le m\le M}\#\{j\in\mathcal{Q}_{\circ}\;:\;\pi_{m}\cap\mathbb{C}_{j}\ne\emptyset\}\cdot d_{\mathbb{R},1,S}(\pi_{m})\le Cd_{\mathbb{R},1,S}(\pi),\label{eq:sumisthesum}
\end{equation}
in which the last inequality is by \eqref{jpiCj} and \eqref{splitisgood}.
Also, we have
\begin{equation}
\mathbf{1}_{\mathcal{A}_{i}}d_{\mathbb{R},1,S}(\pi\cap\mathbb{C}_{j_{1}})\le\mathbf{1}_{\mathcal{A}_{i}}d_{\mathbb{R}^{*},\mathbb{A}_{j_{1}},1,S}(\mathrm{around})\le\exp\left\{ \gamma\frac{\dist(\mathbb{C}_{i},\mathbb{C}_{j_{1}}^{\circ})}{S}\right\} d_{(3\mathbb{E}_{i})^{*},\mathbb{A}_{j_{1}},K^{-\gamma\theta_{0}},S}(\mathrm{around}).\label{eq:boundthesmallerpart}
\end{equation}
Here, the first inequality is because the part of $\pi$ intersecting
$\mathbb{C}_{j}$ could be replaced by a circuit around the annulus
$\mathbb{A}_{j}$ to produce a new crossing, so $\pi\cap\mathbb{C}_{j}$
must have a smaller LGD length than the annular circuit around $\mathbb{A}_{j}$
by the optimality of $\pi$. (A similar, but not quite identical,
argument was used to derive \eqref{replacepimbyanannulus}.) The second
inequality in \eqref{boundthesmallerpart} is by applying \eqref{mostusefulscaling}
and \eqref{scaledeltabymax} using the coarse field bound \eqref{A2idef}.

Plugging \eqref{sumisthesum} and \eqref{boundthesmallerpart} into
the right-hand side of \eqref{expandthesquare}, we obtain
\begin{multline}
\sum_{i\in\mathcal{Q}_{*}}\left(\left(\frac{d_{\mathbb{R},1,S}^{(i)}(\pi_{i}^{-})}{d_{\mathbb{R},1,S}(\pi)}-1\right)^{+}\right)^{2}\mathbf{1}_{\mathcal{A}_{i}}\le\left(\frac{\max\limits _{j\in\mathcal{Q}_{\circ}}d_{(3\mathbb{E}_{i})^{*},\mathbb{A}_{j_{1}},K^{-\gamma\theta_{0}},S}(\mathrm{around})}{d_{\mathbb{R},1,S}(\pi)}\right)\cdot\\
\cdot\max_{j_{2}\in\mathcal{Q}_{\circ}}\sum_{\substack{j_{1}\in\mathcal{Q}_{\circ}\\
i\in\mathcal{Q}_{*}\\
\mathcal{D}(i)\ni j_{1},j_{2}
}
}\exp\left\{ \gamma\frac{\dist(\mathbb{C}_{i},\mathbb{C}_{j_{1}}^{\circ})}{S}\right\} \prod_{j\in\{j_{1},j_{2}\}}\left(\exp\left\{ \gamma\|h_{\mathbb{R}^{*}}-h_{\mathbb{R}^{*}}^{(i)}\|_{L^{\infty}(\mathbb{C}_{j}^{\circ})}\right\} -1\right).\label{eq:plugthingsin}
\end{multline}
We can bound the sum on the right-hand side by
\begin{align}
\sum_{\substack{j_{1}\in\mathcal{Q}_{\circ}\\
i\in\mathcal{Q}_{*}\\
\mathcal{D}(i)\ni j_{1},j_{2}
}
} & \exp\left\{ \gamma\frac{\dist(\mathbb{C}_{i},\mathbb{C}_{j_{1}}^{\circ})}{S}\right\} \prod_{j\in\{j_{1},j_{2}\}}\left(\exp\left\{ \gamma\|h_{\mathbb{R}^{*}}-h_{\mathbb{R}^{*}}^{(i)}\|_{L^{\infty}(\mathbb{C}_{j}^{\circ})}\right\} -1\right)\nonumber \\
 & \le\sum_{\substack{j_{1}\in\mathcal{Q}_{\circ}\\
i\in\mathcal{Q}_{*}\\
\mathcal{D}(i)\ni j_{1},j_{2}
}
}\exp\left\{ \gamma\frac{\dist(\mathbb{C}_{i},\mathbb{C}_{j_{1}}^{\circ})}{S}\right\} \prod_{j\in\{j_{1},j_{2}\}}\left(\exp\left\{ \gamma G^{*}\exp\left\{ -\frac{\dist_{\Euc}(\mathbb{C}_{i},\mathbb{C}_{j}^{\circ})^{2}}{CS^{2}}\right\} \right\} -1\right)\le\e^{C\gamma G^{*}},\label{eq:boundthesum}
\end{align}
where $G^{*}$ is as in \eqref{Gstardef} and the last inequality
is a simple estimate on the sum. Plugging \eqref{boundthesum} into
\eqref{plugthingsin}, we and applying Hölder's inequality and \lemref{Gbounds},
we obtain \eqref{farawaystuff}.
\end{proof}

\subsubsection{The close fine field effect\label{subsec:closefinefield}}

Now we deal with the part of the path that is close to the part of
the white noise being resampled: the last term of \eqref{efronstein-readyforlemmas}.
Here, when we resample, we replace the part of the path that is close
to the resampled box by a circuit around an annulus surrounding the
box. The ``fine field'' has been totally changed in the resampling,
so there is nothing to be gained by keeping track of it. However,
the ``coarse field'' should only be changed by a small, smooth difference,
and we want to keep track of this so that the sum of all the changes
will be bounded by the total weight of the original path. Thus we
define the notation (recalling that $\theta_{0}$ was fixed in \eqref{etadef})
\begin{gather}
\begin{aligned}F_{i}^{\uparrow} & =\max\limits _{x\in\mathbb{E}_{i}^{\circ}}h_{\mathbb{R}^{*}:\mathbb{E}_{i}^{*}}^{(i)}(x), & F_{i}^{\downarrow} & =\min\limits _{x\in\mathbb{E}_{i}^{\circ}}h_{\mathbb{R}^{*}:\mathbb{E}_{i}^{*}}(x),\end{aligned}
\label{eq:Fidef}
\end{gather}

The difference between $F_{i}^{\uparrow}$ and $F_{i}^{\downarrow}$
represents the ``change in the coarse field,'' and it will be of
order $1$, as we show in the next two lemmas.
\begin{lem}
\label{lem:coarsechange}For every $\lambda>0$, we have a constant
$C<\infty$, depending on $\lambda$, so that
\begin{equation}
\mathbf{E}\exp\left\{ \lambda\sup_{x\in\mathbb{E}_{i}^{\circ}}|h_{\mathbb{R}^{*}:\mathbb{E}_{i}^{*}}^{(i)}(x)-h_{\mathbb{R}^{*}:\mathbb{E}_{i}^{*}}(x)|\right\} \le C.\label{eq:coarsechange}
\end{equation}
\end{lem}

\begin{proof}
Inequality \eqref{Gijmoment} of \lemref{Gbounds} implies that for
any $\lambda>0$ there is a $C<\infty$, depending only on $\lambda$,
so that
\[
\mathbf{E}\exp\left\{ \lambda\sup_{x\in\mathbb{E}_{i}^{*}\setminus\mathbb{E}_{i}^{\circ}}|h_{\mathbb{R}^{*}}^{(i)}(x)-h_{\mathbb{R}^{*}}(x)|\right\} \le C.
\]
By the maximum principle for harmonic functions, we have that
\[
\sup_{x\in\mathbb{E}_{i}^{\circ}}|h_{\mathbb{R}^{*}:\mathbb{E}_{i}^{*}}^{(i)}(x)-h_{\mathbb{R}^{*}:\mathbb{E}_{i}^{*}}(x)|\le\sup_{x\in\mathbb{E}_{i}^{*}\setminus\mathbb{E}_{i}^{\circ}}|h_{\mathbb{R}^{*}}^{(i)}(x)-h_{\mathbb{R}^{*}}(x)|,
\]
by which we obtain \eqref{coarsechange}.
\end{proof}
\begin{lem}
\label{lem:coarsefieldsmooth}For any $B>0$ and $A<\infty$ we have
a $C>0$ so that
\begin{equation}
\mathbf{E}\e^{\gamma B(F_{i}^{\uparrow}-F_{i}^{\downarrow})}\le CK^{1/A}.\label{eq:ffexpmoment}
\end{equation}
\end{lem}

\begin{proof}
We have that
\begin{align}
F_{i}^{\uparrow}-F_{i}^{\downarrow} & =\max_{x\in\mathbb{E}_{i}^{\circ}}h_{\mathbb{R}^{*}:\mathbb{E}_{i}^{*}}^{(i)}(x)-\min_{x\in\mathbb{E}_{i}^{\circ}}h_{\mathbb{R}^{*}:\mathbb{E}_{i}^{*}}(x)=\max_{x,y\in\mathbb{E}_{i}^{\circ}}(h_{\mathbb{R}^{*}:\mathbb{E}_{i}^{\circ}}^{(i)}(x)-h_{\mathbb{R}^{*}:\mathbb{E}_{i}^{*}}(y))\nonumber \\
 & \le\max_{x,y\in\mathbb{E}_{i}^{\circ}}|h_{\mathbb{R}^{*}:\mathbb{E}_{i}^{*}}(x)-h_{\mathbb{R}^{*}:\mathbb{E}_{i}^{*}}(y)|+\max_{x\in\mathbb{E}_{i}^{\circ}}|h_{\mathbb{R}^{*}:\mathbb{E}_{i}^{*}}^{(i)}(x)-h_{\mathbb{R}^{*}:\mathbb{E}_{i}^{*}}(x)|.\label{eq:FplusFminusbreakup}
\end{align}
To bound the first term of \eqref{FplusFminusbreakup}, we note that
by \lemref{fluctstailbound}, we have for any $\lambda>0$ that there
is a $C$ depending on $\lambda$ so that
\[
\mathbf{E}\exp\left\{ \lambda\max_{x,y\in\mathbb{E}_{i}^{\circ}}|h_{\mathbb{R}^{*}:\mathbb{E}_{i}^{*}}(x)-h_{\mathbb{R}^{*}:\mathbb{E}_{i}^{*}}(y)|\right\} \le C.
\]
Therefore, by Jensen's inequality, we have for any $A>0$ that
\begin{align}
\mathbf{E}\exp & \left\{ \lambda\max_{i\in\mathcal{Q}_{*}}\max_{x,y\in\mathbb{E}_{i}^{\circ}}|h_{\mathbb{R}^{*}:\mathbb{E}_{i}^{*}}(x)-h_{\mathbb{R}^{*}:\mathbb{E}_{i}^{*}}(y)|\right\} \le\left(\mathbf{E}\exp\left\{ A\lambda\max_{i\in\mathcal{Q}_{*}}\max_{x,y\in\mathbb{E}_{i}^{\circ}}|h_{\mathbb{R}^{*}:\mathbb{E}_{i}^{*}}(x)-h_{\mathbb{R}^{*}:\mathbb{E}_{i}^{*}}(y)|\right\} \right)^{1/A}\nonumber \\
 & \le\left(\sum_{i\in\mathcal{Q}_{*}}\mathbf{E}\exp\left\{ A\lambda\max_{x,y\in\mathbb{E}_{i}^{\circ}}|h_{\mathbb{R}^{*}:\mathbb{E}_{i}^{\circ}}(x)-h_{\mathbb{R}^{*}:\mathbb{E}_{i}^{\circ}}(y)|\right\} \right)^{1/A}\le C|\mathcal{Q}_{*}|^{1/A}\le CK^{2/A},\label{eq:fluctsJensen}
\end{align}
for some $C$ depending on $A$ and $\lambda$. Considering the second
term of \eqref{FplusFminusbreakup}, we note that, by \lemref{coarsechange},
we have 
\begin{equation}
\mathbf{E}\exp\left\{ \lambda\max_{x\in\mathbb{E}_{i}^{\circ}}|h_{\mathbb{R}^{*}:\mathbb{E}_{i}^{*}}^{(i)}(x)-h_{\mathbb{R}^{*}:\mathbb{E}_{i}^{*}}(x)|\right\} \le C\label{eq:coarsechange-pply}
\end{equation}
for some $C$ depending on $\lambda$. Bounding the exponential moments
of the first and second terms of \eqref{FplusFminusbreakup} using
\eqref{fluctsJensen} and \eqref{coarsechange-pply} respectively,
and then using Young's inequality, yields \eqref{ffexpmoment}.
\end{proof}
Now we define
\begin{align*}
H_{i} & =\frac{d_{(3\mathbb{E}_{i})^{*},\mathbb{A}_{i},\e^{-\gamma F_{i}^{\uparrow}},S}^{(i)}(\mathrm{around})}{d_{\mathbb{E}_{i},\e^{-\gamma F_{i}^{\downarrow}},S}(\min;S/\diam_{\Euc}(\mathbb{E}_{i})}, & H & =\max_{i\in\mathcal{Q}_{\circ}}H_{i},
\end{align*}
representing the cost of replacing a geodesic across a box with respect
the original fine field with a circuit around the box with respect
to the resampled fine field. The choices of $\delta$ (as $\e^{-\gamma F_{i}^{\uparrow}}$
and $\e^{-\gamma F_{i}^{\downarrow}}$) in the definition of $H_{i}$
represent the worst-case contributions of the coarse field, which
we bounded in the previous lemma. We now bound the sum of all of the
``replacement'' annuli, divided by the total crossing length, in
terms of $H$.
\begin{lem}
\label{lem:boundintermsofH}There is a constant $C<\infty$ so that
\begin{equation}
\frac{\sum\limits _{\substack{i\in\mathcal{Q}_{*}\\
\mathcal{M}(i)\ne\emptyset
}
}d_{\mathbb{R}^{*},\mathbb{A}_{i},1,S}^{(i)}(\mathrm{around})}{d_{\mathbb{R},1,S}(\mathrm{L},\mathrm{R})}\le CH.\label{eq:sumestimate}
\end{equation}
\end{lem}

\begin{proof}
By \eqref{scaledeltabymax} and \eqref{ddi} of \propref{fieldonsmallerbox},
we have
\begin{align*}
d_{\mathbb{R}^{*},\mathbb{A}_{i},1,S}^{(i)}(\mathrm{around}) & \le d_{(3\mathbb{E}_{i})^{*},\mathbb{A}_{i},\e^{-\gamma F_{i}^{\uparrow}},S}^{(i)}(\mathrm{around}),\\
d_{\mathbb{R}^{*},\mathbb{E}_{i},1,S}(x_{m-1},x_{m}) & \ge d_{\mathbb{E}_{i},\e^{-\gamma F_{i}^{\downarrow}},S}(x_{m-1},x_{m}),
\end{align*}
where we recall that $x_{m-1},x_{m}$ are the endpoints of $\pi_{m}$
as defined in \subsecref{splitting}. Thus we have
\begin{align}
d_{\mathbb{R}^{*},\mathbb{A}_{i},1,S}^{(i)}(\mathrm{around})\le d_{(3\mathbb{E}_{i})^{*},\mathbb{A}_{i},\e^{-\gamma F_{i}^{\uparrow}},S}^{(i)}(\mathrm{around}) & \le H_{i}d_{\mathbb{E}_{i},\e^{-\gamma F_{i}^{\downarrow}},S}(x_{m-1},x_{m})\nonumber \\
 & \le H_{i}d_{\mathbb{R}^{*},\mathbb{E}_{i},1,S}(x_{m-1},x_{m}),\label{eq:diratiobound-1}
\end{align}
where the last inequality is by \eqref{scaledeltabymax} of \propref{fieldonsmallerbox}.
Also, we have a constant $C$ so that 
\begin{equation}
\sum_{\substack{i\in\mathcal{Q}_{*}\\
m\in\mathcal{M}(i)
}
}d_{\mathbb{R}^{*},\mathbb{E}_{i},1,S}(x_{m-1},x_{m})=\sum_{\substack{i\in\mathcal{Q}_{*}\\
m\in\mathcal{M}(i)
}
}d_{\mathbb{R},1,S}(x_{m-1},x_{m})\le C\sum_{m=1}^{M}d_{\mathbb{R},1,S}(x_{m-1},x_{m})=Cd_{\mathbb{R},1,S}(\mathrm{L},\mathrm{R}),\label{eq:addupswitches}
\end{equation}
where the first equality is by \eqref{pathstaysinEi}, the inequality
is by \eqref{iMsize}, and the second equality is by \eqref{splitisgood}.
Therefore, for each $i\in\mathcal{Q}$, if we set $m(i)=\min\mathcal{M}(i)$
(or just an arbitrary element of $\mathcal{M}(i)$ picked in some
other way), then we have
\begin{align*}
\sum_{\substack{i\in\mathcal{Q}_{*}\\
\mathcal{M}(i)\ne\emptyset
}
}\frac{d_{\mathbb{R}^{*},\mathbb{A}_{i},1,S}^{(i)}(\mathrm{around})}{d_{\mathbb{R},1,S}(\mathrm{L},\mathrm{R})}\le H\sum_{i\in\mathcal{Q}_{*}}\frac{d_{\mathbb{R}^{*},\mathbb{E}_{i},1,S}(x_{m(i)-1},x_{m(i)})}{d_{\mathbb{R},1,S}(\mathrm{L},\mathrm{R})} & \le H\frac{\sum\limits _{i\in\mathcal{Q}_{*}}\sum\limits _{m\in\mathcal{M}(i)}d_{\mathbb{R}^{*},\mathbb{E}_{i},1,S}(x_{m-1},x_{m})}{d_{\mathbb{R},1,S}(\mathrm{L},\mathrm{R})}\le CH,
\end{align*}
where the first inequality is by \eqref{diratiobound-1} and the third
is by \eqref{addupswitches}. This proves \eqref{sumestimate}.
\end{proof}
\begin{lem}
\label{lem:sumconclusion}For any $B<\infty$ and $A<\infty$ we have
a $C<\infty$ so that
\begin{equation}
\mathbf{E}\left[\left(\sum_{\substack{i\in\mathcal{Q}_{*}\\
\mathcal{M}(i)\ne\emptyset
}
}\frac{\mathbf{1}_{\mathcal{A}_{i}}d_{\mathbb{R}^{*},\mathbb{A}_{i},1,S}^{(i)}(\mathrm{around})}{d_{\mathbb{R},1,S}(\mathrm{L},\mathrm{R})}\right)^{B}\right]\le CK^{2}\overline{\chi}_{K^{\eta}S}^{B}.\label{eq:sumconclusion}
\end{equation}
\end{lem}

\begin{proof}
First we fix $i\in\mathcal{Q}_{*}$. We have, by Hölder's inequality,
\begin{multline}
(\mathbf{E}(H_{i}^{B}\mathbf{1}_{\mathcal{A}_{i}}))^{3}=\left(\mathbf{E}\left[\left(\frac{d_{(3\mathbb{E}_{i})^{*},\mathbb{A}_{i},\e^{-\gamma F_{i}^{\uparrow}},S}^{(i)}(\mathrm{around})}{d_{\mathbb{E}_{i},\e^{-\gamma F_{i}^{\downarrow}},S}(\min;S/\diam_{\Euc}(\mathbb{E}_{i}))}\mathbf{1}_{\mathcal{A}_{i}}\right)^{B}\right]\right)^{3}\le\mathbf{E}\left[\left(\frac{d_{(3\mathbb{E}_{i})^{*},\mathbb{A}_{i},\e^{-\gamma F_{i}^{\uparrow}},S}^{(i)}(\mathrm{around})}{\Theta_{\mathbb{B}(S),\e^{-\gamma F_{i}^{\uparrow}}}^{\mathrm{hard}}(p_{1})}\right)^{3B}\right]\cdot\\
\cdot\mathbf{E}\left[\left(\frac{d_{\mathbb{E}_{i},\e^{-\gamma F_{i}^{\downarrow}},S}^{(i)}(\min;S/\diam_{\Euc}(\mathbb{E}_{i}))}{\Theta_{\mathbb{B}(S),\e^{-\gamma F_{i}^{\downarrow}}}^{\mathrm{easy}}(p_{0})}\right)^{-3B}\mathbf{1}_{\mathcal{A}_{i}}\right]\mathbf{E}\left[\left(\frac{\Theta_{\mathbb{B}(S),\e^{-\gamma F_{i}^{\uparrow}}}^{\mathrm{hard}}(p_{1})}{\Theta_{\mathbb{B}(S),\e^{-\gamma F_{i}^{\downarrow}}}^{\mathrm{easy}}(p_{0})}\mathbf{1}_{\mathcal{A}_{i}}\right)^{3B}\right].\label{eq:HiHolder}
\end{multline}
We bound each of the factors in turn. First we compute
\begin{equation}
\mathbf{E}\left[\left(\frac{d_{(3\mathbb{E}_{i})^{*},\mathbb{A}_{i},\e^{-\gamma F_{i}^{\uparrow}},S}^{(i)}(\mathrm{around})}{\Theta_{\mathbb{B}(S),\e^{-\gamma F_{i}^{\uparrow}}}^{\mathrm{hard}}(p_{1})}\right)^{3B}\right]=\mathbf{E}\left[\mathbf{E}\left[\left(\frac{d_{(3\mathbb{E}_{i})^{*},\mathbb{A}_{i},\e^{-\gamma F_{i}^{\uparrow}},S}^{(i)}(\mathrm{around})}{\Theta_{\mathbb{B}(S),\e^{-\gamma F_{i}^{\uparrow}}}^{\mathrm{hard}}(p_{1})}\right)^{3B}\;\middle|\;F_{i}^{\uparrow}\right]\right]\le\mathbf{E}C\le C\label{eq:hardmomentbd}
\end{equation}
by \corref{hardmomentbound} and \eqref{RprimeR-multiplicative} of
\lemref{RprimeR}. For the second two terms, we first note that, on
the event $\{F_{i}^{\downarrow}\le\theta_{0}\log K\}\supset\{F_{i}^{\uparrow}\le\theta_{0}\log K\}\supset\mathcal{A}_{i}$,
we have
\begin{equation}
\frac{2\gamma F_{i}^{\downarrow}}{4+\gamma^{2}}\le\frac{2\gamma F_{i}^{\uparrow}}{4+\gamma^{2}}\le\frac{2\gamma\theta_{0}\log K}{4+\gamma^{2}}\le\eta\log K,\label{eq:Aicond}
\end{equation}
where the first inequality is by the definitions \eqref{Fidef} of
$F_{i}^{\downarrow}$ and $F_{i}^{\uparrow}$, the second inequality
is by the definition \eqref{scrAidef} of the event $\mathcal{A}_{i}$,
and the third inequality is by the definition \eqref{etadef} of $\eta$.
Thus we have
\begin{align}
\mathbf{E} & \left[\left(\frac{d_{\mathbb{E}_{i},\e^{-\gamma F_{i}^{\downarrow}},S}(\min;S/\diam_{\Euc}(\mathbb{E}_{i}))}{\Theta_{\mathbb{B}(S),\e^{-\gamma F_{i}^{\downarrow}}}^{\mathrm{easy}}(p_{0})}\right)^{-3B}\mathbf{1}_{\mathcal{A}_{i}}\right]\nonumber \\
 & \le\mathbf{E}\left[\mathbf{E}\left[\left(\frac{d_{\mathbb{E}_{i},\e^{-\gamma F_{i}^{\downarrow}},S}(\min;S/\diam_{\Euc}(\mathbb{E}_{i}))}{\Theta_{\mathbb{B}(S),\e^{-\gamma F_{i}^{\downarrow}}}^{\mathrm{easy}}(p_{0})}\right)^{-3B}\mathbf{1}_{\{F_{i}^{\downarrow}\le\theta_{0}\log K\}}\;\middle|\;F_{i}^{\downarrow}\right]\right]\nonumber \\
 & =\mathbf{E}\left[\mathbf{E}\left[\left(\frac{d_{\e^{\frac{2\gamma F_{i}^{\downarrow}}{\gamma^{2}+4}}\mathbb{E}_{i},1,\e^{\frac{\gamma}{\gamma^{2}/2+2}F_{i}^{\downarrow}}S}(\min;S/\diam_{\Euc}(\mathbb{E}_{i}))}{\Theta_{\e^{\frac{2\gamma}{\gamma^{2}+4}F_{i}^{\downarrow}}\mathbb{B}(S),1}^{\mathrm{easy}}(p_{0})}\right)^{-3B}\;\middle|\;F_{i}^{\downarrow}\right]\mathbf{1}_{\{F_{i}^{\downarrow}\le\theta_{0}\log K\}}\right]\nonumber \\
 & \le C\mathbf{E}\left[\overline{\chi}_{\e^{\frac{2\gamma}{\gamma^{2}+4}F_{i}^{\downarrow}}S}^{3B}\mathbf{1}_{\{F_{i}^{\downarrow}\le\theta_{0}\log K\}}\right]\le C\mathbf{E}\overline{\chi}_{K^{\eta}S}^{3B}=C\overline{\chi}_{K^{\eta}S}^{3B},\label{eq:easymomentbd}
\end{align}
where the second equality is by \propref{LQGscaling}, the first inequality
is by \propref{minmomentbound}, and the second inequality is by \eqref{Aicond}.

For the third term in \eqref{HiHolder}, put $K^{*}=\exp\left\{ \frac{2\gamma F_{i}^{\uparrow}}{4+\gamma^{2}}\right\} $.
On the event $\mathcal{A}_{i}$, by \eqref{Aicond} we have $K^{*}\le K^{\eta}$.
Then we can write, using \eqref{scaledists} twice, recalling \defref{chidefs},
and applying \eqref{mostusefulscaling}, that
\begin{align*}
\mathbf{1}_{\mathcal{A}_{i}}\Theta_{\mathbb{B}(S),\e^{-\gamma F_{i}^{\uparrow}}}^{\mathrm{hard}}(p_{1})=\mathbf{1}_{\mathcal{A}_{i}}\Theta_{\mathbb{B}(K^{*}S)}^{\mathrm{hard}}(p_{1})\le\mathbf{1}_{\mathcal{A}_{i}}\overline{\chi}_{K^{\eta}S}\Theta_{\mathbb{B}(K^{*}S)}^{\mathrm{easy}}(p_{0}) & =\mathbf{1}_{\mathcal{A}_{i}}\overline{\chi}_{K^{\eta}S}\Theta_{\mathbb{B}(S),\e^{-\gamma F_{i}^{\uparrow}}}^{\mathrm{easy}}(p_{0})\\
 & \le\mathbf{1}_{\mathcal{A}_{i}}\overline{\chi}_{K^{\eta}S}\e^{\gamma(F_{i}^{\uparrow}-F_{i}^{\downarrow})}\Theta_{\mathbb{B}(S),\e^{-\gamma F_{i}^{\downarrow}}}^{\mathrm{easy}}(p_{0}).
\end{align*}
Therefore, we have
\begin{equation}
\mathbf{E}\left[\left(\frac{\Theta_{\mathbb{B}(S),\e^{-\gamma F_{i}^{\uparrow}}}^{\mathrm{hard}}(p_{1})}{\Theta_{\mathbb{B}(S),\e^{-\gamma F_{i}^{\downarrow}}}^{\mathrm{easy}}(p_{0})}\mathbf{1}_{\mathcal{A}_{i}}\right)^{3B}\right]\le C\overline{\chi}_{K^{\eta}S}^{3B}\mathbf{E}\e^{3B\gamma(F_{i}^{\uparrow}-F_{i}^{\downarrow})}\le CK\overline{\chi}_{K^{\eta}S}^{3B},\label{eq:ratiobd}
\end{equation}
where the last inequality is by \lemref{coarsefieldsmooth}. Plugging
\eqref{hardmomentbd}, \eqref{easymomentbd}, and \eqref{ratiobd}
into \eqref{HiHolder} yields
\begin{equation}
\mathbf{E}(H_{i}^{B}\mathbf{1}_{\mathcal{A}_{i}})\le CK\overline{\chi}_{K^{\eta}S}^{B}.\label{eq:Hibd-everything}
\end{equation}
Then we can write
\[
\mathbf{E}\left[\left(\sum_{\substack{i\in\mathcal{Q}_{*}\\
\mathcal{M}(i)\ne\emptyset
}
}\frac{d_{\mathbb{R}^{*},\mathbb{A}_{i},1,S}^{(i)}(\mathrm{around})}{d_{\mathbb{R},1,S}(\mathrm{L},\mathrm{R})}\mathbf{1}_{\mathcal{A}_{i}}\right)^{B}\right]\le C\mathbf{E}(H^{B}\mathbf{1}_{\mathcal{A}_{i}})\le C\sum_{i\in\mathcal{Q}_{*}}\mathbf{E}(H_{i}^{B}\mathbf{1}_{\mathcal{A}_{i}})\le CK^{2}\overline{\chi}_{K^{\eta}S}^{B},
\]
where the first inequality is by \lemref{boundintermsofH} and the
Cauchy--Schwarz inequality and the second inequality uses \eqref{Hibd-everything}
and the fact that $|\mathcal{Q}_{*}|\le CK^{2}$. But the last display
is \eqref{sumconclusion}.
\end{proof}

\subsubsection{The maximum small box crossing weight\label{subsec:maxsmall}}

We now show that the maximum annular circuit distance is much smaller
(for large $K$) than the crossing quantile of the large rectangle
$\mathbb{R}$. This will be used for bounding the middle factor of
\eqref{farawaystuff} as well as the first factor in the last term
of \eqref{efronstein-readyforlemmas}. Here we crucially use that
$\gamma<2$, as this means that even after considering the coarse
field, subboxes ``look'' smaller than the large box.
\begin{lem}
\label{lem:maxofsmaller}Define notation as in the statement of \lemref{boundintermsofH}.
Suppose that $T_{0}\le K^{\eta}S$, where $T_{0}$ is as in \propref{pllb-inductive}.
For any $B\ge0$ there is a $c>0$ so that
\begin{align}
\mathbf{E}\left[\left(\max_{i\in\mathcal{Q}_{*}}d_{(3\mathbb{E}_{i})^{*},\mathbb{A}_{i},K^{-\gamma\theta_{0}}}^{(i)}(\mathrm{around})\right)^{B}\right] & ,\mathbf{E}\left[\left(\max_{i\in\mathcal{Q}_{*}}d_{(3\mathbb{E}_{i})^{*},\mathbb{A}_{i},K^{-\gamma\theta_{0}}}(\mathrm{around})\right)^{B}\right]\le C\overline{\chi}_{K^{\eta}S}^{B}K^{-c}\Theta_{\mathbb{B}(KS)}^{\mathrm{easy}}(p_{0})^{B},\label{eq:maxofsmaller}
\end{align}
\end{lem}

\begin{proof}
We have
\begin{equation}
\mathbf{E}\left[\left(\max_{i\in\mathcal{Q}_{*}}d_{(3\mathbb{E}_{i})^{*},\mathbb{A}_{i},K^{-\gamma\theta_{0}}}^{(i)}(\mathrm{around})\right)^{B}\right]\le\sum_{i\in\mathcal{Q}_{*}}\mathbf{E}d_{(3\mathbb{E}_{i})^{*},\mathbb{A}_{i},K^{-\gamma\theta_{0}}}^{(i)}(\mathrm{around})^{B}\le C|\mathcal{Q}_{*}|^{\frac{1}{2}}\Theta_{\mathbb{B}(K^{\eta}S)}^{\mathrm{hard}}(p_{1})^{B},\label{eq:maxmoment}
\end{equation}
where the second inequality is by \eqref{scaledists} and \corref{hardmomentbound},
and in the last expression the constant $C$ depends on $B$. (Recall
that $p_{1}$ is fixed as in \propref{hardconcentration}.) On the
other hand, we have that
\[
\Theta_{\mathbb{B}(K^{\eta}S)}^{\mathrm{hard}}(p_{1})\le\overline{\chi}_{K^{\eta}S}\Theta_{\mathbb{B}(K^{\eta}S)}^{\mathrm{easy}}(p_{0})\le C\overline{\chi}_{K^{\eta}S}K^{-c(1-\eta)}\Theta_{\mathbb{B}(KS)}^{\mathrm{easy}}(p_{0}),
\]
where the first inequality is by \defref{chidefs} and the second
is by \eqref{pllb-inductive}. Here we use the assumption that $K^{\eta}S\ge T_{0}$.
Plugging the last inequality into \eqref{maxmoment} yields
\begin{align*}
\mathbf{E}\left[\left(\max_{i\in\mathcal{Q}_{*}}d_{(3\mathbb{E}_{i})^{*},\mathbb{A}_{i},K^{-\gamma\theta_{0}}}^{(i)}(\mathrm{around})\right)^{B}\right] & \le C\overline{\chi}_{K^{\eta}S}^{B}|\mathcal{Q}_{*}|^{1/2}K^{-Bc(1-\eta)}\Theta_{\mathbb{B}(KS)}^{\mathrm{easy}}(p_{0})^{B}\\
 & \le C\overline{\chi}_{K^{\eta}S}^{B}K^{1-Bc(1-\eta)}\Theta_{\mathbb{B}(KS)}^{\mathrm{easy}}(p_{0})^{B}.
\end{align*}
Choosing $B$ large enough so that $-c'\coloneqq1-Bc(1-\eta)<0$,
 this becomes
\[
\mathbf{E}\left[\left(\max_{i\in\mathcal{Q}_{*}}d_{(3\mathbb{E}_{i})^{*},\mathbb{A}_{i},K^{-\gamma\theta_{0}}}^{(i)}(\mathrm{around})\right)^{B}\right]\le C\overline{\chi}_{K^{\eta}S}^{B}K^{-c'}\Theta_{\mathbb{B}(KS)}^{\mathrm{easy}}(p_{0})^{B},
\]
which is half of \eqref{maxofsmaller}. The other half follows in
the same way, noting that \eqref{maxmoment} uses nothing about the
correlations between the $d_{(3\mathbb{E}_{i})^{*},\mathbb{A}_{i},K^{-\gamma\theta_{0}}}^{(i)}(\mathrm{around})$s
for varying $i$.
\end{proof}

\subsubsection{The effect of requiring small balls}

Our Efron--Stein argument required restricting the balls used to
cover the path to be of size at most $S$. This was important for
our percolation argument, because otherwise a single ball could be
used to cover the path in potentially very many of the $\mathbb{C}_{i}$s.
However, the effect of this requirement should be negligible, because
at large scales, we do not expect large balls to be used: recall from
\eqref{negmoments} that the LQG measure has all negative moments.
The next lemma quantifies this intuition.
\begin{lem}
\label{lem:smallballseffect}Suppose $KS\ge T_{0}$ (defined as in
\propref{pllb-inductive}). There are constants $C<\infty$ and $c>0$
so that
\[
\mathbf{E}\left(\log\frac{d_{\mathbb{R},1,S}(\mathrm{L},\mathrm{R})}{d_{\mathbb{R}}(\mathrm{L},\mathrm{R})}\right)^{2}\le CK^{C}S^{-c}.
\]
\end{lem}

\begin{proof}
By \eqref{RprimeR}, we have
\[
\log\frac{d_{\mathbb{R},1,S}(\mathrm{L},\mathrm{R})}{d_{\mathbb{R}}(\mathrm{L},\mathrm{R})}\le\log\left(1+\frac{K^{3}}{d_{\mathbb{R}}(\mathrm{L},\mathrm{R})}\right)\le\frac{K^{3}}{d_{\mathbb{R}}(\mathrm{L},\mathrm{R})}
\]
almost surely. Therefore, we have, as long as $KS\ge T_{0}$, that
\[
\mathbf{E}\left[\left(\log\frac{d_{\mathbb{R},1,S}(\mathrm{L},\mathrm{R})}{d_{\mathbb{R}}(\mathrm{L},\mathrm{R})}\right)^{2}\right]\le K^{6}\mathbf{E}d_{\mathbb{R}}(\mathrm{L},\mathrm{R})^{-2}\le CK^{6}\overline{\chi}_{KS}\left(\frac{KS}{T_{0}}\right)^{-c}\Theta_{\mathbb{B}(T_{0})}^{\mathrm{easy}}(p)^{-2}\le CK^{C}S^{-c}.
\]
with the second inequality by \propref{minmomentbound}, the third
by \propref{pllb-inductive}, and the last by \eqref{chiUaprioribound}.
\end{proof}

\subsubsection{The large coarse field error term\label{subsec:largecoarsefield}}

Here we bound the term of \eqref{efronstein-readyforlemmas} corresponding
to the event that the coarse field is extraordinarily large.
\begin{lem}
\label{lem:lgcoarsefielderror}There are constants $C<\infty$ and
$c>0$ so that 
\[
\sum_{i\in\mathcal{Q}_{*}}\mathbf{E}\left[\left(\left(\log\frac{d_{\mathbb{R},1,S}^{(i)}(\mathrm{L},\mathrm{R})}{d_{\mathbb{R},1,S}(\mathrm{L},\mathrm{R})}\right)^{+}\right)^{2}\mathbf{1}_{\mathcal{A}_{i}^{c}}\right]\le CK^{-c}\overline{\chi}_{S}^{3}.
\]
\end{lem}

\begin{proof}
Whenever $\alpha,\beta\in(1,\infty)$ and $1/\alpha+1/\beta=1$, we
have
\begin{equation}
\mathbf{E}\left[\left(\left(\log\frac{d_{\mathbb{R},1,S}^{(i)}(\mathrm{L},\mathrm{R})}{d_{\mathbb{R},1,S}(\mathrm{L},\mathrm{R})}\right)^{+}\right)^{2}\mathbf{1}_{\mathcal{A}_{i}^{c}}\right]\le4\left[\mathbf{P}(\mathcal{A}_{i}^{c})\right]^{1/\alpha}\left[\mathbf{E}\left(\left(\log\frac{d_{\mathbb{R},1,S}^{(i)}(\mathrm{L},\mathrm{R})}{d_{\mathbb{R},1,S}(\mathrm{L},\mathrm{R})}\right)^{+}\right)^{2\beta}\right]^{1/\beta}\label{eq:secondterm}
\end{equation}
by Hölder's inequality. By \lemref{Aicbound}, $\mathbf{P}(\mathcal{A}_{i}^{c})\le K^{-\theta_{0}^{2}/2}$,
so
\begin{equation}
\sum_{i\in\mathcal{Q}_{*}}[\mathbf{P}(\mathcal{A}_{i}^{c})]^{1/\alpha}\le K^{2-\frac{\theta_{0}^{2}}{2\alpha}},\label{eq:Pbd}
\end{equation}
and the exponent $-2c\coloneqq2-\frac{\theta_{0}^{2}}{2\alpha}$ is
negative so long as $\alpha$ is chosen sufficiently close to $1$.
Also, we have constants $A,C>0$ so that by \propref{concentration-boost-hard}
and \propref{hardquantile},
\begin{equation}
\mathbf{E}d_{\mathbb{R},1,S}^{(i)}(\mathrm{L},\mathrm{R})\le CK^{A}\Theta_{\mathbb{B}(S)}^{\mathrm{hard}}(p_{1})\label{eq:expdibd}
\end{equation}
and by \propref{minmomentbound} and \propref{pllb-inductive},
\begin{equation}
\mathbf{E}\left[\left(d_{\mathbb{R},1,S}(\mathrm{L},\mathrm{R})\right)^{-1}\right]\le C\overline{\chi}_{S}K^{A}\Theta_{\mathbb{B}(S)}^{\mathrm{easy}}(p_{0})^{-1}.\label{eq:expdbd}
\end{equation}
For any $b>0$, we have a constant $C=C(b)$ so that $(\log x)^{+}\le Cx^{b}$
for all $x>0$. Using this fact, the Cauchy--Schwarz and Jensen inequalities,
and \eqref{expdibd} and \eqref{expdbd}, we obtain in particular
that as long as $4b\beta\le1$, for each $i\in\mathcal{Q}_{*}$,
\begin{align}
\left[\mathbf{E}\left(\left(\log\frac{d_{\mathbb{R},1,S}^{(i)}(\mathrm{L},\mathrm{R})}{d_{\mathbb{R},1,S}(\mathrm{L},\mathrm{R})}\right)^{+}\right)^{2\beta}\right]^{1/\beta} & \le C\left[\left(K^{A}\Theta_{\mathbb{B}(S)}^{\mathrm{hard}}(p_{1})\right)\left(\overline{\chi}_{S}K^{A}\Theta_{\mathbb{B}(S)}^{\mathrm{easy}}(p_{0})^{-1}\right)\right]^{4b}.\label{eq:Elogbd}
\end{align}
Taking $b$ so small that $4\beta b\le1$, $8Ab<c$ and $b\le1/4$,
we obtain from \eqref{secondterm}, \eqref{Pbd} and \eqref{Elogbd}
that
\begin{multline*}
\sum_{i\in\mathcal{Q}_{*}}\mathbf{E}\left[\left(\left(\log\frac{d_{\mathbb{R},1,S}^{(i)}(\mathrm{L},\mathrm{R})}{d_{\mathbb{R},1,S}(\mathrm{L},\mathrm{R})}\right)^{+}\right)^{2}\mathbf{1}_{\mathcal{A}_{i}^{c}}\right]\le\left(\sum_{i\in\mathcal{Q}_{*}}\mathbf{P}(\mathcal{A}_{i}^{c})^{1/\alpha}\right)\left(\max_{i\in\mathcal{Q}_{*}}\mathbf{E}\left(\left(\left(\log\frac{d_{\mathbb{R},1,S}^{(i)}(\mathrm{L},\mathrm{R})}{d_{\mathbb{R},1,S}(\mathrm{L},\mathrm{R})}\right)^{+}\right)^{2\beta}\right)\right)^{1/\beta}\\
\le CK^{-c}\overline{\chi}_{S}^{b}\frac{\Theta_{\mathbb{B}(S)}^{\mathrm{hard}}(p_{1})}{\Theta_{\mathbb{B}(S)}^{\mathrm{easy}}(p_{0})}\le CK^{-c}\overline{\chi}_{S}^{b+1}\le CK^{-c}\overline{\chi}_{S}^{3},
\end{multline*}
where in the penultimate inequality we recall \defref{chidefs}.
\end{proof}

\subsubsection{Putting it all together: the proof of Lemma~\ref{lem:finalvariancebound}\label{subsec:proofoffinalvarbound}}

We are finally ready to prove \lemref{finalvariancebound}, which
will imply \thmref{variancebound} as indicated at the end of \subsecref{efronstein}
(on page~\pageref{varianceboundproof}).
\begin{proof}[Proof of \lemref{finalvariancebound}.]
We have, using \lemref{coarsefieldeffect}, \lemref{farawaystuff},
and \lemref{smallballseffect} to bound the first three terms of \eqref{efronstein-readyforlemmas},
and \lemref{lgcoarsefielderror} to bound the last term, that there
are constants $C<\infty$ and $c>0$ so that as long as $KS\ge T_{0}$,
we have
\begin{align}
\Var & \left(\log d_{\mathbb{R}}(\mathrm{L},\mathrm{R})\right)\nonumber \\
 & \le C\log K+C\e^{C\sqrt{\log K}}\left(\mathbf{E}\max_{j\in\mathcal{Q}_{\circ}}d_{(3\mathbb{E}_{i})^{*},\mathbb{A}_{j},K^{-\gamma\theta_{0}},S}(\mathrm{around})^{3}\right)^{1/3}\left(\mathbf{E}d_{\mathbb{R},1,S}(\pi)^{-3}\right)^{1/3}+CK^{C}S^{-c}\nonumber \\
 & \quad+2\left(\mathbf{E}\left[\max_{i\in\mathcal{Q}_{*}}d_{(3\mathbb{E}_{i})^{*},\mathbb{A}_{i},K^{-\gamma\theta_{0}},S}^{(i)}(\mathrm{around})^{3}\right]\mathbf{E}[d_{\mathbb{R},1,S}(\mathrm{L},\mathrm{R})^{-3}]\mathbf{E}\left[\left(\sum_{i\in\mathcal{Q}_{*}}\frac{\mathbf{1}_{\mathcal{A}_{i}}d_{\mathbb{R}^{*},\mathbb{A}_{i},1,S}^{(i)}(\mathrm{around})}{d_{\mathbb{R},1,S}(\mathrm{L},\mathrm{R})}\right)^{3}\right]\right)^{1/3}\nonumber \\
 & \quad+CK^{-c}\overline{\chi}_{S}^{3}.\label{eq:varhalfway}
\end{align}
Now for any $B\ge3$ there is a constant $C$ so that
\begin{equation}
\left(\mathbf{E}\left[\left(\sum_{i\in\mathcal{Q}_{*}}\frac{\mathbf{1}_{\mathcal{A}_{i}}d_{\mathbb{R}^{*},\mathbb{A}_{i},1,S}^{(i)}(\mathrm{around})}{d_{\mathbb{R},1,S}(\mathrm{L},\mathrm{R})}\right)^{3}\right]\right)^{1/3}\le CK^{2/B}\overline{\chi}_{K^{\eta}S}\label{eq:dealwithL1}
\end{equation}
by Jensen's inequality and \lemref{sumconclusion}. Also, by \lemref{maxofsmaller},
there is a $c>0$ so that, as long as $K^{\eta}S\ge T_{0}$,
\begin{equation}
\left(\mathbf{E}\left[\left(\max_{i\in\mathcal{Q}_{*}}d_{(3\mathbb{E}_{i})^{*},\mathbb{A}_{i},K^{-\gamma\theta_{0}}}^{(i)}(\mathrm{around})\right)^{3}\right]\right)^{\frac{1}{3}},\left(\mathbf{E}\left[\left(\max_{i\in\mathcal{Q}_{*}}d_{(3\mathbb{E}_{i})^{*},\mathbb{A}_{i},K^{-\gamma\theta_{0}}}(\mathrm{around})\right)^{3}\right]\right)^{\frac{1}{3}}\le C\overline{\chi}_{K^{\eta}S}K^{-c}\Theta_{\mathbb{B}(KS)}^{\mathrm{easy}}(p_{0}).\label{eq:dealwithmax}
\end{equation}
Finally, we have
\begin{equation}
\left(\mathbf{E}[d_{\mathbb{R},1,S}(\mathrm{L},\mathrm{R})^{-3}]\right)^{\frac{1}{3}}=\left(\mathbf{E}[d_{\mathbb{R},1,S}(\pi)^{-3}]\right)^{\frac{1}{3}}\le C\overline{\chi}_{KS/100}\Theta_{\mathbb{B}(KS)}^{\mathrm{easy}}(p_{0})^{-1}\label{eq:dealwithmin}
\end{equation}
by \propref{minmomentbound} (applied with $Q=100$).

Plugging \eqref{dealwithL1}, \eqref{dealwithmax}, and \eqref{dealwithmin}
into the last term of \eqref{varhalfway}, and \eqref{dealwithmax}
and \eqref{dealwithmin} into the second term of \eqref{varhalfway},
we obtain
\begin{align*}
\Var\left(\log d_{\mathbb{R}}(\mathrm{L},\mathrm{R})\right) & \le C\log K+C\e^{C\sqrt{\log K}}\overline{\chi}_{K^{\eta}S}K^{-c}\overline{\chi}_{KS/100}+CK^{C}S^{-c}+CK^{2/B-c}\overline{\chi}_{KS/100}\overline{\chi}_{K^{\eta}S}^{2}+CK^{-c}\overline{\chi}_{S}^{3}\\
 & \le C\log K+CK^{-c/2}\overline{\chi}_{KS/100}^{3}+CK^{C}S^{-c},
\end{align*}
with the second equality valid so long as $B$ is chosen so large
that $2/B\le c/4$ and $K$ is large enough that $K\ge100$, $K^{\eta-1}\ge1/100$
and $\e^{C\sqrt{\log K}}\le K^{c/4}$. 
\end{proof}

\section{Subsequential scaling limits\label{sec:chaining}}

The variance bound proved in the previous section will allow us to
show that the renormalized metric is ``almost'' Hölder continuous.
(The ``almost'' is because the metric takes discrete values---this
point is not important.) We recall the definition of $\Theta^{*}(S)$
in \corref{Thetastardef}.
\begin{thm}
\label{thm:Holder}There is a constant $c>0$ and a function $f:\mathbf{R}_{>0}\to\mathbf{R}_{>0}$
so that $\lim\limits _{S\to\infty}f(S)=0$ and the following holds.
For every $\eps>0$, there is a constant $C=C(\eps)<\infty$ so that
for every $S\ge0$, with probability at least $1-\eps$, we have that
for every $x,y\in\mathbb{B}(S)$,
\[
\frac{d_{\mathbb{B}(S)}(x,y)}{\Theta^{*}(S)}\le C(\eps)\left(\left|\frac{x-y}{S}\right|^{c}\vee f(S)\right).
\]
\end{thm}

\begin{cor}
\label{cor:finalcor}The family
\[
\left\{ \frac{d_{\mathbb{B}(1),\delta}(\cdot,\cdot)}{\Theta^{*}(\delta^{-\frac{2}{\gamma^{2}+4}})}\right\} _{\delta>0}
\]
is tight in the uniform topology on $[0,1]^{2}$.
\end{cor}

\begin{proof}
This follows immediately from \thmref{Holder}, \propref{LQGscaling},
and \lemref{AA} (the last being a slight generalization of the Arzelà--Ascoli
theorem).
\end{proof}
We note that \corref{finalcor} immediately implies almost all of
\thmref{maintheorem} by \propref{ddunderline} and Prokhorov's theorem,
except that the limiting objects may be pseudometrics rather than
metrics. In this section we prove \thmref{Holder}, and then in the
next section we will prove that the limiting objects are in fact metrics.
We will need a preliminary lemma, which allows us to deal with the
``base case'' of our diameter chaining argument by showing that
it is unlikely that very small boxes will have large diameter.
\begin{lem}
\label{lem:basecase}There is a $C<\infty$ and a $\nu_{0}>1$ so
that if $\nu\in(1,\nu_{0})$ then there is an $\alpha>0$ so that
the following holds. Fix $S,K>0$ and let $\mathbb{R}=\mathbb{B}(S)$.
Divide $\mathbb{R}^{\circ}$ into $J$-many $K^{-1}S\times K^{-1}S$
boxes labeled $\mathbb{C}_{1},\ldots,\mathbb{C}_{J}$, with $J\le CK^{2}$.
Then for all $u>0$,
\begin{equation}
\mathbf{P}\left(\max_{1\le j\le J}\mu_{\mathbb{R}^{*}}(\mathbb{C}_{j})\ge u\right)\le Cu^{-\nu}S^{(2+\gamma^{2}/2)\nu}K^{-\alpha}.\label{eq:basecase}
\end{equation}
\end{lem}

\begin{proof}
Define $\tilde{\nu}_{0}$ to be the $\nu_{0}$ from \propref{momentsbounded}.
Fix $\nu\in(1,\tilde{\nu}_{0})$. For $1\le j\le J$, define
\[
F_{j}=\max_{x\in\mathbb{C}_{j}^{\circ}}h_{\mathbb{R}^{*}:\mathbb{C}_{j}^{*}}(x).
\]
Then we note that $F_{j}$ and $\mu_{\mathbb{C}_{j}^{*}}(\mathbb{C}_{j})$
are independent. Also, by \eqref{mostusefulscaling}, we have
\begin{equation}
\mu_{\mathbb{R}^{*}}(\mathbb{C}_{j})\le\e^{\gamma F_{j}}\mu_{\mathbb{C}_{j}^{*}}(\mathbb{C}_{j}).\label{eq:3RCj}
\end{equation}
We recall from \eqref[s]{mostusefulscaling}~and~\ref{lem:fluctstailbound}
that $F_{j}$ has expectation of order $1$ and tails like those of
a Gaussian with variance $\log K$. Therefore, we have
\begin{equation}
\mathbf{E}\e^{\gamma\nu F_{j}}\le CK^{\frac{1}{2}\gamma^{2}\nu^{2}}.\label{eq:expFj}
\end{equation}
On the other hand, we have that $\mu_{\mathbb{C}_{j}^{*}}(\mathbb{C}_{j})\overset{\mathrm{law}}{=}(S/K)^{2+\gamma^{2}/2}\mu_{\mathbb{B}^{*}}(\mathbb{B})$,
where $\mathbb{B}=[0,1]^{2}$, so
\begin{equation}
\mathbf{E}\mu_{\mathbb{C}_{j}^{*}}(\mathbb{C}_{j})^{\nu}=(S/K)^{(2+\gamma^{2}/2)\nu}\mathbf{E}\left[\mu_{\mathbb{B}^{*}}(\mathbb{B})^{\nu}\right].\label{eq:expmmu}
\end{equation}
Therefore, by \eqref{3RCj} and the independence of the two factors
on its right-hand side, along with \eqref{expFj} and \eqref{expmmu},
we have
\[
\mathbf{E}\left[\mu_{\mathbb{R}^{*}}(\mathbb{C}_{j})^{\nu}\right]\le CS^{(2+\gamma^{2}/2)\nu}K^{\frac{1}{2}\gamma^{2}\nu^{2}-(2+\gamma^{2}/2)\nu}\mathbf{E}\left[\mu_{\mathbb{B}^{*}}(\mathbb{B})^{\nu}\right]\le CS^{(2+\gamma^{2}/2)\nu}K^{\frac{1}{2}\gamma^{2}\nu^{2}-(2+\gamma^{2}/2)\nu}
\]
in which we note that $\mathbf{E}\mu_{\mathbb{B}^{*}}(\mathbb{B})^{\nu}$
is bounded above by a fixed absolute constant according to \eqref{posmoments}.
Then we have
\[
\mathbf{E}\left[\max_{1\le j\le J}\mu_{\mathbb{R}^{*}}(\mathbb{C}_{j})^{\nu}\right]\le J\mathbf{E}\left[\mu_{\mathbb{R}^{*}}(\mathbb{C}_{j})^{\nu}\right]\le CS^{(2+\gamma^{2}/2)\nu}K^{\frac{1}{2}\gamma^{2}\nu^{2}-(2+\gamma^{2}/2)\nu+2}.
\]
If we define $f(\nu)=\frac{1}{2}\gamma^{2}\nu^{2}-(2+\frac{1}{2}\gamma^{2})\nu+2$,
then  $f(1)=0$ and $f'(1)=\frac{1}{2}\gamma^{2}-2<0$, since $\gamma<2$.
This implies that there is a $\nu_{0}>1$ such that so that $f(\nu)<0$
for $\nu\in(1,\nu_{0})$, which means that if we further take $\nu_{0}<\tilde{\nu}_{0}$,
then for each $\nu\in(1,\nu_{0})$ we have an $\alpha>0$ so that
\[
\mathbf{E}\left[\max_{1\le j\le J}[\mu_{\mathbb{R}^{*}}(\mathbb{C}_{j})]^{\nu}\right]\le CS^{(2+\gamma^{2}/2)\nu}K^{-\alpha}.
\]
Inequality~\eqref{basecase} then comes from Markov's inequality.
\end{proof}
\begin{prop}
\label{prop:chaining}There is a constant $c>0$ so that for every
$\eps>0$, there is a constant $C=C(\eps)<\infty$ so that for every
$S>0$ the following holds. Let $\mathbb{R}=\mathbb{B}(S)$. For integers
$t\in\mathbf{Z}_{\ge0}$ and $0\le i,j\le2^{t}-1$, define
\[
\mathbb{B}_{t;i,j}=\begin{cases}
(i2^{-t}S,j2^{-t+1}S)+[0,2^{-t}S]\times[0,2^{-t+1}S] & \text{if \ensuremath{t} is even;}\\
(i2^{-t+1}S,j2^{-t}S)+[0,2^{-t+1}S]\times[0,2^{-t}S] & \text{if \ensuremath{t} is odd.}
\end{cases}
\]
Then
\begin{equation}
\mathbf{P}\left(\bigcup_{t_{0}=0}^{\lfloor c\log_{2}S\rfloor}\left\{ \max_{i,j=0}^{2^{t_{0}}-1}\sup_{x,y\in\mathbb{B}_{t_{0};ij}}\frac{d_{\mathbb{R}}(x,y)}{\Theta^{*}(S)}\ge C(\eps)2^{-ct_{0}}\right\} \right)<\eps.\label{eq:probto0}
\end{equation}
\end{prop}

\begin{proof}
Because the proof is somewhat technical and involves many parameters,
we divide it into steps. The first step is to use a chaining argument
to express the maximum in \eqref{probto0} in terms of a sum of hard
crossing distances of smaller boxes at many scales. Then we define
a budget for the hard crossing distance at each scale, and show that
this budget is, with high probability, not exceeded at any scale.

\emph{Step 1: the chaining argument. }By a chaining argument illustrated
in \figref{pointtopoint},
\begin{figure}
\begin{centering}
\tiny
\begin{tikzpicture}[x=0.35in,y=0.35in]
\draw[step=1,thin] (0,0) grid (8,4);
\draw[fill=black,opacity=0.1] (0,0) rectangle (8,4);
\draw[fill=black,opacity=0.2] (0,0) rectangle (2,4);
\draw[fill=black,opacity=0.2] (4,0) rectangle (6,4);
\draw[fill=black,opacity=0.3] (0,0) rectangle (2,1);
\draw[fill=black,opacity=0.3] (4,2) rectangle (6,3);
\draw[ultra thick, yellow,style={decorate,decoration={snake,amplitude=2,segment length=50}}] (0,2.1) -- (8,0.2);
\draw[very thick, yellow,style={decorate,decoration={snake,amplitude=2,segment length=20}}] (0.2,0) -- (0.9,4);
\draw[very thick, yellow,style={decorate,decoration={snake,amplitude=2,segment length=20}}] (5.4,0) -- (4.9,4);
\draw[thick, yellow,style={decorate,decoration={snake,amplitude=1}}] (0,0.4) -- (2,0.5);
\draw[thick, yellow,style={decorate,decoration={snake,amplitude=1}}] (4,2.5) -- (6,2.7);
\draw[yellow,style={decorate,decoration={snake,amplitude=1}}] (5.5,2.3) -- (5.5,2.62);
\draw[yellow,style={decorate,decoration={snake,amplitude=1}}] (1.5,0.7) -- (1.5,0.45);
\filldraw [white] (5.5,2.3) circle (2.5pt);
\filldraw [white] (1.5,0.7) circle (2.5pt);
\end{tikzpicture}
\par\end{centering}
\caption{\label{fig:pointtopoint}Connecting any two points with two hard crossings
at each scale, plus two point-to-point paths within $\mathbb{B}_{t_{*};i,j}$s.
Any two points $x,y\in\mathbb{R}$ can be connected by a combination
of a hard crossing in each of the $\mathbb{B}_{t;i,j}$s containing
either $x$ or $y$, a connection from $x$ to a point on the hard
crossing of the $\mathbb{B}_{t_{*};i,j}$ containing $x$, and a connection
from $y$ to a point on the hard crossing of the $\mathbb{B}_{t_{*};i,j}$
containing $y$. This is because the hard crossing at each scale is
guaranteed to intersect the hard crossings at the next smaller scale,
and so all of the crossings can be joined into a crossing joining
the two points.}
\end{figure}
 we have, for any $t,t_{*}\in\mathbf{Z}_{\ge0}$, that
\begin{align}
\max_{i,j=0}^{2^{t}-1}\sup_{x,y\in\mathbb{B}_{t;i,j}}d_{\mathbb{R}}(x,y) & \le2\max_{i,j=0}^{2^{t_{*}}-1}\sup_{x,y\in\mathbb{B}_{t_{*};i,j}}d_{\mathbb{R}^{*},\mathbb{B}_{t_{*};i,j}}(x,y)+2\sum_{s=t}^{t_{*}}\max_{i,j=0}^{2^{s}-1}d_{\mathbb{R}^{*},\mathbb{B}_{s;i,j}}(\mathrm{hard})\nonumber \\
 & \le C\max_{i,j=0}^{2^{t_{*}}-1}\mu_{\mathbb{R}^{*}}(\mathbb{B}_{t_{*};i,j})+2\sum_{s=t}^{t_{*}}\max_{i,j=0}^{2^{s}-1}d_{\mathbb{R}^{*},\mathbb{B}_{s;i,j}}(\mathrm{hard}).\label{eq:xydistbound}
\end{align}
(See \cite[Proposition 6.7]{DD16} for a similar construction described
in more detail.) The second inequality in \eqref{xydistbound} is
by the fact that all points in $\mathbb{B}_{t_{*};i,j}$ can be connected
by some fixed number of balls inside $\mathbb{B}_{t_{*};i,j}^{\circ}$
with centers in $\mathbb{B}_{t_{*};i,j}$.

We need to understand the coarse field on each of the $\mathbb{B}_{t;i,j}$s.
To this end, we define
\begin{align*}
F_{t;i,j} & =\max_{x\in\mathbb{B}_{t;i,j}}h_{\mathbb{R}^{*}:\mathbb{B}_{t;i,j}^{*}}(x), & F_{t} & =\max_{i,j=0}^{2^{t}-1}F_{t;i,j}.
\end{align*}
Also define $\mathbb{S}_{s}=\mathbb{B}(2^{-(1-\eta)s})$. By \eqref{scaledeltabymax}
of \propref{fieldonsmallerbox}, \eqref{scaledists}, and \propref{LQGscaling},
we have
\begin{equation}
d_{\mathbb{R}^{*},\mathbb{B}_{s;i,j}}(\mathrm{hard})\le d_{\mathbb{B}_{s;i,j},\e^{-\gamma F_{s}}}(\mathrm{hard})\le\e^{(\gamma F_{s}-\gamma\theta_{0}s\log2)^{+}}d_{\mathbb{B}_{s;i,j},2^{-\gamma\theta_{0}s}}(\mathrm{hard})\overset{\mathrm{law}}{=}\e^{(\gamma F_{s}-\gamma\theta_{0}s\log2)^{+}}d_{\mathbb{S}_{s}}(\mathrm{hard}).\label{eq:diambound}
\end{equation}

By \eqref{xydistbound}, \eqref{diambound}, and a union bound, we
have, for any choices of $B$ and $q_{t}$ ($0\le t\le t_{*})$, 
\begin{align}
\mathbf{P} & \left(\bigcup_{t_{0}=0}^{t_{*}}\left\{ \max_{i,j=0}^{2^{t_{0}}-1}\sup_{x,y\in\mathbb{B}_{t_{0};i,j}}d_{\mathbb{R}}(x,y)\ge2q_{t_{*}}+2\sum_{t=t_{0}}^{t_{*}-1}2^{Bt^{2/3}}q_{t}\right\} \right)\nonumber \\
 & \le\sum_{t=1}^{t_{*}-1}\mathbf{P}\left[\e^{\gamma F_{t}-\gamma\theta_{0}t\log2}\ge2^{Bt^{2/3}}\right]+\sum_{t=0}^{t_{*}-1}4^{t}\mathbf{P}(d_{\mathbb{S}_{t}}(\mathrm{hard})\ge q_{t})+\mathbf{P}\left(\max_{i,j=0}^{2^{t_{*}}-1}\mu_{\mathbb{R}^{*}}(\mathbb{B}_{t_{*};i,j})\ge q_{t_{*}}\right).\label{eq:masterformula}
\end{align}
Here, the $q_{t}$s represent a ``distance budget'' that is available
at each scale, so that the distance at each scale stays within the
budget except on events whose probabilities we hope will be summable.

\emph{Step 2: defining the distance budget. }To use \eqref{masterformula},
we have to define the budgets $q_{t}$, and then estimate the terms
on the right-hand side. First, we define 
\begin{equation}
t_{1}=\lfloor(1-\eta)^{-1}\log_{2}(S/T_{0})\rfloor,\label{eq:t1def}
\end{equation}
so 
\begin{equation}
2^{-(1-\eta)t_{1}}S\ge T_{0},\label{eq:t1ST0}
\end{equation}
where $T_{0}$ is defined as in \propref{pllb-inductive}. Then, for
$0\le t<t_{*}$, put
\[
q_{t}=(1+2^{A(t\wedge t_{1})+D})\Theta^{*}(2^{-(1-\eta)t}S),
\]
with constants $A$ and $D$ to be chosen later. By \propref{concentration-boost-hard}
and \corref{Thetastardef}, we have
\begin{equation}
\mathbf{P}(d_{\mathbb{S}_{t}}(\mathrm{hard})\ge q_{t})\le C\e^{-A^{8/7}(t\wedge t_{1})^{8/7}/C-D^{8/7}/C}.\label{eq:dSt}
\end{equation}
On the other hand, by \corref{Thetastardef}, \propref{thetamonotone},
and \propref{pllb-inductive} (recalling \eqref{t1ST0}) that there
are constants $c>0$ and $C_{1}<\infty$ so that
\[
\Theta^{*}(2^{-(1-\eta)t}S)\le\Theta^{*}(2^{-(1-\eta)(t\wedge t_{1})}S)\le C_{1}2^{-c(1-\eta)(t\wedge t_{1})}\Theta^{*}(S).
\]
Therefore, we have
\begin{equation}
q_{t}\le C_{1}2^{(A-c(1-\eta))(t\wedge t_{1})+D}\Theta^{*}(S)\label{eq:qtbd}
\end{equation}
for all $0\le t<t_{*}$.

Also, by \eqref{coarsebound-universal} of \corref{maxcoarse}, we
have
\begin{equation}
\mathbf{P}\left[\e^{\gamma F_{t}-\gamma\theta_{0}t\log2}\ge2^{Bt^{2/3}}\right]\le C\exp\left\{ -\frac{B^{2}t^{1/3}}{2\gamma^{2}(\log2)^{3}}\right\} .\label{eq:coarsebound-chaining}
\end{equation}

Substituting \eqref{dSt}, \eqref{coarsebound-chaining}, \eqref{qtbd},
and \eqref{basecase} from \lemref{basecase} into \eqref{masterformula}
yields, for any $\nu\in(1,\nu_{0})$ (defining $\nu_{0}$ and also
$\alpha$ as in \lemref{basecase}), a constant $C$ so that
\begin{align}
\mathbf{P} & \left(\bigcup_{t_{0}=0}^{t_{*}-1}\left\{ \max_{i,j=0}^{2^{t_{0}}-1}\sup_{x,y\in\mathbb{B}_{t_{0};i,j}}d_{\mathbb{R}}(x,y)\ge q_{t_{*}}+C_{1}\Theta^{*}(S)2^{D}\sum_{t=t_{0}}^{t_{*}-1}2^{Bt^{2/3}+(A-c(1-\eta))(t\wedge t_{1})}\right\} \right)\nonumber \\
 & \le C\sum_{t=0}^{t_{*}-1}\e^{-\frac{B^{2}t^{1/3}}{2\gamma^{2}(\log2)^{3}}}+C\e^{-D^{8/7}}\sum_{t=0}^{t_{*}-1}\e^{t\log4-A^{8/7}(t\wedge t_{1})^{8/7}/C}+Cq_{t_{*}}^{-\nu}S^{(2+\gamma^{2}/2)\nu}2^{-t_{*}\alpha}.\label{eq:pluggedintomaster}
\end{align}

\emph{Step 3: analyzing the right-hand side of \eqref{pluggedintomaster}.
}Now we fix $A=c(1-\eta)/2$ and analyze the right-hand side of \eqref{pluggedintomaster}.
To bound the third term, we take
\begin{equation}
t_{*}=\left\lceil \frac{\nu}{\alpha}(2+\gamma^{2}/2)\log_{2}S\right\rceil ,\label{eq:tstardef}
\end{equation}
so $S^{(2+\gamma^{2}/2)\nu}2^{-t_{*}\alpha}\le1$. We note that the
first term of \eqref{pluggedintomaster} can be bounded above by 
\[
C\sum_{t=0}^{\infty}\e^{-\frac{B^{2}t^{1/3}}{2\gamma^{2}(\log2)^{3}}},
\]
which goes to $0$ as $B\to\infty$. The second term of \eqref{pluggedintomaster}
can be bounded by writing 
\begin{align}
\sum_{t=0}^{t_{*}-1}\e^{t\log4-A^{8/7}(t\wedge t_{1})^{8/7}/C} & \le\sum_{t=0}^{\infty}\e^{t\log4-A^{8/7}t^{8/7}/C}+\e^{-A^{8/7}t_{1}^{8/7}/C}\sum_{t=t_{1}}^{t_{*}-1}4^{t}\nonumber \\
 & \le C\left(1+\exp\left\{ -A^{8/7}\left((1-\eta)^{-1}\log_{2}\frac{S}{T_{0}}-1\right)^{8/7}/C+\left(\frac{\nu}{\alpha}(2+\gamma^{2}/2)\log_{2}S\right)\log4\right\} \right),\label{eq:oursum}
\end{align}
using the values \eqref{t1def} and \eqref{tstardef} of $t_{1}$
and $t_{*}$ fixed above. The right-hand side of \eqref{oursum} is
also bounded as $S\to\infty$. Therefore, we have
\[
\lim_{\substack{B\to\infty\\
q_{t_{*}}\to\infty\\
\theta\to\infty
}
}\mathbf{P}\left(\bigcup_{t_{0}=0}^{t_{*}-1}\left\{ \max_{i,j=0}^{2^{t_{0}}-1}\sup_{x,y\in\mathbb{B}_{t_{0};i,j}}d_{\mathbb{R}}(x,y)\ge q_{t_{*}}+\theta\Theta^{*}(S)\sum_{t=t_{0}}^{t_{*}-1}2^{Bt^{2/3}-\frac{1}{2}c(1-\eta)(t\wedge t_{1})}\right\} \right)=0,
\]
uniformly in $S$. In particular, we can take $\theta=q_{t_{*}}=B$,
and use \lemref{limtoinfty} to fold these parameters into the sum,
so the limit is simplified to%
\begin{comment}
, we can fold $q_{t_{*}}$ and the factor $B$ into the sum, so we
simplify the last limit to
\end{comment}
\begin{equation}
\lim_{B\to\infty}\mathbf{P}\left(\bigcup_{t_{0}=0}^{t_{*}-1}\left\{ \max_{i,j=0}^{2^{t_{0}}-1}\sup_{x,y\in\mathbb{B}_{t_{0};i,j}}d_{\mathbb{R}}(x,y)\ge\Theta^{*}(S)\sum_{t=t_{0}}^{t_{*}-1}2^{Bt^{2/3}-\frac{1}{2}c(1-\eta)(t\wedge t_{1})}\right\} \right)=0.\label{eq:maxprob}
\end{equation}

\emph{Step 4: analyzing the sum on the left-hand side of \eqref{maxprob}.}
The limit \eqref{maxprob} yields a bound on the crossing distances,
but we need to analyze the sum on the left-hand side to understand
exactly what kind of bound. We have, for a constant $C(B)$ depending
on $B$ but independent of $S$ and $t_{0}$, and a constant $c$
independent of $B,S,t_{0}$, that
\begin{align}
\sum_{t=t_{0}}^{t_{*}-1}2^{Bt^{2/3}-\frac{1}{2}c(1-\eta)(t\wedge t_{1})} & \le C(B)\e^{-c_{1}t_{0}}+C2^{-\frac{1}{2}c(1-\eta)t_{1}+(B+1)t_{*}^{2/3}}\nonumber \\
 & \le C(B)\e^{-c_{1}t_{0}}+C2^{-\frac{1}{2}c(1-\eta)((1-\eta)^{-1}\log_{2}(S/T_{0})-1)+(B+1)(\frac{\nu}{\alpha}(2+\gamma^{2}/2)\log_{2}S+1)^{2/3}},\label{eq:insidesum1}
\end{align}
where the second inequality is by plugging in the definitions \eqref{t1def}
and \eqref{tstardef}. Now we note that we can take $C(B)$ so that
\[
2^{-\frac{1}{2}c(1-\eta)((1-\eta)^{-1}\log_{2}(S/T_{0})-1)+(B+1)(\frac{\nu}{\alpha}(2+\gamma^{2}/2)\log_{2}S+1)^{2/3}}\le C(B)S^{-c_{2}},
\]
where $c_{2}>0$ is a constant which does not depend on $B$. Also,
as long as $t_{0}\le\frac{c_{2}}{c_{1}}\log S$, we have $S^{-c_{2}}\le\e^{-c_{1}t_{0}}$,
so in this case by \eqref{insidesum1} we have
\[
\sum_{t=t_{0}}^{t_{*}-1}2^{Bt^{2/3}-\frac{1}{2}c(1-\eta)(t\wedge t_{1})}\le C(B)\e^{-c_{1}t_{0}}.
\]
Plugging this into \eqref{maxprob}, we have
\[
\lim_{B\to\infty}\mathbf{P}\left(\bigcup_{t_{0}=0}^{\lfloor\frac{c_{2}}{t_{1}}\log S\rfloor\wedge(t_{*}-1)}\left\{ \max_{i,j=0}^{2^{t_{0}}-1}\sup_{x,y\in\mathbb{B}_{t_{0};i,j}}d_{\mathbb{R}}(x,y)\ge\Theta^{*}(S)C(B)\e^{-c_{1}t_{0}}\right\} \right)=0,
\]
which is equivalent to \eqref{probto0}.
\end{proof}
Finally, we are ready to prove \thmref{Holder}.
\begin{proof}[Proof of \thmref{Holder}.]
By \propref{chaining}, for each $\eps>0$ we have a constant $C=C(\eps)$
so that, with probability $1-\eps$, for each $0\le t_{0}\le c\log_{2}S$,
\[
\sup_{\substack{x,y\in\mathbb{B}(S)\\
|x-y|\le2^{-t_{0}}S
}
}\frac{d_{\mathbb{B}(S)}(x,y)}{\Theta^{*}(S)}\le2C(\eps)2^{-ct_{0}}.
\]
Therefore,
\[
\sup_{x,y\in\mathbb{B}(S)}\frac{d_{\mathbb{R}}(x,y)/\Theta^{*}(S)}{|\frac{x-y}{S}|\vee S^{-c}}\le2C(\eps)
\]
with probability $1-\eps$, so the theorem is proved.
\end{proof}

\section{Bi-Hölder continuity of the limiting metrics\label{sec:invholder}}

Up to now, we have established the existence of a subsequential limiting
function for our sequence of approximating metrics. It is of course
clear that any limiting function must satisfy positivity, symmetry,
and the triangle inequality, qualifying it for the title of ``pseudo-metric.''
However, it is not \emph{a priori} clear that a limiting pseudometric
must be a true metric---that is, that it is positive definite. That
is what we will prove in this section. In fact, we will prove the
stronger statement that the Euclidean metric must be Hölder continuous
with respect to any limiting LGD metric. The following proposition
gives a uniform upper bound on the moments of the Hölder coefficient,
from which the Hölder continuity of the Euclidean metric with respect
to any subsequential limit of the LGD metric follows by Fatou's lemma.
\begin{prop}
For every $B\in(1,\infty)$, we have an $\alpha_{0}>0$ so that for
any $\alpha\in(0,\alpha_{0})$, there is a constant $C<\infty$ so
that for any $\delta>0$, if we define $\mathbb{B}=\mathbb{B}(1)$,
then
\begin{equation}
\mathbf{E}\left[\left(\max_{x,y\in\mathbb{B}}\frac{|x-y|}{\left[d_{\mathbb{B},\delta}(x,y)/\Theta^{*}(\delta^{-\frac{2}{\gamma^{2}+4}})\right]^{\alpha}}\right)^{B}\right]\le C.\label{eq:invholdermomentbound}
\end{equation}
\end{prop}

\begin{proof}
Note that
\begin{equation}
\begin{split}\max_{x,y\in\mathbb{B}}\frac{|x-y|}{\left[d_{\mathbb{B},\delta}(x,y)/\Theta^{*}(\delta^{-\frac{2}{\gamma^{2}+4}})\right]^{\alpha}} & =\max_{n\in\mathbf{N}}\max_{\substack{x,y\in\mathbb{B}\\
|x-y|\ge2^{-n}
}
}\frac{|x-y|}{\left[d_{\mathbb{B},\delta}(x,y)/\Theta^{*}(\delta^{-\frac{2}{\gamma^{2}+4}})\right]^{\alpha}}\\
 & \le2\max_{n\in\mathbf{N}}\left(2^{-n}\left[\frac{d_{\mathbb{B},\delta}(\min;2^{-n})}{\Theta^{*}(\delta^{-\frac{2}{\gamma^{2}+4}})}\right]^{-\alpha}\right).
\end{split}
\label{eq:divide}
\end{equation}
By \propref{minmomentbound} and \thmref{variancebound}, we have
for every $A<\infty$ a constant $C<\infty$ so that, for each $a\in(0,1)$,
we have
\[
\mathbf{E}[d_{\mathbb{B}}(\min;a)^{-A}]\le C(1+a^{-C})[\Theta^{*}(\delta^{-\frac{2}{\gamma^{2}+4}})]^{-A}.
\]
This implies that
\begin{equation}
\mathbf{E}\left[\max_{n\in\mathbf{N}}\left(\frac{d_{\mathbb{B}}(\min;2^{-n})/\Theta^{*}(\delta^{-\frac{2}{\gamma^{2}+4}})}{2^{-n/\alpha}}\right)^{-A}\right]\le\sum_{n=0}^{\infty}\mathbf{E}\left[\left(\frac{d_{\mathbb{B}}(\min;2^{-n})/\Theta^{*}(\delta^{-\frac{2}{\gamma^{2}+4}})}{2^{-n/\alpha}}\right)^{-A}\right]\le C\sum_{n=0}^{\infty}\frac{(1+2^{Cn})}{2^{An/\alpha}},\label{eq:summable}
\end{equation}
where the constant $C$ still depends on $A$. As long as $\alpha$
is chosen small enough so that $A/\alpha\ge C$, the last sum is finite.
In particular, choose $\alpha$ so small that $A/\alpha\ge B$. Then
we have
\begin{align*}
\mathbf{E}\left[\max_{n\in\mathbf{N}}\left(\frac{d_{\mathbb{B}}(\min;2^{-n})/\Theta^{*}(\delta^{-\frac{2}{\gamma^{2}+4}})}{\left[d_{\mathbb{B},\delta}(x,y)/\Theta^{*}(\delta^{-\frac{2}{\gamma^{2}+4}})\right]^{\alpha}}\right)^{A/\alpha}\right] & \le\mathbf{E}\left[\left(2\max_{n\in\mathbf{N}}2^{-n}\left[\frac{d_{\mathbb{B},\delta}(\min;2^{-n})}{\Theta^{*}(\delta^{-\frac{2}{\gamma^{2}+4}})}\right]^{-\alpha}\right)^{A/\alpha}\right]\\
 & =2^{A/\alpha}\mathbf{E}\left[\max_{n\in\mathbf{N}}\left(\frac{d_{\mathbb{B},\delta}(\min;2^{-n})/\Theta^{*}(\delta^{-\frac{2}{\gamma^{2}+4}})}{2^{-n/\alpha}}\right)^{-A}\right]\le C,
\end{align*}
where the first inequality is by \eqref{divide} and the second is
by \eqref{summable}. Then \eqref{invholdermomentbound} follows by
Jensen's inequality.
\end{proof}

\appendix

\section{Technical lemmas\label{sec:appendix-techlemmas}}
\begin{lem}
\label{lem:relatequantiles}Let $X$ be a positive random variable
and let $\Theta_{X}$ be its quantile function. For any $p,q\in(0,1)$,
we have
\[
\Theta_{X}(q)\le\exp\left\{ \sqrt{\Var(\log X)}\left((1-q)^{-1/2}+p^{-1/2})\right)\right\} \Theta_{X}(p).
\]
\end{lem}

\begin{proof}
Define $\mu=\mathbf{E}\log X$. We have, for $a\ge\e^{\mu}$,
\begin{align*}
1-F_{X}(a)=\mathbf{P}(X\ge a) & =\mathbf{P}(\log X\ge\log a)=\mathbf{P}(\log X-\mu\ge\log a-\mu)\le\frac{\Var(\log X)}{\left(\log a-\mu)\right)^{2}}.
\end{align*}
This means that %
\begin{comment}
$F_{X}(a)\ge1-\frac{\Var(\log X)}{\left(\log a-\mathbf{E}\log X)\right)^{2}},$
so, inverting, we have
\end{comment}
{} $\Theta_{X}\left(1-\frac{\Var(\log X)}{\left(\log a-\mu\right)^{2}}\right)\le a.$
Now if $a=\exp\left\{ \sqrt{\frac{\Var(\log X)}{1-p}}+\mu\right\} $,
then $a\ge\exp\{\mathbf{E}\log X\}$ and $1-\frac{\Var(\log X)}{\left(\log a-\mu\right)^{2}}=p,$
so 
\[
\Theta_{X}(p)\le\exp\left\{ \sqrt{\frac{\Var(\log X)}{1-p}}+\mu\right\} .
\]

Similarly, we have, if $a\le\exp\{\mathbf{E}\log X\}$,
\begin{align*}
F_{X}(a)=\mathbf{P}(X\le a) & =\mathbf{P}(\log X\le\log a)\le\mathbf{P}(\log X-\mu\le\log a-\mu)\le\frac{\Var(\log X)}{\left(\log a-\mu\right)^{2}}.
\end{align*}
Therefore, $a\le\Theta_{X}\left(\frac{\Var(\log X)}{\left(\log a-\mu\right)^{2}}\right).$
Now if $a=\exp\left\{ -\sqrt{\frac{\Var(\log X)}{p}}+\mu\right\} $,
then $a\le\exp\{\mathbf{E}\log X\}$ and $\frac{\Var(\log X)}{\left(\log a-\mu\right)^{2}}=p,$
so
\[
\Theta_{X}(p)\ge\exp\left\{ -\sqrt{\frac{\Var(\log X)}{p}}+\mu\right\} .
\]
Therefore, for any $p,q\in(0,1)$, we have
\begin{align*}
\Theta_{X}(q) & \le\exp\left\{ \sqrt{\frac{\Var(\log X)}{1-q}}+\mu\right\} \le\exp\left\{ \sqrt{\Var(\log X)}\left((1-q)^{-1/2}+p^{-1/2})\right)\right\} \exp\left\{ -\sqrt{\frac{\Var(\log X)}{p}}+\mu\right\} \\
 & \le\exp\left\{ \sqrt{\Var(\log X)}\left((1-q)^{-1/2}+p^{-1/2})\right)\right\} \Theta_{X}(p).\qedhere
\end{align*}
\end{proof}
\begin{lem}
\label{lem:hugedeviations}Suppose that $p<\frac{1}{2}$ and $X_{1},\ldots,X_{N}$
are iid $\Bernoulli(p)$ random variables. Let $S_{N}=\sum_{i=1}^{N}X_{i}$.
Then 
\[
\mathbf{P}\left[S_{N}/N\ge1/2\right]\le(8p)^{N/2}.
\]
\end{lem}

\begin{proof}
We have 
\[
\mathbf{P}\left[S_{N}/N\ge1/2\right]=\mathbf{P}\left[\e^{\lambda S_{N}}\ge\e^{\lambda N/2}\right]\le\frac{\left(\mathbf{E}\e^{\lambda X_{i}}\right)^{N}}{\e^{\lambda N/2}}=\left(\frac{p\e^{\lambda}+1-p}{\e^{\lambda/2}}\right)^{N}.
\]
Putting $\lambda=\log\frac{1-p}{p}$, we obtain
\[
\mathbf{P}\left[S_{N}/N\ge1/2\right]\le\left(2(1-p)\left(\frac{1-p}{p}\right)^{-\frac{1}{2}}\right)^{N}\le(8p)^{N/2}.\qedhere
\]
\end{proof}

\begin{lem}
\label{lem:gradientbound}Suppose that $\mathbb{B}$ is a rectangle
and $\dist_{\Euc}(x,\partial\mathbb{B})\ge c\diam_{\Euc}(\mathbb{B})$
for some constant $c$. Then there is a constant $c$, depending on
$C$, so that
\[
|\nabla_{y}p_{t}^{\mathbb{B}}(x,y)|\le C\frac{|x-y|}{t^{2}}\e^{-\frac{|x-y|^{2}}{2t}}.
\]
\end{lem}

\begin{proof}
This proof essentially appeared before in \cite[Appendix 1]{RV14}
and was refined in \cite{DZZ18}, but as we need a somewhat different
statement, we present the proof in full. We have
\[
p_{t}^{\mathbb{B}}(x,y)=\frac{\e^{-\frac{|x-y|^{2}}{2t}}}{2\pi t}q_{t}^{\mathbb{B}}(x,y),
\]
where
\[
q_{t}^{\mathbb{B}}(x,y)=\mathbf{P}\left(B_{s}-\frac{s}{t}B_{t}+x+\frac{s}{t}(y-x)\in\mathbb{B}\text{ for all }s\le t\right).
\]
Therefore, we have
\begin{equation}
\nabla_{y}p_{t}^{\mathbb{B}}(x,y)=-(x-y)\frac{\e^{-\frac{|x-y|^{2}}{2t}}}{4\pi t^{2}}q_{t}^{\mathbb{B}}(x,y)+\frac{\e^{-\frac{|x-y|^{2}}{2t}}}{2\pi t}\nabla_{y}q_{t}^{\mathbb{B}}(x,y).\label{eq:productrule}
\end{equation}
Fix $y_{1},y_{2}\in\mathbb{R}$. Let $z^{(1)}$ be the $x$-coordinate
of the right edge of $\mathbb{B}$. Let $E_{i}$ be the event 
\[
\left\{ \max_{s\le t}(B_{s}-\frac{s}{t}B_{t}+x+\frac{s}{t}(y_{i}-x))^{(1)}>z^{(1)}\right\} ,
\]
where the superscript $(1)$ means to consider the $x$-coordinate.
We note that

\[
\left|(B_{s}-\frac{s}{t}B_{t}+x+\frac{s}{t}(y_{1}-x))^{(1)}-(B_{s}-\frac{s}{t}B_{t}+x+\frac{s}{t}(y_{2}-x))^{(1)}\right|\le|y_{1}-y_{2}|.
\]
Therefore, if $E_{1}$ occurs but $E_{2}$ does not, then we must
have that 
\[
\max_{s\le t}(B_{s}-\frac{s}{t}B_{t}+x+\frac{s}{t}(y_{1}-x))^{(1)}\in[z^{(1)},z^{(1)}+|y_{1}-y_{2}|].
\]
We have that 
\begin{align*}
\mathbf{P} & \left(\max_{s\le t}(B_{s}-\frac{s}{t}B_{t}+x+\frac{s}{t}(y_{1}-x))^{(1)}\in[z^{(1)},z^{(1)}+(y_{1}-y_{2})^{(1)}]\right)\\
 & =\mathbf{P}\left(\max_{s\le t}(B_{s}-\frac{s}{t}B_{t}+x+\frac{s}{t}(y_{1}-x))^{(1)}\ge z^{(1)}\right)\\
 & \qquad-\mathbf{P}\left(\max_{s\le t}(B_{s}-\frac{s}{t}B_{t}+x+\frac{s}{t}(y_{1}-x))^{(1)}\ge z^{(1)}+(y_{1}-y_{2})^{(1)}\right).
\end{align*}
By the reflection principle, we have that
\begin{align*}
\mathbf{P}\left(\max_{s\le t}(B_{s}-\frac{s}{t}B_{t}+x+\frac{s}{t}(y_{1}-x))^{(1)}\ge z^{(1)}\right) & =\frac{1}{2}\e^{-\frac{(z^{(1)}+y_{1}^{(1)}-x)^{2}+(y_{1}^{(1)}-x)^{2}}{2t}}\\
\mathbf{P}\left(\max_{s\le t}(B_{s}-\frac{s}{t}B_{t}+x+\frac{s}{t}(y_{1}-x))^{(1)}\ge z^{(1)}+(y_{1}-y_{2})^{(1)}\right) & =\frac{1}{2}\e^{-\frac{(z^{(1)}+2y_{1}^{(1)}-y^{(2)}-x)^{2}+(y_{1}^{(1)}-x)^{2}}{2t}},
\end{align*}
so
\begin{align*}
\mathbf{P} & \left(E_{1}\setminus E_{2}\right)\le\frac{1}{2}\left(\e^{-\frac{(z^{(1)}+y_{1}^{(1)}-x)^{2}+(y_{1}^{(1)}-x)^{2}}{2t}}-\e^{-\frac{(z^{(1)}+2y_{1}^{(1)}-y^{(2)}-x)^{2}+(y_{1}^{(1)}-x)^{2}}{2t}}\right).
\end{align*}
Now we have
\begin{align*}
\lim_{y_{2}^{(1)}\to y_{1}^{(1)}} & \frac{1}{2(y_{2}^{(1)}-y_{1}^{(1)})}\left(\e^{-\frac{(z^{(1)}+y_{1}^{(1)}-x)^{2}+(y_{1}^{(1)}-x)^{2}}{2t}}-\e^{-\frac{(z^{(1)}+2y_{1}^{(1)}-y^{(2)}-x)^{2}+(y_{1}^{(1)}-x)^{2}}{2t}}\right)\\
 & =\frac{\left(z^{(1)}+y_{1}^{(1)}-x\right)}{2t}\e^{-\frac{(z^{(1)}+y_{1}^{(1)}-x)^{2}+(y_{1}^{(1)}-x)^{2}}{2t}}.
\end{align*}
This implies that $\left|\nabla_{y}q_{t}^{\mathbb{B}}(x,y_{1})\right|\le C\frac{\diam_{\Euc}(\mathbb{R})}{2t}\exp\left\{ -\frac{\diam_{\Euc}(\mathbb{R})^{2}}{Ct}\right\} $,
and plugging into \eqref{productrule} we obtain%
\begin{align*}
\left|\nabla_{y}p_{t}^{\mathbb{B}}(x,y)\right| & \le\frac{|x-y|}{4\pi t^{2}}\e^{-\frac{|x-y|^{2}}{2t}}+C\frac{\e^{-\frac{|x-y|^{2}}{2t}}}{4\pi t^{2}}\diam_{\Euc}(\mathbb{R})\exp\left\{ -\frac{\diam_{\Euc}(\mathbb{R})^{2}}{Ct}\right\} \le C\frac{|x-y|}{t^{2}}\e^{-\frac{|x-y|^{2}}{2t}}.\qedhere
\end{align*}
\end{proof}
\begin{lem}
\label{lem:AA}Let $\mathbb{X}$ be a rectangular subset of $\mathbf{R}^{d}$
and suppose that we have a family of functions $\{f_{n}:\mathbb{X}\to\mathbf{R}\}$
so that, for some constant $C$, we have
\begin{equation}
|f_{n}(x)|\le C\label{eq:bdd}
\end{equation}
and
\begin{equation}
|f_{n}(x)-f_{n}(y)|\le C\left(|x-y|^{\alpha}\vee\beta_{n}\right)\label{eq:almostholder}
\end{equation}
for some $\alpha>0$ and some sequence $\beta_{n}\to0$. Then the
family $\{f_{n}:\mathbb{X}\to\mathbf{R}\}$ is precompact in the uniform
topology.
\end{lem}

\begin{proof}
Extend $f_{n}$ to an open subset $\mathbb{X}'$ of $\mathbf{R}^{d}$
containing $\mathbb{X}$ by reflecting it across the boundaries; note
that this can be done in such a way that \eqref{almostholder} still
holds, perhaps with a larger constant $C$. Let $\rho:\mathbf{R}^{d}\to\mathbf{R}$
be a smooth positive function such that $\supp\rho\subset B(0,1)$
and define $\rho_{\eps}(x)=\eps^{-d}\rho(\eps^{-1}x).$ Let $*$ denote
convolution. We have that
\begin{equation}
\begin{aligned}|\rho_{\eps}*f_{n}(x)-f_{n}(x)| & =\left|\int\rho_{\eps}(y)f_{n}(x-y)\,\dif y-f_{n}(x)\right|\le\int\rho_{\eps}(y)|f_{n}(x-y)-f_{n}(x)|\,\dif y\\
 & \le C\int\rho_{\eps}(y)(|y|^{\alpha}\vee\beta_{n})\,\dif x\le C(\eps^{\alpha}\vee\beta_{n}).
\end{aligned}
\label{eq:rhoandf}
\end{equation}
Moreover, if $|x-z|^{\alpha}\ge\beta_{n}$, then we have
\begin{align}
\left|\rho_{\eps}*f_{n}(x)-\rho_{\eps}*f_{n}(z)\right| & =\left|\int[f_{n}(x-y)-f_{n}(z-y)]\rho_{\eps}(y)\,\dif y\right|\le C|x-z|^{\alpha},\label{eq:ctsfar}
\end{align}
while if $|x-z|^{\alpha}\le\beta_{n}$, then we have, if $\eps=\beta_{n}$,
that
\begin{equation}
\begin{aligned}\left|\rho_{\eps}*f_{n}(x)-\rho_{\eps}*f_{n}(z)\right| & =\int f_{n}(y)[\rho_{\eps}(x-y)-\rho_{\eps}(z-y)]\,\dif y=\int[f_{n}(y)-f_{n}(x)]\cdot[\rho_{\eps}(x-y)-\rho_{\eps}(z-y)]\,\dif y\\
 & \le C\beta_{n}\int\left|\rho_{\eps}(x-y)-\rho_{\eps}(z-y)\right|\,\dif y\le C\|\rho\|_{\mathcal{C}^{1}}\beta_{n}\eps^{-1}|x-z|=C|x-z|.
\end{aligned}
\label{eq:ctclose}
\end{equation}
Together, \eqref{ctsfar} and \eqref{ctclose} imply the family $\{\rho_{\beta_{n}}*f_{n}\}$
is equicontinuous; since the family is evidently bounded by \eqref{bdd},
the Arzelà--Ascoli theorem implies that there is a continuous function
$f$ and a subsequence $(n_{k})$ so that
\begin{equation}
\lim_{k\to\infty}\rho_{\beta_{n_{k}}}*f_{n_{k}}=f\label{eq:mollsconv}
\end{equation}
uniformly. On the other hand, \eqref{rhoandf} implies that
\begin{equation}
\lim_{k\to\infty}|\rho_{\beta_{n}}*f_{n}-f_{n}|=0\label{eq:mollsclose}
\end{equation}
uniformly. Combining \eqref{mollsconv} and \eqref{mollsclose} implies
that $\lim\limits _{k\to\infty}f_{n}=f$, which completes the proof.
\end{proof}
\bibliographystyle{habbrv}
\phantomsection\addcontentsline{toc}{section}{\refname}\bibliography{tightness-gg}

\end{document}